\documentclass[reqno,12pt]{amsart}
\usepackage{amssymb,amsmath,amsthm,}
    \usepackage{a4wide}
\usepackage{amscd}
\usepackage{amsfonts}
\usepackage{amssymb}
\usepackage{latexsym}
\usepackage{color}
\usepackage{esint}
\usepackage{graphicx}
\usepackage{float}
\graphicspath{{Figures/}}

\usepackage{graphicx}
\usepackage[breaklinks=true,bookmarks=false]{hyperref}
\usepackage{cleveref}

\usepackage[makeroom]{cancel}

\newtheorem{theorem}{Theorem}

\newtheorem{corollary}[theorem]{Corollary}
\newtheorem{definition}{Definition}
\newtheorem{lemma}{Lemma}
\newtheorem{proposition}[theorem]{Proposition}
\newtheorem{remark}{Remark}

\let\e=\varepsilon

\let\p=\partial

\let\O=\Omega

\let\o=\omega

\numberwithin{equation}{section}

\let\hide\iffalse

\newcommand{\R}{\mathbb{R}}

\newcommand{\be}{\begin{equation}}
\newcommand{\bm}{\begin{multline}}
\newcommand{\ee}{\end{equation}}
\newcommand{\dd}{\mathrm{d}}

\newcommand{\xb}{x_{\mathbf{b}}}
\newcommand{\tb}{t_{\mathbf{b}}}

\newcommand{\xf}{x_{\mathbf{f}}}
\newcommand{\tf}{t_{\mathbf{f}}}

\newcommand{\bxp}{\mathbf{x}_{p^1}^1}
\newcommand{\xpj}{\mathbf{x}_{p^1,j}^1}

\newcommand{\Bes}{\begin{eqnarray*}}
\newcommand{\Ees}{\end{eqnarray*}}
\newcommand{\Be}{\begin{equation} }
\newcommand{\Ee}{\end{equation}}

\def\p{\partial}

\def\O{\Omega}
\def\R{\mathbb{R}}

\def\B{\begin{equation}}
\def\E{\end{equation}}
\def\BN{\begin{eqnarray*}}
\def\EN{\end{eqnarray*}}

\begin{document}
\title{Gradient Decay in the Boltzmann theory of Non-isothermal boundary}

 \author{Hongxu Chen}
\address{Department of Mathematics, The Chinese University of Hong Kong, Shatin, N.T., Hong Kong, email:  hongxuchen.math@gmail.com}
 \author{Chanwoo Kim}
 \address{Department of Mathematics, University of Wisconsin-Madison, Madison, WI, 53706, USA, email: chanwookim.math@gmail.com}
 
\date{\today}

\maketitle

\begin{abstract}
We consider the Boltzmann equation in convex domain with non-isothermal boundary of diffuse reflection. For both unsteady/steady problems, we construct solutions belong to $W^{1,p}_x$ for any $p<3$. We prove that the unsteady solution converges to the steady solution in the same Sobolev space exponentially fast as $t \rightarrow \infty$.

\end{abstract}

\section{Introduction}
One of the central questions in kinetic theory has been finding effective control of derivatives of solutions to the Boltzmann equation
\Be
 \partial_t F  + v\cdot \nabla_x F   = Q(F,F)  \ \ \text{in} \ \ \R_+ \times \O \times \R^3 \label{F}
 \Ee
in the presence of a physical boundary $\p\O$. Here we consider the hard sphere collision operator, where $Q(F,F)$ takes the form:
\begin{equation}\label{collision operator}
\begin{split}
   Q(F_1,F_2) & := Q_{\text{gain}}(F_1,F_2)-Q_{\text{loss}}(F_1,F_2) \\
     & =\int_{\mathbb{R}^3}\int_{\mathbb{S}^2}|(v-u) \cdot \o|
     \big[F_1(u')F_2(v')-F_1(u)F_2(v) \big]\dd \omega \dd u,
\end{split}
\end{equation}
with $u'=u+[(v-u)\cdot \omega]\omega$, $v'=v-[(v-u)\cdot \omega]\omega$ with $\omega \in \mathbb{S}^2$.

A particular motivation for this question could be justifying one of the major a priori assumptions of the celebrated theory of Devillettes-Villani \cite{DV}, namely that a \textit{uniform in time and space smoothness} of Boltzmann solutions. While there has been much progress on the regularity estimate, such a desired uniform estimate is far less satisfactory. In this paper, we focus on the diffuse reflection boundary condition 
\Be
 F(\cdot, x,v)|_{n(x)\cdot v<0}    = M_W(x,v)\int_{n(x)\cdot u>0} F(\cdot,x,u) \{n(x)\cdot u\} \dd u ,\ \ x\in \p \Omega, \label{bc_F}
\Ee
where the outward normal at the boundary $\p\O$ is denoted by $n(x)$. This condition implies an instantaneous equilibration of bounce-back particles to a thermodynamic equilibrium of the wall temperature $T_W(x)$ at the boundary point $x\in \partial\Omega$:
\begin{equation}\label{Wall Maxwellian}
M_W(x,v) 
= M_{1, 0, T_W (x)}(v)
:=\frac{1}{2\pi [T_W(x)]^2}e^{-\frac{|v|^2}{2T_W(x)}}.
\end{equation} 
When the wall temperature is non-constant (non-isothermal boundary) any steady solutions, if exist, to the Boltzmann equation with the diffuse reflection boundary condition \eqref{F_steady}-\eqref{bc_F_steady} is not a local Maxwellian and hence not a (local) minimizer of entropy:
\begin{align}
     v\cdot \nabla_x F_s &= Q(F_s,F_s), \label{F_steady}\\
     F_s(x,v)|_{n(x)\cdot v<0} &= M_W(x,v)
    \int_{n(x) \cdot u>0} F_s(x,u)\{n(x)\cdot u\} \dd u, \ \ x\in \p \O \label{bc_F_steady}.
\end{align}
Esposito-Guo-Kim-Marra established the construction of steady solutions and their dynamical stability in $L^\infty$ when the data are small in \cite{EGKM}. On the other hand, the diffuse reflection boundary condition solely stabilizes the dynamics even without the intermolecular collision \cite{JK, JK2}.


So far any available results of regularity to the Boltzmann equation with the diffuse reflection boundary condition stop before a $H^1$ threshold, and hence the following question is open:
\Be
\nabla_x F(t,x,v) \in L^2 (\O \times \R^3)  .
\Ee
In \cite{GKTT}, the second author and collaborators prove that dynamic solutions belong to $W^{1,p}_x$ for $p<2$, while such control grows exponentially in time. In \cite{GKTT, CK}, they use a kinetic distance function and prove a \textit{weighted} estimate beyond $H^1_x$. All these studies left the $H^1$ threshold problem unsolved (also see \cite{Ikun}).

In this paper, we pass over this $H^1$ threshold for the first time, which would\textbf{ mark a resolution of a major open problem in a boundary regularity theory of kinetic theorem}. Moreover we prove that both steady/unsteady solutions belong to $W^{1,p}_x$ for $p<3$. Furthermore, we prove that such $W^{1,p}_x$-control of unsteady solutions decays exponentially in time. Therefore the result would make\textbf{ a surprising breakthrough beyond a standing open problem!}

 \begin{theorem}[Informal statement of the Main Theorem]\label{thm:informal_W1p}
Assume the domain is strictly convex (see \eqref{convex}) and {\color{black}the boundary is $C^3$}. Suppose $\sup_{x \in \p\O}|T_W(x)-T_0|\ll  1$ for some constant $T_0>0$ and $T_W(x)  \in C^1(\p\O)$. Then any steady solution $F_s (x,v)$ of a finite total mass satisfies 
\Be
 \frac{1}{\sqrt{M_{1,0,T_0} (v)}} |\nabla_x F_s(x,v) |  \in L^p (\O \times \R^3) \ \ \text{for all } \ p<3 .
\Ee
Here, a global Maxwellian of constant temperature $T_0$ is denoted by $M_{1,0,T_0} = \frac{1}{2\pi [T_0]^2}e^{-\frac{|v|^2}{2T_0}}$.

Suppose an initial datum 
$$F_0 (x,v) = F_s (x,v) + \sqrt{M_{1, 0, T_0} (v) } f_0(x,v)$$ has the same total mass as the steady solution $F_s (x,v)$, 
and $\nabla_x f_0(x,v) \rightarrow 0$ pointwisely, fast enough as $|v| \rightarrow \infty$. Finally we assume a smallness condition $\| e^{\theta|v|^2} f_0  \|_\infty\ll 1 $ for some $\theta>0$. Then for any $p<3$,
\Be
 \nabla_x F(t,x,v)\rightarrow \nabla_x F_s(x,v)    \ \text{in} \ L^p (\O \times \R^3) \ \ \text{exponentially fast as} \ t\rightarrow \infty. 
\Ee

 


\end{theorem}


Our result, apart from its proof, has two implications. First, it shows that the generic solutions of the steady/unsteady Boltzmann equations are regular globally as $W^{1,3-}$. Second, it shows that the gradient of solutions are asymptotically stable, which demonstrates the versatility of the method of \cite{EGKM} and \cite{GKTT}.

\textbf{Notations. }Throughout this paper we will use the following notations 
\begin{align}
 f\lesssim  g \Leftrightarrow \text{ there exists } 0<C<\infty \text{ such that } 0 \leq f\leq Cg;  \label{lesssim} \\
 f\thicksim g   \Leftrightarrow \text{ there exists } 0<C<\infty \text{ such that } 0 \leq \frac{f}{C}\leq g\leq Cf;  \label{thicksim}  \\
 f \ll g\Leftrightarrow \text{ there exists a small constant } c>0 \text{ such that } 0 \leq  f \leq  cg;  
 \label{ll}\\
  \Vert f\Vert_\infty = \Vert f\Vert_{L^\infty(\Omega\times \mathbb{R}^3)}, \  |f|_\infty = \Vert f\Vert_{L^\infty(\p\Omega\times \mathbb{R}^3)}, \ \Vert f\Vert_p = \Vert f\Vert_{L^p(\Omega\times \mathbb{R}^3)}.\label{norm} 
\end{align}
Moreover, $f = o(1) \Leftrightarrow  |f|\ll 1$ and $\langle v\rangle = \sqrt{1+|v|^2}$.

\ \\

\section{Results}

\subsection{Basic setting}

In this paper we assume the domain is a bounded open subset of $\mathbb{R}^3$ defined as $\O = \{x \in \R^3: \xi(x) <0\}$ and the boundary is defined as $\p\O=\{x\in \mathbb{R}^3: \xi(x)=0\}$ via a $C^3$ function $\xi : \R^3 \rightarrow \R$. Equivalently we assume that for all $q\in \partial \Omega$, there exists a $C^3$ function ${\color{black} \eta_q}$ and $0<\delta_1 \ll 1$, such that
\Be\label{O_p}
\eta_q:
B_+(0; \delta_1)
\ni \mathbf{x}_q := (\mathbf{x}_{q,1},\mathbf{x}_{q,2},\mathbf{x}_{q,3})
 \rightarrow
{\color{black} \mathbb{R}^3},
\Ee
{\color{black} where the map is one-to-one and onto to the image $\mathcal{O}_q := \eta_q (B_+ (0;\delta_1))$ when $\delta_1$ is sufficiently small. Moreover, $\eta_q (\mathbf{x}_q) \in \p\O$ if and only if $\mathbf{x}_{q,3}=0$ within the image of $\eta_q$.} We refer to \cite{EGKM2} for the construction of such $\xi$ and $\eta_q$. We further assume that the domain is strictly convex in the following sense:
 \Be\label{convex}
\sum_{i,j=1}^3\zeta_i \zeta_j\p_i\p_j \xi(x) \gtrsim |\zeta|^2  \  \text{ for   all }   x \in \bar{\O}  \text{ and }  \zeta \in \R^3.
 \Ee
Without loss of generality, we may assume that $\nabla \xi \neq 0$ near $\p\O$. Then the outward normal can be defined as 
\begin{equation}\label{normal}
n(x):=\frac{\nabla \xi(x)}{|\nabla \xi(x)|}.
\end{equation}

The boundary of the phase space is
\[\gamma:=\{(x,v)\in \p\O\times \mathbb{R}^3\}.\]
We decompose $\gamma$ as
\begin{align}
   \gamma_+ &:=  \{(x,v)\in \p\O\times \mathbb{R}^3: n(x)\cdot v>0\}, \text{ (the outgoing set) }\notag\\
    \gamma_- &:= \{(x,v)\in \p\O\times \mathbb{R}^3: n(x)\cdot v<0\}, \text{ (the incoming set)} \notag\\
    \gamma_0 & := \{(x,v)\in \p\O\times \mathbb{R}^3: n(x)\cdot v=0\}. \text{ (the grazing set)}\notag
\end{align}

\ \\

\subsection{Main result}

In this paper, we investigate the regularity of the dynamical problem \eqref{F}. We also analyze the rate at which the gradient of the solution to the dynamical problem \eqref{F} approaches the gradient of the solution to the steady problem \eqref{F_steady}. Without loss of generality,  we assume the wall temperature has a small fluctuation around $1$($T_0=1$ in Theorem \ref{thm:informal_W1p}). We denote $\mu$ as the global Maxwellian:
\begin{equation}\label{mu}
\mu:= \frac{1}{2\pi}e^{-\frac{|v|^2}{2}} = \sqrt{M_{1,0,1}(v)}.
\end{equation}
Then we consider the steady problem~\eqref{F_steady} as a perturbation around the Maxwellian:
\Be\label{linearization_steady}
F_s(x,v)=\mu(v)+ \sqrt{\mu(v)} f_s(x,v)  ,
\Ee
and consider the dynamical problem~\eqref{F} as a perturbation around the steady problem~\eqref{F_steady}
\begin{equation}\label{f_d}
F(t,x,v) = F_s(x,v) + \sqrt{\mu(v)} f(t,x,v).
\end{equation}
The equation for $f$ in~\eqref{f_d} reads
\begin{equation}\label{f_equation}
\p_t f +v\cdot \nabla_x f +\nu(v) f = K(f)+\Gamma(f_s,f)+\Gamma(f,f_s)+\Gamma(f,f),
\end{equation}
where we denote the linear Boltzmann operator as
\begin{equation}\label{Gamma}
 Lf=\nu(v) f - K(f)=-\frac{Q(\mu,\sqrt{\mu}f)}{\sqrt{\mu}}-\frac{Q(\sqrt{\mu}f,\mu)}{\sqrt{\mu}},
\end{equation}
and the nonlinear Boltzmann operator as
\begin{equation}\label{Gamma_f}
\Gamma(f,g) = Q(\sqrt{\mu}f,\sqrt{\mu}g)/\sqrt{\mu}.    
\end{equation}
The properties of these Boltzmann operators can be found in Lemma \ref{lemma:k_gamma}. Here $f_s$ is defined in~\eqref{linearization_steady}. 

The initial condition is given as 
\begin{equation}\label{initial}
f_0(x,v)=f(0,x,v)=(F(0,x,v)-F_s(x,v))/\sqrt{\mu}.    
\end{equation}
On the boundary $\gamma_-$, $F$ and $F_s$ are given by~\eqref{bc_F} and~\eqref{bc_F_steady}, hence the boundary condition of $f$ is given by 
\begin{equation}\label{f_bc}
f(t,x,v)|_{n(x)\cdot v<0} = \frac{M_W(x,v)}{\sqrt{\mu(v)}}\int_{n(x)\cdot u>0} f(t,x,u)\sqrt{\mu(u)}\{n(x)\cdot u\} \dd u.
\end{equation}

From the equation of $f$ in~\eqref{f_equation}, we define the initial condition for $\p_t f(0,x,v)$ as follows:
\begin{equation}\label{partial_t_f0}
\p_t f(0,x,v) : = -v\cdot \nabla_x f(0,x,v)-L(f(0))+\Gamma(f_s,f(0))+\Gamma(f(0),f_s)+\Gamma(f(0),f(0)).
\end{equation}

It is well-known that the Boltzmann equation possesses a generic singularity at the boundary. We adopt the following weight of \cite{GKTT}:
 \begin{definition} We define a \text{kinetic distance}:
\Be\label{kinetic_distance}
\begin{split}
\alpha(x,v)   : = \chi( \tilde{\alpha}(x,v)  ) ,  \ \
\tilde{\alpha}(x,v)   : = \sqrt{ |v \cdot \nabla_x \xi(x)|^2 - 2 \xi(x) (v \cdot \nabla_x^2 \xi(x) \cdot v)}, \ \ (x,v) \in \bar{\O} \times \R^3,
\end{split}
\Ee
where $\chi: [0,\infty) \rightarrow [0,\infty)$ stands for a non-decreasing smooth function such that
\Be\label{chi}
\chi (s) = s \ \text{for} \ s \in [0, 1/2 ],  \ \chi(s) = 1 \ \text{for} \ s \in [ 2, \infty ), \ and \  | \chi^\prime(s) | \leq 1 \ \text{for}  \  s \in [0,\infty).
\Ee
\end{definition}
\begin{remark}
We note that $\alpha\equiv 0$ on the grazing set $\gamma_0$. An important property of this kinetic weight is that $\alpha(x,v)$ is almost invariant along the characteristic, see Lemma \ref{lemma:velocity}.

Convexity is a necessary condition for studying the regularity of the Boltzmann equation. It is well-known that singularity propagates in non-convex domains~\cite{K}, and in such cases, we can only expect bounded variation regularity at best~\cite{GKTT2}. The convexity condition ensures the positivity of the kinetic weight in \eqref{kinetic_distance}. This weight has been instrumental in the development of the study of the regularity of the Boltzmann equation, leading to progress in recent years~\cite{CKL,CKQ,chen2}.
\end{remark}

In this paper, we first establish a weighted $C^1$ estimate for the dynamical problem that converges exponentially fast to the weighted $C^1$ estimate to the steady problem. 
\begin{theorem}\label{thm:dynamic_regularity}
Assume the domain is convex (\ref{convex}); {\color{black}the boundary is $C^3$;} $\sup_{x \in \p\O}|T_W(x)-T_0|\ll  1$ for some constant $T_0>0$; the initial condition~\eqref{initial} is continuous away from $\gamma_0$ and satisfy
\begin{equation}\label{int_initial_0}
\iint_{\O\times \mathbb{R}^3}f(0,x,v)\sqrt{\mu(v)} \dd x \dd v =0,    
\end{equation}
\begin{equation}\label{f0_stability}
\Vert wf(0)\Vert_\infty + |wf(0)|_\infty \ll 1, \ \ w(v):=e^{\theta|v|^2} \text{ for some } 0<\theta<1/4,
\end{equation}
and the compatibility condition
\begin{equation}\label{compatibility}
f(0,x,v)|_{n(x)\cdot v<0} = \frac{M_W(x,v)}{\sqrt{\mu(v)}} \int_{n(x)\cdot u>0}f(0,x,u)\sqrt{\mu(u)}\{n(x)\cdot u\} \dd u,
\end{equation}
then there exists a unique solution $f$ to the dynamical problem \eqref{f_d} and~\eqref{f_bc}. Moreover, $f$ is continuous away from $\gamma_0$, and for some $\lambda \ll 1$, it satisfies
\begin{equation}\label{dym_stability}
    e^{\lambda t}\Big[\| wf(t)\|_{\infty} +|wf(t)|_\infty\Big]\lesssim \Vert wf(0)\Vert_\infty + |wf(0)|_\infty.
\end{equation}

If we further assume $T_W(x)\in C^1(\p\O)$, the initial condition $\partial_t f(0,x,v)$ in \eqref{partial_t_f0} and $\alpha\nabla_x f(0,x,v)$ are continuous away from $\gamma_0$ and satisfy
\begin{equation}\label{f0_regularity}
\begin{split}
  \Vert w_{\tilde{\theta}}(v)\alpha(x,v)\nabla_x f(0)\Vert_\infty   &   <\infty , \ \ w_{\tilde{\theta}}(v) = e^{\tilde{\theta}|v|^2} \text{ for some }\tilde{\theta} \ll 1,  
\end{split}
\end{equation}
\begin{equation}\label{partial_f0_bdd}
 \Vert w \p_t f(0,x,v)\Vert_\infty  <\infty,   
\end{equation}
and satisfy the compatibility condition
\begin{equation}\label{compatibility_derivative}
\partial_t f(0,x,v)|_{\gamma-} = \frac{M_W(x,v)}{\sqrt{\mu(v)}}\int_{n(x)\cdot u>0} \p_t f(0,x,u)\sqrt{\mu(u)}\{n(x)\cdot u\} \dd u.
\end{equation}
Then for all $t>0$ and $\lambda \ll 1$ in~\eqref{dym_stability}, $f(t,x,v)\in C^1(\bar{\Omega}\times \mathbb{R}^3 \backslash \gamma_0)$ and satisfies:
\begin{equation}\label{weight_C1}
\begin{split}
 &e^{\lambda t}\Vert w_{\tilde{\theta}}(v)\alpha(x,v)\nabla_x f(t)\Vert_\infty   \\
 &\lesssim \Vert w_{\tilde{\theta}}(v)\alpha(x,v) \nabla_x f(0)\Vert_\infty+ \Vert w\p_t f(0)\Vert_\infty  + \sup_t e^{\lambda t}\{\Vert wf(t)\Vert_\infty + |wf(t)|_\infty\} .
\end{split}
\end{equation}
\end{theorem}

\begin{remark}
According to \cite{GKTT}, the continuity assumption on $\alpha\nabla_x f(0)$ and the compatibility assumption \eqref{compatibility_derivative} imply that $f$ belongs to $C^1$ away from the grazing set. The additional assumption~\eqref{partial_f0_bdd} is needed to control the contribution of the time derivative $\partial_t f$ in the characteristic formula~\eqref{nabla_x_1} - \eqref{nabla_x_8}. 

Here, we can observe that the additional assumption \eqref{partial_f0_bdd} is not actually a very restrictive one. In fact, it is very similar to the assumption \eqref{f0_regularity}. To see this, we first note that the last four terms on the right-hand side of \eqref{partial_t_f0} can be bounded in $L^\infty$. Now, on the boundary $\partial\Omega$, let $\tau_1, \tau_2$, and $n(x)$ be the basis vectors. We can decompose the first term as follows:
\begin{align*}
v\cdot \nabla_x f(0,x,v) &= (n(x)\cdot v) \partial_n f(0,x,v) + \sum_{i=1,2} (\tau_i \cdot v) \partial_{\tau_i} f(0,x,v).
\end{align*}
We can see that the assumption $(n(x)\cdot v) \partial_n f(0,x,v)$ is bounded in $L^\infty$ is almost the same as \eqref{f0_regularity}. This is because both assumptions cancel out the singularity, and we have $|n(x)\cdot v| \thicksim \alpha(x,v)$.
\end{remark}

\begin{remark}
In this paper, we enhance the weighted $C^1$ estimate of the dynamical problem by making it uniform in time. Additionally, we prove that the weighted $C^1$ norm of the dynamical problem converges to the weighted $C^1$ norm of the steady problem with an exponential rate.    
\end{remark}

The weighted $C^1$ estimate~\eqref{weight_C1} allows us to study the $W^{1,p}$ estimate without any weight. The following theorem is the formal version of Theorem \ref{thm:informal_W1p}, which proves a $W^{1,p}$ estimate for both the steady and dynamical problems.
\begin{theorem}\label{thm:W1p}
Assume the domain is convex (\ref{convex}); the boundary is $C^3$; $\sup_{x \in \p\O}|T_W(x)-T_0|\ll  1$ for some constant $T_0>0$; $T_W(x)\in C^1(\p\O)$, then there is a unique steady solution to~\eqref{linearization_steady} and~\eqref{bc_F_steady}, moreover, $f_s$ belongs to $W^{1,p}(\Omega\times \mathbb{R}^3)$ for $p<3$:
\begin{equation}\label{f_s_W1p}
\Vert \nabla_x f_s\Vert_p  \lesssim \Vert T_W-T_0\Vert_{C^1(\p\O)}.
\end{equation}

Assume all conditions in Theorem \ref{thm:dynamic_regularity} are satisfied so that we have the decay~\eqref{weight_C1}, then the dynamical problem $f(t)$ in~\eqref{f_d} further belongs to $W^{1,p}(\Omega\times \mathbb{R}^3)$ for $p<3$ and
\begin{equation}\label{f_W1p}
\begin{split}
    &e^{\lambda t}\Vert \nabla_x f(t)\Vert_p  \lesssim e^{\lambda t}\Vert w_{\tilde{\theta}}(v)\alpha(x,v)\nabla_x f(t)\Vert_\infty  <\infty.
\end{split}
\end{equation}

\end{theorem}

\begin{remark}
It is worth mentioning that the $W^{1,p}$ estimate we establish does not involve any weight. As shown in \eqref{f_W1p}, our result proves that the $W^{1,p}$ norm of the dynamical problem converges to the $W^{1,p}$ norm of the steady problem with exponential rate.
\end{remark}

\begin{remark}
 It is important to note that Theorem \ref{thm:W1p} only establishes an interior $W^{1,p}$ estimate without any weight for $p<3$. {\color{black}We don't expect the boundary regularity result up to $W^{1,p}$ for $p<3$. Indeed, Kim and his collaborators construct a counterexample of such boundary regularity in Lemma 12 of \cite{GKTT}, in which they prove $\int_0^1 \int_{\gamma_-}|\nabla_x f(s,x,v)|^2 \dd \gamma \dd s = +\infty$ for free transport equation.} 

However, the interior estimate fails for the second-order derivative. We refer detailed discussion to Appendix.   
\end{remark}

\textbf{Outline.} The structure of this paper is straightforward. In Section \ref{sec:prelim}, we present several lemmas that serve as preliminary results. In Section \ref{sec:C1}, we establish the weighted $C^1$ estimate and prove Theorem \ref{thm:dynamic_regularity}. Section \ref{sec:W1p} is dedicated to applying Theorem \ref{thm:dynamic_regularity} to obtain the $W^{1,p}$ estimate in Theorem \ref{thm:W1p}.

\ \\

\section{Preliminary}\label{sec:prelim}
In this section, we will begin by introducing and reparametrizing the backward exit position as stochastic cycles. Next, we will present several lemmas that describe the properties of stochastic cycles, Boltzmann operators, and kinetic weight.


\ \\

\subsection{Stochastic cycles and reparametrization.}
We mainly define the stochastic cycles and the reparametrization, then deduce another representation of the boundary condition in~\eqref{eqn: diffuse for f}. Denote $\tb(x,v),\xb(x,v)$ as the backward exit time and backward exit position. 
\Be\label{BET}
 \tb(x,v) :=   \sup \{ s>0:x-sv \in \O\},\quad {\color{black}\xb(x,v):= x-  \tb(x,v) v.}
\Ee

\begin{definition}\label{definition: sto cycle}
We define a stochastic cycles as $(x^0,v^0)= (x,v) \in \bar{\O} \times \R^3$ and inductively
\begin{align}
&x^1:= \xb(x,v), \   v^1 \in \{v^1\in \mathbb{R}^3:n(x^1)\cdot v^1>0\} , \label{vi}\\
&  v^{k}\in \{v^{k}\in \mathbb{R}^3:n(x^k)\cdot v^k>0\}, \ \ \text{for} \  k \geq 1,
\label{xi}
\\
 &x^{1} := \xb(x^k, v^k) , \ \tb^{k}:= \tb(x^k,v^k) \ \ \text{for} \  n(x^k) \cdot v^k\geq 0  . \label{tbi}
\end{align}
Choose $t\geq0.$ We define $t^0=t$ and
\Be\label{ti}
t^{k} =  t-  \{ \tb + \tb^1 + \cdots + \tb^{k-1}\},  \ \ \text{for} \  k \geq 1.
\Ee

\end{definition}

\begin{remark}
 Here $x^{k+1}$ depends on $(x,v,x^1,v^1,\cdots, x^k,v^k)$, while $v^k$ is a free parameter whose domain~\eqref{xi} only depends on $x^k$.
\end{remark}

 Recall \eqref{O_p}. Since the boundary is compact and $C^3$, for fixed $0<\delta_1 \ll 1$, we choose a finite number of $p \in \tilde{\mathcal{P}} \subset\p\O$ and $0<\delta_2\ll 1$ such that $\mathcal{O}_p=\eta_p(
B_+(0; \delta_1)) \subset B(p;\delta_2) \cap \bar{\O}$ and $\{\mathcal{O}_p \}$ forms a finite covering of $\partial \Omega$. We further choose an interior covering $\mathcal{O}_0 \subset \O$ such that $\{ \mathcal{O}_p\}_{p \in \mathcal{P}}$ with $\mathcal{P} = \tilde{\mathcal{P}}\cup \{0\}$ forms an open covering of $\bar\O$.
We define a partition of unity
\Be\label{iota}
\mathbf{1}_{\bar\O} (x)=
\sum_{p \in \mathcal{P}} \iota_p(x)
 \text{ such that }0 \leq \iota_p(x) \leq 1, \ \
  \iota_p(x) \equiv 0 \ \text{for} \ x  \notin \mathcal{O}_p.
 \Ee
 Without loss of generality (see \cite{KL}) we can always reparametrize $\eta_p$ such that $\partial_{\mathbf{x}_{p,i}} \eta_p \neq 0$ for $i=1,2,3$ at $\mathbf{x}_{p,3}=0$, and an \textit{orthogonality} holds as
\Be\label{orthogonal}
  \partial_{\mathbf{x}_{p,i}}\eta_p \cdot \partial_{\mathbf{x}_{p,j}}\eta_p =0 \ \ \text{at} \ \ \mathbf{x}_{p,3}=0 \text{  for  } i\neq j \text{ and } i,j\in \{1,2,3\}.
\Ee
And at $\mathbf{x}_{p,3}=0$, the $\mathbf{x}_{p,3}$ derivative gives the outward normal: 
\begin{equation}\label{normal: p_3}
{\color{black}n_p(\mathbf{x}_p) = \frac{\partial_{\mathbf{x}_{p,3}}\eta_p}{\langle \partial_{\mathbf{x}_{p,3}}\eta_p,\partial_{\mathbf{x}_{p,3}}\eta_p\rangle}.}    
\end{equation}
 For simplicity, we denote
\begin{equation}\label{partial_i eta}
 \partial_i \eta_p(\mathbf{x}_p): = \partial_{\mathbf{x}_{p,i}} \eta_p.
\end{equation}

\begin{definition} For $x \in \bar{\O}$, we choose $p \in\mathcal{P}$ as in (\ref{O_p}). We define
\begin{align}
    T_{\mathbf{x}_p}&
    =\left(
                               \begin{array}{ccc}
           \frac{\p_1 \eta_p(\mathbf{x}_p)}{\sqrt{g_{p,11}(\mathbf{x}_p) }}
           &      \frac{\p_2 \eta_p(\mathbf{x}_p)}{\sqrt{g_{p,22}(\mathbf{x}_p) }}
           &     \frac{\p_3 \eta_p(\mathbf{x}_p)}{\sqrt{g_{p,33}(\mathbf{x}_p) }}
            \\
                               \end{array}
                             \right)^t,\label{T}
\end{align}
with $g_{p,ij}(\mathbf{x}_p)    =\langle \partial_i \eta_p(\mathbf{x}_p),\partial_j \eta_p(\mathbf{x}_p)\rangle$ for $i,j\in \{1,2,3\}$.
Here $A^t$ stands as the transpose of a matrix $A$. Note that when $\mathbf{x}_{p,3}=0$, $T_{\mathbf{x}_p}       \frac{\p_i \eta_p(\mathbf{x}_p)}{\sqrt{g_{p,ii}(\mathbf{x}_p) }}
  = e_i$ for $i=1,2,3$ where $\{e_i\}$ is a standard basis of $\R^3$.

We define
\Be\label{bar_v}
\mathbf{v}_j(\mathbf{x}_p) = \frac{\p_j \eta_p(\mathbf{x}_p)}{\sqrt{g_{p,jj}(\mathbf{x}_p) }} \cdot  v.
\Ee
\end{definition}

We note that from (\ref{orthogonal}), the map $T_{\mathbf{x}_p}$ is an orthonormal matrix when $\mathbf{x}_{p,3}=0$. Therefore both maps $v \rightarrow \mathbf{v} (\mathbf{x}_p )$ and $\mathbf{v} (\mathbf{x}_p ) \rightarrow v$ have a unit Jacobian {\color{black} at $\mathbf{x}_{p,3}=0$}. This fact induces a new representation of boundary integration of diffuse boundary condition in~\eqref{f_bc}: For $x \in \p\O$ and $p \in\mathcal{P}$ as in (\ref{O_p}),
\begin{equation}
\begin{split}
\label{eqn: diffuse for f}
 \int_{n(x)\cdot v>0}f(t,x,v)\sqrt{\mu(v)}\{n(x)\cdot v\}\dd v
 = \int_{\mathbf{v} _{p ,3}>0}f(t, \eta_{p } (\mathbf{x}_{p }    ), T^t_{\mathbf{x} _{p }} \mathbf{v} ( \mathbf{x}_p) )\sqrt{\mu(\mathbf{v} (\mathbf{x}_p))}\mathbf{v} _{ 3}(\mathbf{x}_p) \dd\mathbf{v}  (\mathbf{x}_p).
\end{split}
\end{equation}
We have used the fact of $\mu(v )=\mu(|v |)=\mu(|T^t_{\mathbf{x}_{p } }\mathbf{v} (\mathbf{x}_p) |)=\mu(|\mathbf{v} (\mathbf{x}_p)  |)=\mu(\mathbf{v} (\mathbf{x}_p) ) $ and $\mathbf{x}_{p,3}=0$. 

 Now we reparametrize the stochastic cycle using the local chart defined in Definition \ref{definition: sto cycle}.

\begin{definition}\label{definition: chart}
Recall the stochastic cycles (\ref{xi}). For each cycle $x^k$ let us choose $p^k \in \mathcal{P}$ in (\ref{O_p}). Then we denote
\Be\begin{split}\label{xkvk}
\mathbf{x}^k_{p^k}&:= (\mathbf{x}^k_{p^k,1}, \mathbf{x}^k_{p^k,2},0)   \text{ such that }
\eta_{p^k} (\mathbf{x}^k_{p^k}) = x^k, \ \ \text{for} \  k \geq 1,
\\
\mathbf{v}^k_{p^k,i}&:=   \frac{\p_i \eta_{p^k}(\mathbf{x}_{p^k}^k)}{\sqrt{g_{p^k,ii}(\mathbf{x}_{p^k}^k) }} \cdot  v^k ,  \ \ \text{for} \  k \geq 1. 
\end{split}
\Ee
From \eqref{normal: p_3}, the outward normal at $x^k$ can also be denoted as
\begin{equation}\label{normal at xk}
    {\color{black}n(x^k)} {\color{black} = n_{p^k}(\mathbf{x}_{p^k}^k). }
\end{equation}

Finally, we define
 \Be
\p_{\mathbf{x}^{k}_{p^{k},i}}[
a( \eta_{p^{k}} ( \mathbf{x}_{p^{k} }^{k}  ),    {v}^{k} ) ]
 :=
\frac{\p \eta_{p^{k}}(\mathbf{x}^{k}_{p^{k},i})}{\p \mathbf{x}^{k}_{p^{k},i}}
\cdot \nabla_x a ( \eta_{p^{k}} ( \mathbf{x}_{p^{k} }^{k}  ), v^{k}) , \ \ i=1,2.
\label{fBD_x1}
\Ee

\end{definition}

\ \\

\subsection{Properties of stochastic cycles, Boltzmann operators, and kinetic weight.}

The derivative of $\xb(x,v),\tb(x,v)$ is given by the following lemma:
\begin{lemma}[Lemma 2.1 in \cite{CK}]\label{lemma:tb_xb_x}
The derivative of $\tb$ to $x_j$ and $v_j$ reads
\begin{equation}\label{xi deri tb}
\frac{\p \tb^{1}}{\p x_j^{1}} =
	\frac{1}{ \mathbf{v}^{2}_{p^{2},3}  }
	\frac{  \p_{3} \eta_{{p}^{2}}(\mathbf{x}_{p^2}^2)}{\sqrt{g_{{p}^{2},33}(\mathbf{x}_{p^2}^2)}} \cdot
	e_j,
\end{equation}

\begin{equation}\label{vi deri tb}
\frac{\partial \tb^{1}}{\partial v_j^{1}}=-\frac{\tb^{1}e_j}{\mathbf{v}^{2}_{p^{2},3}}\cdot \frac{\partial_3 \eta_{p^{2}}}{\sqrt{g_{p^{2},33}}}\Big|_{x^{2}}.
\end{equation}
This leads to
\Be
\begin{split}\label{nabla_tbxb}
\nabla_x \tb = \frac{n(\xb)}{n(\xb) \cdot v},\ \
\nabla_v \tb = - \frac{\tb n(\xb)}{n(\xb) \cdot v},\\
\nabla_x \xb = Id_{3\times 3} - \frac{n(\xb) \otimes v}{n(\xb) \cdot v},\ \
\nabla_v \xb = - \tb Id + \frac{ \tb n(\xb) \otimes v}{n(\xb) \cdot v}.
\end{split}\Ee

For $i=1,2$, $j=1,2$, $k=0,1$(recall $x^0=x$ in Definition \ref{definition: sto cycle}), 
\begin{equation}\label{xi deri xbp}
\frac{\partial \mathbf{x}_{p^{k+1},i}^{k+1}}{\partial x^{k}_j}=\frac{1}{\sqrt{g_{p^{k+1}, ii} (\mathbf{x}^{k+1}_{p^{k+1} } )}}
\left[
\frac{\p_{i} \eta_{p^{k+1}} (\mathbf{x}^{k+1}_{p^{k+1} } ) }{\sqrt{g_{p^{k+1},ii}(\mathbf{x}^{k+1}_{p^{k+1} } ) }}
- \frac{\mathbf{v}_{p^{k+1}, i}^{k+1}}{\mathbf{v}_{p^{k+1}, 3}^{k+1}}
\frac{\p_{3} \eta_{p^{k+1}} (\mathbf{x}^{k+1}_{p^{k+1} } )  }{\sqrt{g_{p^{k+1},33}(\mathbf{x}^{k+1}_{p^{k+1} } )}}
\right] \cdot e_j,
\end{equation}

\begin{equation}\label{xip deri xbp}
\frac{\p \mathbf{x}^{2}_{p^{2},i}}{\p{\mathbf{x}^{1  }_{p^{1 },j}}} = \frac{1}{\sqrt{g_{p^{2}, ii} (\mathbf{x}^{2}_{p^{2} } )}}
\left[
\frac{\p_{i} \eta_{p^{2}} (\mathbf{x}^{2}_{p^{2} } ) }{\sqrt{g_{p^{2},ii}(\mathbf{x}^{2}_{p^{2} } ) }}
- \frac{\mathbf{v}^{2}_{p^{2}, i}}{\mathbf{v}^{2}_{p^{2}, 3}}
\frac{\p_{3} \eta_{p^{2}} (\mathbf{x}^{2}_{p^{2} } )  }{\sqrt{g_{p^{2},33}(\mathbf{x}^{2}_{p^{2} } )}}
\right] \cdot \p_j \eta_{p^{1}}(\mathbf{x}^{1}_{p^{1} } ) ,
\end{equation}

\begin{equation}\label{vi deri xbp}
\frac{\partial \mathbf{x}^{2}_{p^{2},i}}{\partial v_j^{1}}=-\tb^{1}e_j \cdot \frac{1}{\sqrt{g_{p^{2},ii}(\mathbf{x}_{p^{2}}^{2}))}} \Big[\frac{\partial_i \eta_{p^{2}}(\mathbf{x}^{2}_{p^{2} } )}{\sqrt{g_{p^{2},ii}(\mathbf{x}^{2}_{p^{2} } )}}- \frac{\mathbf{v}_{p^{1},i}^{1}}{\mathbf{v}_{p^{2},3}^{2}}\frac{\partial_3 \eta_{p^{2}}(\mathbf{x}^{2}_{p^{2} } )}{\sqrt{g_{p^{2},33}(\mathbf{x}^{2}_{p^{2} } )}} \Big].
\end{equation}
\end{lemma}

The properties of the collision operator $\nu(v)$, $Kf$ and $\Gamma(f,f)$ in~\eqref{f_equation} are summarized in Lemma \ref{lemma:k_gamma} and Lemma \ref{Lemma: k tilde}.

\begin{lemma}\label{Lemma: k tilde}
The linear Boltzmann operator $K(f)$ in~\eqref{f_equation} is given by
\begin{equation}\label{Kf}
Kf(x,v)=\int_{\mathbb{R}^3}\mathbf{k}(v,u)f(x,u)\dd u.    
\end{equation}

For some $\varrho>0$, the kernel $\mathbf{k}(v,u)$ satisfies:
\Be\label{k_varrho}
 |\mathbf{k}  (v,u)| \lesssim \mathbf{k}_\varrho (v,u), \  | \nabla_ u\mathbf{k}  (v,u)| \lesssim  \langle u\rangle\mathbf{k}_\varrho (v,u)/|v-u|, \ \mathbf{k}_\varrho (v,u) := e^{- \varrho |v-u|^2}/ |v-u|.
  \Ee

If $0<\frac{\tilde{\theta}}{2}<\varrho$, if $0<\tilde{\varrho}<  \varrho- \frac{\tilde{\theta}}{2}$,
\begin{equation}\label{k_theta}
\mathbf{k}(v,u) \frac{e^{\tilde{\theta} |v|^2}}{e^{\tilde{\theta} |u|^2}} \lesssim  \mathbf{k}_{\tilde{\varrho}}(v,u) ,
\end{equation}
where $\mathbf{k}$ and $\mathbf{k}_{\tilde{\varrho}}$ are defined in~\eqref{k_varrho}. 

Moreover, the derivative on $\mathbf{k}(v,u)$ shares similar property: for some $\tilde{\varrho}<\varrho-\frac{\tilde{\theta}}{2}$,
\begin{equation}\label{nabla_k_theta}
| \nabla_v \mathbf{k}(v,u)| \frac{e^{\tilde{\theta}| v|^2}}{e^{\tilde{\theta}| u|^2}} \lesssim \frac{[1+|v|^2]\mathbf{k}_{\tilde{\varrho}}(v,u)}{| v-u|}    .
\end{equation}

\end{lemma}

\begin{proof}

\textit{Proof of \eqref{k_theta}}

From the standard Grad estimate in \cite{R}, the $\mathbf{k}$ in \eqref{k_varrho} equals to $\mathbf{k}_1(v,u)+\mathbf{k}_2(v,u)$, where
\begin{align}
\mathbf{k}_1(v,u)   &= C_{\mathbf{k}_1}|u-v|e^{-\frac{|v|^2+|u|^2}{4}}, \label{k1}\\
 \mathbf{k}_2(v,u)  & = C_{\mathbf{k}_2}\frac{1}{|u-v|}e^{-\frac{1}{8}|u-v|^{2}-\frac{1}{8}
\frac{(|u|^{2}-|v|^{2})^{2}}{|u-v|^{2}}}. \label{k2}
\end{align}
Then for some $\frac{1}{8}\geq \varrho>0$,
		\Be\notag
		\begin{split}
			\mathbf{k}(v,u) \frac{e^{\tilde{\theta} |v|^2}}{e^{\tilde{\theta} |u|^2}}
			\lesssim  \frac{1}{|v-u| } \exp\left\{- {\varrho} |v-u|^{2}
			-  {\varrho} \frac{ ||v|^2-|u|^2 |^2}{|v-u|^2} + \tilde{\theta} |v|^2 - \tilde{\theta} |u|^2
	\right\}.
		\end{split}\Ee

		Let $v-u=\eta $ and $u=v-\eta $, the exponent equals
		\begin{eqnarray*}
			&&- \varrho|\eta |^{2}-\varrho\frac{||\eta |^{2}-2v\cdot \eta |^{2}}{%
				|\eta |^{2}}-\tilde{\theta} \{|v-\eta |^{2}-|v|^{2}\} \\
			&=&-2 \varrho |\eta |^{2}+ 4 \varrho v\cdot \eta - 4 \varrho\frac{|v\cdot
				\eta |^{2}}{|\eta |^{2}}-\tilde{\theta} \{|\eta |^{2}-2v\cdot \eta \} \\
			&=&(-2 \varrho-\tilde{\theta}  )|\eta |^{2}+(4 \varrho+2\tilde{\theta} )v\cdot \eta -%
			4 \varrho\frac{\{v\cdot \eta \}^{2}}{|\eta |^{2}}.
		\end{eqnarray*}%
If $0<\tilde{\theta} <2 \varrho$ then the discriminant of the above quadratic form of
		$|\eta |$ and $\frac{v\cdot \eta }{|\eta |}$ is
		\begin{equation*}
		(4 \varrho+2\tilde{\theta} )^{2}-4
		(-2 \varrho-\tilde{\theta}  )(-%
		4 \varrho)
		=4\tilde{\theta} ^{2}- 16 \varrho^2<0.
		\end{equation*}%
		Hence, the quadratic form is negative definite. We thus have, for $%
		0<\tilde{\varrho}< \varrho - \frac{\tilde{\theta}}{2}  $, the following perturbed quadratic form is still negative definite: $
		-(\varrho - \tilde{\varrho})|\eta |^{2}-(\varrho - \tilde{\varrho})\frac{||\eta
			|^{2}-2v\cdot \eta |^{2}}{|\eta |^{2}}-\tilde{\theta} \{|\eta |^{2}-2v\cdot \eta \}  \leq 0,$ i.e, 
\begin{align*}
- {\varrho} |v-u|^{2}
			-  {\varrho} \frac{ ||v|^2-|u|^2 |^2}{|v-u|^2} + \tilde{\theta} |v|^2 - \tilde{\theta} |u|^2     \leq -\tilde{\rho}|v-u|^2,  
\end{align*}
we conclude \eqref{k_theta}.

\textit{Proof of \eqref{nabla_k_theta}.} Taking the derivative to \eqref{k1} and \eqref{k2} we have
\begin{align*}
\nabla_v \mathbf{k}_1(v,u)    & = \frac{v-u}{| v-u|}\mathbf{k}_1(v,u) - v \mathbf{k}_1(v,u),
\end{align*}
\begin{align*}
\nabla_v \mathbf{k}_2(v,u)    &  = \frac{v-u}{| v-u|^2} \mathbf{k}_2(v,u) \\
&- \mathbf{k}_2(v,u)\Big[\frac{v-u}{4} + \frac{v(| u|^2-| v|^2)| v-u|^2 - (| u|^2-| v|^2)^2  (v-u) }{4| v-u|^4} \Big] .
\end{align*}
Since $|u|^2 - |v|^2    = |u-v|^2 + 2v\cdot (u-v) $, we have
\begin{align*}
   & \big| |u|^2 - |v|^2\big| \lesssim |u-v|^2 + |v||u-v| , \\
   & \big| |u|^2 - |v|^2\big|^2 \lesssim |u-v|^4 + |v|^2 |u-v|^2.
\end{align*}
This leads to
\begin{align*}
\Big|\nabla_v \mathbf{k}(v,u) \frac{e^{\tilde{\theta}| v|^2}}{e^{\tilde{\theta}| u|^2}} \Big|    &\lesssim  \big[\frac{1+|u|^2+|v|^2}{| v-u|}+| v-u|+\langle v\rangle\big] [\mathbf{k}_1(v,u)+\mathbf{k}_2(v,u) ]   \frac{e^{\tilde{\theta}| v|^2}}{e^{\tilde{\theta}| u|^2}}  \\
    &\lesssim \big[\frac{1 +|v|^2}{| v-u|}+| v-u|+\langle v\rangle\big] \mathbf{k}_{\tilde{\varrho}}(v,u) \\
    & =\big[\frac{1 +|v|^2}{| v-u|}+| v-u|+\langle v\rangle\big] \frac{e^{-\tilde{\varrho}|v-u|^2}}{|v-u|} \lesssim \frac{1+|v|^2}{|v-u|} \frac{e^{-c|v-u|^2}}{|v-u|}
\end{align*}
for some $c<\tilde{\varrho}$. In the second line we have applied \eqref{k_theta}. In the last line we used that for $c<\tilde{\varrho}$, we have
\begin{align*}
    &  e^{-\tilde{\varrho}|v-u|^2} \lesssim \frac{e^{-c|v-u|^2}}{|v-u|}, \ e^{-\tilde{\varrho}|v-u|^2} \lesssim \frac{e^{-c|v-u|^2}}{|v-u|^2}.
\end{align*}
Since $c<\tilde{\varrho}<\varrho-\frac{\tilde{\theta}}{2}$, for ease of notation, we conclude \eqref{nabla_k_theta} with coefficient $\tilde{\varrho}$.

\end{proof}

\begin{lemma}\label{lemma:k_gamma}
With $\tilde{\theta}\ll \theta$, $K(f)$ is bounded as
\begin{equation}\label{Kf_bdd}
\begin{split}
\Vert w_{\tilde{\theta}}(v)K(f)(t,x,v)\Vert_\infty    &  \lesssim e^{-\lambda t}\sup_t e^{\lambda t}\Vert wf(t)\Vert_\infty,\\
  | w_{\tilde{\theta}}(v)K(f)(t,x,v)|_\infty  &\lesssim e^{-\lambda t} \sup_t e^{\lambda t}| wf(t)|_\infty.
\end{split}
\end{equation}

Moreover, for $\nu$ and $\Gamma$ in~\eqref{f_equation}, we have
{\color{black}
\begin{equation}\label{Gamma bounded}
\begin{split}
 \Big\Vert \frac{w}{\langle v\rangle}\Gamma(f,g)(t,x,v) \Big\Vert_\infty    & \lesssim \Vert wf(t)\Vert_\infty\times \Vert wg(t)\Vert_\infty,\\
  \Big| \frac{w}{\langle v\rangle}\Gamma(f,g)(t,x,v) \Big|_\infty    &\lesssim  | wf(t)|_\infty\times  |wg(t)|_\infty,
\end{split}
\end{equation}
}

\begin{equation}\label{nablav nu}
\nu \gtrsim \sqrt{|v|^2+1}, \ \ |\nabla_v \nu| \lesssim 1.
\end{equation}

{\color{black}
\begin{equation}\label{Gamma_est}
\begin{split}
 | \nabla_x \Gamma(f,g)(t,x,v) | &\lesssim \langle v\rangle \| wf(t) \|_\infty
  | \nabla_x g(t,x,v) |  + \Vert wf(t)\Vert_\infty
  \int_{\R^3} \mathbf{k}(v,u)
 |\nabla_x g(t,x,u)|
  \dd u  \\
  & + \Vert wg\Vert_\infty \int_{\mathbb{R}^3} \mathbf{k}(v,u) |\nabla_x f(t,x,u)| \dd u.
\end{split}
\end{equation}
}

\end{lemma}

\begin{proof}

\textit{Proof of~\eqref{Kf_bdd}.}
By the property of $K(f)$ in \eqref{k1} and \eqref{k2}, we apply Lemma \ref{Lemma: k tilde} to have
{\color{black}\begin{align*}
    & |w_{\tilde{\theta}}(v) K(f)(t,x,v)| =\Big|\int_{\mathbb{R}^3} \mathbf{k}(v,u)\frac{w_{\tilde{\theta}}(v)}{w_{\tilde{\theta}}(u)} w_{\tilde{\theta}}(u) f(t,x,u)  \dd u  \Big| \\
    & \lesssim \Vert w_{\tilde{\theta}}f(t)\Vert_\infty \Big| \int_{\mathbb{R}^3} \mathbf{k}_{\tilde{\varrho}}(v,u) \dd u \Big| \leq \Vert wf(t)\Vert_\infty.
\end{align*}}

\textit{Proof of~\eqref{Gamma bounded}.}
{\color{black}By the definition of $\Gamma(f,g)$ in~\eqref{Gamma_f}, we have
\begin{align*}
  &|\Gamma(f,g)(t,x,v)|  \\
 & \lesssim  \iint_{\mathbb{R}^3\times \mathbb{S}^2}|(v-u)\cdot \omega|\sqrt{\mu(u)}[|f(t,x,v')||g(t,x,u')|+|f(t,x,v)||g(t,x,u)|] \\
  & \lesssim w^{-1}(v)\Vert wf(t)\Vert_\infty \Vert wg(t)\Vert_\infty \iint_{\mathbb{R}^3\times \mathbb{S}^2} |(v-u)\cdot \omega|\sqrt{\mu(u)}w^{-1}(u) \dd u\\
  &\lesssim w^{-1}(v) \langle v\rangle\Vert wf(t)\Vert_\infty \Vert wg(t)\Vert_\infty.
\end{align*}
Here we have used $|u'|^2 + |v'|^2 = |u|^2 + |v|^2$.}

\textit{Proof of \eqref{nablav nu}.} The proof is standard, we refer to \cite{G}.

\textit{Proof of \eqref{Gamma_est}.}  {\color{black} Taking the derivative we have
\begin{align*}
   |\nabla_x \Gamma(f,g)| & = |\Gamma(\nabla_x f,g) + \Gamma(f,\nabla_x g) |  \notag \\
   &\leq \Vert wf(t)\Vert_\infty \Big|\Gamma(e^{-\theta |\cdot|^2},\nabla_x g) \Big| + \Vert wg(t)\Vert_\infty\Big|\Gamma(\nabla_x f, e^{-\theta |\cdot |^2})  \Big|.
\end{align*}
The $\Gamma$ terms above share the same form as the linear operator $L$ in \eqref{Gamma}. Then \eqref{Gamma_est} follows the expression of $\nu f$ and $Kf$ in \eqref{Kf}, with replacing $\mu$ by a different exponent $\sqrt{\mu}e^{-\theta |v|^2}$. For ease of notation, we keep the same $\mathbf{k}(v,u)$ in \eqref{Gamma_est}.
}

\end{proof}

The properties of the kinetic weight $\alpha$ are summarized in the following two lemmas.
\begin{lemma}\label{lemma:velocity}
{\color{black}For a $C^3$ convex domain}, the kinetic weight $\alpha,\tilde{\alpha}$ in~\eqref{kinetic_distance} is almost invariant along the characteristic:
\begin{align}
\alpha(x,v)\thicksim \alpha(x-sv,v), \ \tilde{\alpha}(x,v)\thicksim \tilde{\alpha}(x-sv,v)
\ \text{as long as } x-sv \in \bar{\O}
. \label{Velocity_lemma}
\end{align}
Here we recall the notation $\thicksim$ is defined in~\eqref{thicksim}.

From the definition of $\tilde{\alpha}$ in~\eqref{kinetic_distance}, we have an upper bound for $\alpha(x,v)$ as
\begin{equation}\label{n geq alpha}
\alpha(x,v)\leq \min\{1,\tilde{\alpha}(x,v)\}, \ \alpha(x,v)\lesssim |v|.
\end{equation}
Thus 
\begin{equation}\label{bdr_alpha}
 \tilde{\alpha}(x,v)\sim |n(\xb(x,v))\cdot v|, \ \ \frac{1}{|n(\xb(x,v))\cdot v|}\lesssim \frac{1}{\alpha(x,v)}   .
\end{equation}

\end{lemma}

\begin{proof}
\eqref{n geq alpha} follows from the definition of $\chi$ in~\eqref{chi}. 

Since $\xb(x,v)\in \p\O$, we have $|n(\xb(x,v))\cdot v|\sim \tilde{\alpha}(\xb(x,v),v)$. By~\eqref{Velocity_lemma}, we further have $|n(\xb(x,v))\cdot v|\sim \tilde{\alpha}(x,v) \geq \alpha(x,v)$, then we conclude~\eqref{bdr_alpha}.

Then we focus on~\eqref{Velocity_lemma}.

First we prove that $\chi(s)$ has the following property:
\begin{equation}\label{chi_property}
 s\chi'(s)\leq 4\chi(s).   
\end{equation}
From the definition in~\eqref{chi}, when $s\geq 2$, $s\chi'(s)=0\leq 4\chi(s)$. When $s\leq \frac{1}{2}$, we have $\chi(s)=s,$ and thus $s\chi'(s)=s \leq 4\chi(s)=4s$. When $\frac{1}{2}<s<2$, since $\chi'(s)\leq 1$, we have $s\chi'(s)\leq s<2=4\chi(1/2)\leq 4\chi(s)$. Then we conclude~\eqref{chi_property}.

From the definition of $\tilde{\alpha}(x,v)$ in~\eqref{kinetic_distance}, we directly compute 
\begin{align}
  & |v\cdot \nabla_x \alpha(x,v)| = \chi'(\tilde{\alpha}(x,v))|v\cdot \nabla_x \tilde{\alpha}(x,v)| \notag\\
   &=\chi'(\tilde{\alpha}(x,v)) \frac{|2v\cdot \nabla \xi(x)[v\cdot \nabla^2 \xi \cdot v] - 2v\cdot \nabla \xi(x)[v\cdot \nabla^2 \xi \cdot v]- 2v \{ v\cdot \nabla^3 \xi \cdot v \} \xi(x)| }{\tilde{\alpha}(x,v)}  \notag\\
   & = 2\chi'(\tilde{\alpha}(x,v))\frac{|v\{v\cdot \nabla^3 \xi \cdot v\}\xi(x)|}{\tilde{\alpha}(x,v)} \lesssim \chi'(\tilde{\alpha}(x,v)) \frac{|v|^3|\xi(x)|}{\tilde{\alpha}(x,v)} \notag\\
   &\lesssim \chi'(\tilde{\alpha}(x,v))\frac{|v|(\tilde{\alpha}(x,v))^2}{\tilde{\alpha}(x,v)} = |v|\chi'(\tilde{\alpha}(x,v))\tilde{\alpha}(x,v)\lesssim |v|\chi(\tilde{\alpha}(x,v))=|v|\alpha(x,v).\label{computation_alpha}
\end{align}
In the last line, we have used the convexity~\eqref{convex} in the first inequality, and used~\eqref{chi_property} in the second inequality. Then for some $C$, we have
\[-C |v|\alpha(x,v)\leq v\cdot \nabla_x \alpha(x,v) \leq C |v|\alpha(x,v).\]
Since
\[\frac{\dd }{\dd s} \alpha(x-sv,v)=-v\cdot \nabla_x \alpha(x-sv,v),\]
by Gronwall's inequality, we conclude
\[e^{-C|v|s}\alpha(x-sv,v)\leq \alpha(x,v)\leq e^{C|v|s}\alpha(x-sv,v).\]
Since $x-sv\in \bar{\O}$, we have $|v|s\leq C_{\Omega}$ and conclude $\alpha(x,v)\sim \alpha(x-sv,v)$.

Following the same computation in~\eqref{computation_alpha} with replacing $\chi'$ by $1$, for $\tilde{\alpha}(x,v)$ we have
\begin{align*}
|v\cdot \tilde{\alpha}(x,v)|\lesssim |v|\tilde{\alpha}(x,v).
\end{align*}
Again by Gronwall's inequality, we can conclude $\tilde{\alpha}(x,v)\sim \tilde{\alpha}(x-sv,v)$.

We conclude~\eqref{Velocity_lemma}.
\end{proof}

The following nonlocal-to-local estimate for the kinetic weight~\eqref{kinetic_distance} is essential in the analysis of the regularity.
\begin{lemma}\label{lemma:nonlocal_to_local}
Recall the definition of $\alpha$ in~\eqref{kinetic_distance}, then for $C>0$ and $\varrho>0$, we have
\begin{equation}\label{est:nonlocal_wo_e}
\begin{split}
  \int^{t}_{t- \tb(x,v)} \int_{\mathbb{R}^{3}}\frac{e^{-C\langle v\rangle (t-s)} e^{-\varrho |v-u|^{2}}}{|v-u| \alpha (  x-(t-s)v, {u}) }\mathrm{d}u\mathrm{d}s
  &
 \lesssim \frac{1}{\alpha(x,v)}.
\end{split}
\end{equation}
For $\e,\delta\ll 1$, recall the notation $o(1)$ in~\eqref{ll},
\begin{equation}\label{est:nonlocal_with_e}
\begin{split}
  \int^{t}_{t- \e} \int_{\mathbb{R}^{3}}\frac{e^{-C\langle v\rangle (t-s)} e^{-\varrho |v-u|^{2}}}{|v-u| \alpha (  x-(t-s)v, {u}) }\mathrm{d}u\mathrm{d}s
  &
 \lesssim \frac{o(1)}{\alpha(x,v)},
\end{split}
\end{equation}

{\color{black}
\begin{align}
  \int^{t}_{t-\tb(x,v)} \int_{|u|<\delta}\frac{e^{-C\langle v\rangle (t-s)} e^{-\varrho |v-u|^{2}}}{|v-u| \alpha (  x-(t-s)v, {u}) }\mathrm{d}u\mathrm{d}s
  &
 \lesssim \frac{o(1)}{\alpha(x,v)}.   \label{est:nonlocal_delta}
\end{align}
}

\end{lemma}

\begin{proof}
We prove the following estimate. When $t-\tb(x,v)\leq t-t_1 \leq t-t_2 \leq t$, for $0<\beta<1$ and some $C_1>0$,
\begin{align}
&\int^{t-t_1}_{t-t_2} \int_{\mathbb{R}^{3}}\frac{e^{-C\langle v\rangle (t-s)} e^{-\varrho | v-u| ^{2}}}{| v-u|  \alpha (  x-(t-s)v, {u}) }\mathrm{d}u\mathrm{d}s  \lesssim \frac{| e^{-C_1\langle v\rangle t_1 }-e^{-C_1\langle v\rangle t_2}| ^{\beta}}{\alpha(x,v)}, \label{t_1_t_2}\\
&\int^{t-t_1}_{t-t_2} \int_{|u|<\delta}\frac{e^{-C\langle v\rangle (t-s)} e^{-\varrho | v-u| ^{2}}}{| v-u|  \alpha (  x-(t-s)v, {u}) }\mathrm{d}u\mathrm{d}s  \lesssim o(1)\frac{| e^{-C_1\langle v\rangle t_1 }-e^{-C_1\langle v\rangle t_2}| ^{\beta}}{\alpha(x,v)},    \label{t_1_t_2_delta}
\end{align}
Clearly Lemma \ref{lemma:nonlocal_to_local} follows from~\eqref{t_1_t_2} and \eqref{t_1_t_2_delta}.

In the proof we assume $\alpha(x-(t-s)v,u)=\tilde{\alpha}(x-(t-s)v,u)$. For the other case, from~\eqref{kinetic_distance} we have $\alpha(x-(t-s)v,u)\gtrsim 1$, combining with $\frac{e^{-\varrho|v-u|^2}}{|v-u|}\in L^1_u$, we have
\begin{align*}
    &   \int^{t-t_1}_{t-t_2} \int_{\mathbb{R}^{3}}\mathbf{1}_{\alpha(x-(t-s)v,u)\gtrsim 1}\frac{e^{-C\langle v\rangle (t-s)} e^{-\varrho |v-u|^{2}}}{|v-u| \alpha (  x-(t-s)v, {u}) }\mathrm{d}u\mathrm{d}s \\
    &\lesssim   \int^{t-t_1}_{t-t_2} \int_{\mathbb{R}^{3}}\frac{e^{-C\langle v\rangle (t-s)} e^{-\varrho |v-u|^{2}}}{|v-u| }\mathrm{d}u\mathrm{d}s  \\
    &\lesssim \int^{t-t_1}_{t-t_2} e^{-C\langle v\rangle(t-s)} \dd s \lesssim |e^{-C\langle v\rangle t_1}-e^{-C\langle v\rangle t_2}|\lesssim \frac{|e^{-C\langle v\rangle t_1}-e^{-C\langle v\rangle t_2}|}{\alpha(x,v)}.
\end{align*}
In the last inequality, we used~\eqref{n geq alpha}. Then we start to prove~\eqref{t_1_t_2}.

\textit{Step 1. }We claim that for $y\in \bar{\Omega}$ and $\varrho>0$,
\begin{equation}\label{eqn: int alpha du}
\int_{\mathbb{R}^3}\frac{e^{-\varrho| v-u|^2}}{| v-u|   }
\frac{1}{ \alpha(y,u)}
\dd u \lesssim 1+ | \ln
  | \xi(y)|    |   +  | \ln
  | v|    |  .
\end{equation}
\begin{equation}\label{int_alpha_du_delta}
\int_{|u|<\delta}\frac{e^{-\varrho| v-u|^2}}{| v-u|   }
\frac{1}{ \alpha(y,u)}
\dd u \lesssim o(1)[1+ | \ln
  | \xi(y)|    |   +  | \ln
  | v|    |  ].
\end{equation}

Recall (\ref{bar_v}), we set $\mathbf{v}= \mathbf{v} (y)$ and $\mathbf{u}= \mathbf{u} (y)$. For $| u| \gtrsim {| v| } $, we have
\Be\label{com:alpha}
\big[ | \mathbf{u}_3 (y)|^2+| \xi(y)| | u|^2 \big]^{1/2}
\gtrsim \big[| \mathbf{u}_3 (y)|^2+| \xi(y)| | v|^2 \big]^{1/2}.
\Ee
Thus when $\frac{|v|}{4}\leq |u|\leq 4|v|$, we have
\begin{align} 
 {\int_{\frac{| v| }{4} \leq | u| \leq 4| v|  }}\lesssim & \
\iint
 \frac{e^{-\varrho| \mathbf{v}_\parallel-\mathbf{u}_\parallel|^2}}{| \mathbf{v}_\parallel-\mathbf{u}_\parallel| }\dd \mathbf{u}_\parallel
 \int_0^{ 4| v| }\frac{  \dd \mathbf{u}_3}{\big[ | \mathbf{u}_3|^2+| \xi (y)| | v|^2\big]^{1/2}} \notag \\
 \leq& \  \int_0^{4| v| }\frac{\dd \mathbf{u}_3}{\big[ | \mathbf{u}_3|^2+| \xi (y)| | v|^2\big]^{1/2}}\notag \\
= & \ \ln\Big( \sqrt{| \mathbf{u}_3|^2+| \xi(y)| | v|^2}+| \mathbf{u}_3| \Big)\Big| _{0}^{4| v| } \notag \\
=& \ln (\sqrt{16| v|^2 + | \xi(y)| | v|^2}+ 16 | v|^2)- \ln (\sqrt{  | \xi(y)| | v|^2})\notag \\
\lesssim &  \ln | v|  + \ln | \xi(y)| . \label{int_1/a_1}
\end{align}

If $| u| \geq 4| v| $, then $| u-v|^2 \geq \frac{| v|^2}{4}+\frac{| u|^2}{4}$ and hence the exponent is bounded as\\
$e^{- \varrho | v-u|^2}
\leq e^{- \frac{\varrho}{8} | v|^2}e^{- \frac{\varrho}{8} | u|^2} e^{- \frac{\varrho}{2} | v-u|^2}
$. This, together with (\ref{com:alpha}), implies
\begin{align} 
{\int_{| u| \geq 4| v| }}
\lesssim & \  e^{- \frac{\varrho}{8}| v|^2}
\iint
 \frac{e^{-\frac{\varrho}{2}| \mathbf{v}_\parallel-\mathbf{u}_\parallel|^2}}{| \mathbf{v}_\parallel-\mathbf{u}_\parallel| }\dd \mathbf{u}_\parallel
 \int_0^{\infty}\frac{
 e^{-\frac{\varrho}{8}| \mathbf{u}_3|^2}
 }{\big[ | \mathbf{u}_3|^2+| \xi (y)| | v|^2\big]^{1/2}}   \dd \mathbf{u}_3
\notag \\
 \lesssim & \  e^{- \frac{\varrho}{8}| v|^2}
  \int_0^{\infty}\frac{
 e^{-\frac{\varrho}{8}| \mathbf{u}_3|^2}
 }{\big[ | \mathbf{u}_3|^2+| \xi (y)| | v|^2\big]^{1/2}}   \dd \mathbf{u}_3 \notag \\
 \lesssim & \ e^{- \frac{\varrho}{8}| v|^2} +  e^{- \frac{\varrho}{8}| v|^2}
  \int_0^{1}\frac{\dd \mathbf{u}_3}{\big[ | \mathbf{u}_3|^2+| \xi (y)| | v|^2\big]^{1/2}}  \notag\\
= & \ e^{- \frac{\varrho}{8}| v|^2}+ e^{- \frac{\varrho}{8}| v|^2}\ln\Big( \sqrt{| \mathbf{u}_3|^2+| \xi(y)| | v|^2}+| \mathbf{u}_3| \Big)\Big| _{0}^{1}\notag\\
=& \
e^{- \frac{\varrho}{8}| v|^2} \Big\{1+
\ln (\sqrt{1 + | \xi(y)| | v|^2}
+ 1
)- \ln (\sqrt{  | \xi(y)| | v|^2}
)
\Big\} \notag\\
\leq &  \
e^{- \frac{\varrho}{8}| v|^2}
\{\ln | v|  + \ln | \xi(y)| \} . \label{int_1/a_2}
\end{align}

When $| u| \leq \frac{| v| }{4}$, we have $| v-u| \geq \big| | v| -| u| \big| \geq | v| -\frac{| v| }{4}\geq \frac{| v| }{2}$, this leads
\begin{align*}
{\int_{| u| \leq \frac{| v| }{4}}}
&  \lesssim    \frac{
{\color{black}e^{- \frac{\varrho}{4}| v|^2}}
}{| v| }\int_{| \mathbf{u}_3| +| \mathbf{u}_\parallel| \leq \frac{| v| }{2}} \frac{ \dd\mathbf{u}_3 \dd \mathbf{u}_\parallel }{\Big[| \mathbf{u}_3|^2+| \xi(y)| | \mathbf{u}_\parallel|^2 \Big]^{1/2}}\\
&\lesssim | v| 
{\color{black}e^{- \frac{\varrho}{4}| v|^2}}\int_{| 
\mathbf{\tilde{u}}_3
| \leq \frac{1}{2}}
 \int_{| \mathbf{\tilde{u}}_\parallel| \leq \frac{1}{2}}
 \frac{
 \dd \mathbf{\tilde{u}}_\parallel \dd \mathbf{\tilde{u}}_3
}{\Big[| \mathbf{\tilde{u}}_3|^2+| \xi(y)| | \mathbf{\tilde{u}}_\parallel|^2 \Big]^{1/2}},
\end{align*}
where we have used $| v| \tilde{u} = u$. Using the polar coordinate $\mathbf{\tilde{u}}_{1}=| \mathbf{\tilde{u}}_\parallel| \cos\rho, \mathbf{\tilde{u}}_{2}=| \mathbf{\tilde{u}}_\parallel| \sin \rho$, we proceed the computation as
\begin{align} 
& {\int_{| u| \leq \frac{| v| }{4}}} \lesssim \  | v| 
{\color{black}e^{- \frac{\varrho}{4}| v|^2}}\int_0^{ \frac{1}{2}} \dd \mathbf{\tilde{u}}_3
\int_{0}^{2\pi}\int_{0}^{\sqrt{1/2}}
 \frac{   | \mathbf{\tilde{u}}_\parallel| \dd | \mathbf{\tilde{u}}_\parallel|   \dd \rho
}{\Big[| \mathbf{\tilde{u}}_3|^2+| \xi(y)| | \mathbf{\tilde{u}}_\parallel|^2 \Big]^{1/2}} \notag\\
 \lesssim & \
  | v| 
{\color{black}e^{- \frac{\varrho}{4}| v|^2}}
\int_0^{ \frac{1}{2}}
\dd \mathbf{\tilde{u}}_3
\int_{0}^{ {1/2}}
 \frac{   \dd | \mathbf{\tilde{u}}_\parallel|^2
}{\Big[| \mathbf{\tilde{u}}_3|^2+| \xi(y)| | \mathbf{\tilde{u}}_\parallel|^2 \Big]^{1/2}} \notag\\
=& \   | v| 
{\color{black}e^{- \frac{\varrho}{4}| v|^2}}
\int_0^{ \frac{1}{2}} \dd \mathbf{\tilde{u}}_3
\frac{1}{| \xi(y)| }
\Big( \sqrt{
| \mathbf{\tilde{u}}_3|^2+\frac{| \xi(y)| }{2} } - | \mathbf{\tilde{u}}_3| 
\Big) \notag\\
=& \ \frac{| v| 
{\color{black}e^{- \frac{\varrho}{4}| v|^2}}}{| \xi(y)| } \notag\\
& \ \  \times 
  \bigg\{\frac{1}{2}| \mathbf{\tilde{u}}_3| \sqrt{| \mathbf{\tilde{u}}_3|^2+\frac{| \xi| }{2}}+\frac{| \xi| }{4}\log\Big(\sqrt{| \mathbf{\tilde{u}}_3|^2+\frac{| \xi| }{2}}+| \mathbf{\tilde{u}}_3|  \Big)-\frac{1}{2}| \mathbf{\tilde{u}}_3| ^{1/2} \bigg\}\bigg| _{| \mathbf{\tilde{u}}_3| =0}^{| \mathbf{\tilde{u}}_3| =1/2} \notag\\
  = & \  \frac{| v| e^{-C| v|^2}}{| \xi| }\bigg\{\frac{1}{4}\sqrt{\frac{1}{4}+\frac{| \xi| }{2}}+\frac{\xi}{4}\log\Big( \sqrt{\frac{1}{4}+\frac{| \xi| }{2}}+\frac{1}{2}\Big)-\frac{| \xi| }{4}\log\Big(\sqrt{\frac{| \xi| }{2}} \Big)-\frac{1}{8} \bigg\} \notag\\
 \lesssim& \   \frac{| v| e^{-C| v|^2}}{| \xi| }\Big[| \xi| \log(| \xi| )+| \xi| \log\Big(1+\sqrt{1+| \xi| }\Big) \Big] \notag\\
 \lesssim  & \ 1+\log(| \xi(y)| ). \label{int_1/a_3}
\end{align}

Collecting terms from (\ref{int_1/a_1}), (\ref{int_1/a_2}), and (\ref{int_1/a_3}), we prove (\ref{eqn: int alpha du}). 

The proof of \eqref{int_alpha_du_delta} is the same with using the following estimate:
\begin{align*}
    &\int_{|u|<\delta} \frac{e^{-\varrho |v-u|^2}}{|v-u|} \frac{1}{\alpha(y,u)} \dd u \\
 & \lesssim
\iint_{|\mathbf{u}_\parallel|<\delta}
 \frac{e^{-\varrho| \mathbf{v}_\parallel-\mathbf{u}_\parallel|^2}}{| \mathbf{v}_\parallel-\mathbf{u}_\parallel| }\dd \mathbf{u}_\parallel
 \int_0^{\delta}\frac{  \dd \mathbf{u}_3}{\big[ | \mathbf{u}_3|^2+| \xi (y)| | \mathbf{u}_\parallel|^2\big]^{1/2}} \\
 &\lesssim o(1)\int_0^{\delta}\frac{  \dd \mathbf{u}_3}{\big[ | \mathbf{u}_3|^2+| \xi (y)| | \mathbf{u}_\parallel|^2\big]^{1/2}}.
\end{align*}

  \textit{Step 2. } We prove the following statement: {\color{black}{for $x\in \partial \Omega$ }}, we can choose $0 < \tilde{\delta}\ll_\O 1$ such that
\Be
\begin{split}
 \tilde{\delta}^{1/2}   | v\cdot \nabla \xi ( x-(t-s) v) | \gtrsim_\O&  \ | v| \sqrt{-\xi (
x-(t-s) v)},
 \\
  \text{for} & \ s\in
  \Big[t-\tb(x,v),
t-\tb(x,v)+ {\color{black} \tilde{t}} \Big]\cup
 \Big[t- {\color{black} \tilde{t}},
t  \Big]
,
\label{vpxxi_alpha_1}
\end{split} \Ee
\Be
 \tilde{\delta}^{1/2} \times
\tilde{\alpha}(x,v)  \lesssim_\O    | v| \sqrt{-\xi (
x-(t-s) v)}, \ \ \text{for}  \  s\in \Big[
t-\tb(x,v)+ {\color{black} \tilde{t}}, t- {\color{black} \tilde{t}}
\Big], \label{vpxxi_alpha_2}
\Ee
here $\tilde{t} =\min\{ \frac{\tb(x,v)}{2}, \tilde{\delta}\frac{\tilde{\alpha}(x,v)}{| v|^2}\}$. We note that when $\tilde{t}<\tilde{\delta}\frac{\tilde{\alpha}(x,v)}{| v|^2}$, \eqref{vpxxi_alpha_2} vanishes.

If $v=0$ or $v\cdot  \nabla \xi (x)\leq 0$, since {\color{black} $x\in \partial \Omega$}, then (\ref{vpxxi_alpha_1}) and (\ref{vpxxi_alpha_2}) hold clearly with {\color{black}$\tb(x,v)=\tilde{t}=0$}. We may assume $v\neq 0$ and $v\cdot \nabla \xi(x)>0$. {\color{black}Since $v\cdot \nabla\xi (x)> 0$, from \eqref{Velocity_lemma}, we have $| v\cdot \nabla\xi (x_{\mathbf{b}}(x,v) )|  >0$. Then we must have $v\cdot \nabla\xi (x_{\mathbf{b}}(x,v) ) <0$, otherwise $v\cdot \nabla \xi(\xb)>0$ implies $\xb$ is not the backward exit position defined in \eqref{BET}.} By the mean value theorem there exists at least one $t^{\ast } \in (t-\tb(x,v)  , t)$ such that $v\cdot \nabla \xi
(x- t^*v)=0$. Moreover, the convexity in (\ref{convex}) leads to $\frac{d}{ds}\big(v\cdot \nabla \xi (x-(t-s)v)\big)=v\cdot
\nabla ^{2}\xi (x-(t-s)v )\cdot v\geq C_{\xi }| v| ^{2},$ and therefore $t^{\ast }\in (t-\tb(x,v), t)$ is unique. 

Let $s \in \big[t-\tilde{t},t\big]$ for $0<\tilde{\delta}\ll 1$ and $\tilde{t}\leq \tilde{\delta}\frac{\tilde{\alpha}(x,v)}{| v|^2}$. Then from the fact that $v\cdot \nabla_x \xi(x-(t-\tau )v)$ is non-decreasing function in $\tau$ and {\color{black}$x\in \partial \Omega$},
\Be \label{exp_xi}
| v|^2  (-1) \xi(x-(t-s)v)=\int^t_s | v|^2   v\cdot \nabla_x \xi(x-(t-\tau )v)    \dd \tau
\leq \tilde{\delta}  \tilde{\alpha}(x,v)
| v\cdot \nabla_x \xi(x)|  .
\Ee
By~\eqref{Velocity_lemma} $ | v\cdot \nabla_x \xi (x)| \leq \tilde{\alpha}(x,v) \leq C_\O \tilde{\alpha}(x-(t-s)v,v)\leq C_\O \{
| v \cdot \nabla \xi (x-(t-s)v)|  + \Vert  \nabla_x^2 \xi \Vert _\infty | v|  \sqrt{- \xi(x-(t-s)v)}
\},$ we choose $\tilde{\delta}\ll C_\O \Vert  \nabla_x^2 \xi \Vert _\infty^{-2}$ and absorb the $|v|^2 |\xi(x-(t-s)v)|$ term in
\begin{align*}
    &  \tilde{\delta} \tilde{\alpha}(x,v)  \times C_\O  \{
| v \cdot \nabla \xi (x-(t-s)v)|  + \Vert  \nabla_x^2 \xi \Vert _\infty | v|  \sqrt{- \xi(x-(t-s)v)}
\} \\
    & \leq
\tilde{\delta} \times \{C_\O \Vert  \nabla_x^2 \xi \Vert _\infty | v|  \sqrt{- \xi (x-(t-s)v)}\}^2 + \tilde{\delta}\times \{C_\Omega |v\cdot \nabla \xi(x-(t-s)v)|\}^2
\end{align*}
by the left hand side of (\ref{exp_xi}). This gives (\ref{vpxxi_alpha_1}) for $s \in \big[t-
\tilde{t} ,t\big]$. The proof for $s \in \big[t-\tb(x,v),
t-\tb(x,v)+ \tilde{t} \big]$ is same.

For \eqref{vpxxi_alpha_2}, we assume $\tilde{t} = \tilde{\delta} \frac{\tilde{\alpha}(x,v)}{| v|^2}$, otherwise \eqref{vpxxi_alpha_2} vanishes. Since $v\cdot \nabla_x \xi(x-(t-s)v)>0$ when $s>t_*$, we have that $\xi(x-(t-s)v)$ is increasing in $s$ when $s>t_*$, 
thus $| v|^2 (-1) \xi(x-(t-s)v)\geq | v|^2 (-1) \xi\Big(x-\tilde{\delta} \frac{\tilde{\alpha}(x,v)}{| v|^2} v\Big)$ for $s\in \big[t- t^*, t- \tilde{\delta} \frac{\tilde{\alpha}(x,v)}{| v|^2}
\big]$. By an expansion, for $s^*:=t-\tilde{\delta} \frac{\tilde{\alpha}(x,v)}{| v|^2}$,

\Be\label{exp_xi_td}
\begin{split}
    & | v|^2  (-1) \xi\Big(x-\tilde{\delta} \frac{\tilde{\alpha}(x,v)}{| v|^2} v\Big) \\
    &=| v|^2 ( v\cdot \nabla_x \xi(x ))\tilde{\delta} \frac{\tilde{\alpha}(x,v)}{| v|^2}
-\int^t_{s^*}\int^t_{\tau}| v|^2 v\cdot \nabla_x^2 \xi(x-(t-\tau^\prime) v) \cdot v\dd \tau^\prime\dd \tau.
\end{split}
\Ee
The last term of (\ref{exp_xi_td}) is bounded by $ \Vert  \nabla_x^2 \xi \Vert _\infty\tilde{\delta}^2 \big(\frac{\tilde{\alpha}(x,v)}{| v|^2}\big)^2 | v| ^4 \leq  \Vert  \nabla_x^2 \xi \Vert _\infty\tilde{\delta}^2   (\tilde{\alpha}(x,v))^2$. Since $v \cdot \nabla_x \xi(x) = \tilde{\alpha}(x,v)$ when $x\in \p\O$, for $\tilde{\delta} \ll \Vert  \nabla_x^2 \xi \Vert _\infty^{-1/2}$, the right hand side of (\ref{exp_xi_td}) is bounded below by $\frac{\tilde{\delta}}{2} (\tilde{\alpha}(x,v))^2$. This completes the proof of (\ref{vpxxi_alpha_2}) when $s\in \big[t- t^*, t- \tilde{\delta} \frac{\tilde{\alpha}(x,v)}{| v|^2} 
\big]$. The proof for the case of $s\in \big[ t - \tb(x,v)+ \tilde{\delta} \frac{\tilde{\alpha}(x,v)}{| v|^2} , t- t^*
\big]$ is same.

\textit{Step 3. } From (\ref{eqn: int alpha du}), for the proof of \eqref{t_1_t_2}, it suffices to estimate
 \Be\label{int_lnxi+lnv}
 \begin{split}
     &  \int^t_{t-\tb(x,v)} \mathbf{1}_{t-t_2\geq s \geq t-t_1 } e^{-C \langle v\rangle (t-s)} \big|  \ln | \xi(x-(t-s)v)| \big|  \dd s \\
      &  + \int^t_{t-\tb(x,v)} \mathbf{1}_{t-t_2\geq s \geq t-t_1 } e^{-C \langle v\rangle (t-s)}\big(1+ \big|  \ln |  v| \big|  \big)\dd s .
 \end{split}
 \Ee
We bound the second term of (\ref{int_lnxi+lnv}) as
\Be\label{small_t}
(1+| \ln| v| | )\int^{t-t_2}_{t-t_1}e^{-C\langle v\rangle (t-s)}   \lesssim (1+| \ln| v| | )\langle v\rangle^{-1} | e^{-C\langle v\rangle t_2}-e^{-C\langle v\rangle t_1}| .
\Ee
{\color{black}For the first term of \eqref{int_lnxi+lnv}, we first assume $x\in \partial \Omega$.} For utilizing (\ref{vpxxi_alpha_1}) and (\ref{vpxxi_alpha_2}), we split the first term of (\ref{int_lnxi+lnv}) as
 \Be\label{int_ln}
 \underbrace{ \int^t_{t- \tilde{t}}
 + \int^{t -\tb(x,v)+ \tilde{t}}_{t -\tb(x,v) }}_{(\ref{int_ln})_1}
 + \underbrace{ \int_{t -\tb(x,v)+ \tilde{t}} ^{t- \tilde{t}}}_{(\ref{int_ln})_2}.
 \Ee
Without loss of generality, we assume {\color{black}$t-t_2\in [t-\tilde{t},t]$, $t-t_1\in [t-\tb(x,v)+\tilde{t},t-\tilde{t}]$.} For the first term $(\ref{int_ln})_1$ we use a change of variables $s \mapsto   -\xi(x-(t-s)v)$ in $s \in [t-\tb(x,v), t-t^*]$ and $s \in [t-t^*,t ]$ separately with $\dd s = | v\cdot \nabla_x \xi(x-(t-s)v)| ^{-1} \dd | \xi| $. From (\ref{exp_xi}), we deduce $| \xi(x-(t-s)v) |  \leq \tilde{\delta}\frac{\tilde{\alpha}^2(x,v)}{| v|^2}$. Then applying H\"{o}lder inequality with $\beta+(1-\beta)=1$ and using (\ref{vpxxi_alpha_1}), we get
 \begin{align}
 &(\ref{int_ln})_1 \mathbf{1}_{\{t-t_2\in [t-\tilde{t},t]\},t-t_1\in [t-\tb(x,v)+\tilde{t},t-\tilde{t}]}  \notag\\
   & \lesssim \Big(\big[\int^{t-t_2}_{t-\tilde{t}}          e^{-C\langle v\rangle (t-s)/\beta} \dd s \big]^{\beta}+\big[\int^{t-\tb(x,v)+\tilde{t}}_{t-t_1}          e^{-C\langle v\rangle (t-s)/\beta} \dd s \big]^{\beta}\Big)  \notag\\
   & \ \ \ \times  \Big[ \int_0^{\tilde{\delta}\frac{\tilde{\alpha}^2(x,v)}{| v|^2}}
| \ln| \xi|  | ^{1/(1-\beta)}
 \frac{\dd | \xi| }{ \tilde{\delta}^{-1/2}| v|  \sqrt{ |  \xi | }}  \Big]^{1-\beta} \notag\\
      &\lesssim | e^{-C\langle v\rangle t_2/\beta }-e^{-C\langle v\rangle t_1/\beta }| ^\beta\frac{1}{| v| ^{1-\beta}}, \label{est:int_ln_1}
 \end{align}
where we have used $t-\tilde{t}>t-t_1, \quad t-t_2>t-\tb(x,v)+\tilde{t}$ and $\frac{| \ln | \xi| | ^{1/(1-\beta)}}{\sqrt{\xi}} \in L^1_{loc}(0,\infty)$ for $\beta<1$ in the last line.

On the other hand, since $\Vert \xi\Vert_\infty \lesssim 1$, we only consider the contribution of $|\xi|\lesssim 1$ in $\ln(|\xi|)$. From (\ref{vpxxi_alpha_2}),
\begin{align}
 (\ref{int_ln})_2& \leq \int^{t-t_2}_{t-t_1} e^{-C \langle v\rangle (t-s)} \Big|  \ln \Big(\tilde{\delta} \frac{ (\tilde{\alpha}(x,v))^2}{| v|^2}\Big) \Big|  \dd s
 \notag \\
 &\leq
2 \int^{t-t_2}_{t-t_1} e^{-C \langle v\rangle (t-s)}  \{
|  \ln \tilde{\delta} |   + | \ln \tilde{\alpha}(x,v)|  + | \ln | v| | \}
 \dd s
\notag \\
& \leq 2
| e^{-C\langle v\rangle t_2}-e^{-C\langle v\rangle t_1}|  \langle v\rangle^{-1/2}\{
 |  \ln \tilde{\delta} |  +  | \ln \tilde{\alpha}(x,v)|    +  | \ln | v| | \}, \label{est:int_ln_2}
\end{align}
where we have used a similar estimate of (\ref{small_t}). 

{\color{black}Now as assume $x\notin \partial \Omega$. We find $\bar{x}\in \partial \Omega$ and $\bar{t}$ so that 
\[x = \bar{x} - (\bar{t} - t)v \text{ and } \bar{t}>t. \]
Then clearly, $x-(t-s)v = \bar{x}-(\bar{t}-s)v.$ Since $\bar{x}\in \partial \Omega$, applying the same computation as \eqref{est:int_ln_1} and \eqref{est:int_ln_2}, the first term of \eqref{int_lnxi+lnv} is bounded by
\begin{align}
    & \int^{\bar{t}}_{\bar{t}-\tb(\bar{x},v)} \mathbf{1}_{t-t_2\geq s \geq t-t_1} e^{-C\langle v\rangle (\bar{t}-s)}| \ln | \xi(\bar{x}-(\bar{t}-s)v)| |  \dd s \notag\\
    & \lesssim \eqref{est:int_ln_1} + | e^{-C\langle v\rangle t_2}-e^{-C\langle v\rangle t_1}|  \langle v\rangle^{-1/2}\{
 |  \ln \tilde{\delta} |  +  | \ln \tilde{\alpha}(\bar{x},v)|    +  | \ln | v| | \}. \label{x_out_bdr}
\end{align}
Here we used $\tilde{\alpha}(\bar{x},v) \thicksim \tilde{\alpha}(x,v)$ from \eqref{Velocity_lemma} to obtain the same upper bound as \eqref{est:int_ln_1}. Again using $\tilde{\alpha}(\bar{x},v) \thicksim \tilde{\alpha}(x,v)$, we conclude
\[\eqref{x_out_bdr} \lesssim \eqref{est:int_ln_1} + \eqref{est:int_ln_2}.\]}
Then collecting~\eqref{int_lnxi+lnv},~\eqref{small_t},~\eqref{est:int_ln_1} and~\eqref{est:int_ln_2},
we conclude that for $C_1=C/\beta$,
\begin{align}
&\int^{t-t_1}_{t-t_2} \int_{\mathbb{R}^{3}}\frac{e^{-C\langle v\rangle (t-s)} e^{-\varrho | v-u| ^{2}}}{| v-u|  \alpha (  x-(t-s)v, {u}) }\mathrm{d}u\mathrm{d}s \notag\\
  &\lesssim  | e^{-C_1\langle v\rangle t_1 }-e^{-C_1\langle v\rangle t_2}| ^{\beta}\big[ \langle v\rangle^{-1/2}\big(
1+  | \ln | v| |  +  | \ln  \tilde{\alpha}(x,v)| \big) +  \frac{1}{| v| ^{1-\beta}}\big] \label{NLN general}.
\end{align}
Then by~\eqref{n geq alpha}, $\tilde{\alpha}(x,v)\lesssim |v|$, $\alpha(x,v)\lesssim \min\{1,|v|\}$, and thus $1\leq \frac{1}{\alpha(x,v)}$, we bound
\begin{align*}
   \frac{|\ln |v||}{\langle v\rangle^{1/2}} & = \mathbf{1}_{|v|\geq 1} + \mathbf{1}_{|v|<1}\lesssim 1+\mathbf{1}_{|v|< 1} \frac{|\ln |v||}{\langle v\rangle^{1/2}}  \\
& \lesssim 1+\mathbf{1}_{|v|<1}\frac{1}{|v|}\lesssim \frac{1}{\alpha(x,v)},
\end{align*}
\begin{align*}
  \frac{|\ln \tilde{\alpha}(x,v)|}{\langle v\rangle^{1/2}}  & =\mathbf{1}_{\tilde{\alpha}(x,v)\geq 1} + \mathbf{1}_{\tilde{\alpha}(x,v)<1} \lesssim \mathbf{1}_{|v|\gtrsim 1}\frac{|\ln |v||}{\langle v\rangle^{1/2}} + \mathbf{1}_{\alpha(x,v)\lesssim 1}\frac{|\ln \alpha(x,v)|}{\langle v\rangle^{1/2}}\lesssim \frac{1}{\alpha(x,v)},
\end{align*}
\begin{align*}
   \frac{1}{|v|^{1-\beta}} & =\mathbf{1}_{|v|\geq 1} + \mathbf{1}_{|v|<1} \leq 1+\frac{1}{(\alpha(x,v))^{1-\beta}}\lesssim \frac{1}{\alpha(x,v)} .
\end{align*}
Therefore, we conclude
\[\eqref{NLN general}\lesssim \frac{| e^{-C_1\langle v\rangle t_1 }-e^{-C_1\langle v\rangle t_2}|^{\beta}}{\alpha(x,v)},\]
then~\eqref{t_1_t_2} follows. \eqref{t_1_t_2_delta} follows by the extra $o(1)$ from \eqref{int_alpha_du_delta}.

\end{proof}

\ \\

\section{Uniform-in-time weighted $C^1$ estimate}\label{sec:C1}
In this section, we prove Theorem \ref{thm:dynamic_regularity}. We will take $\tilde{\theta} \ll \theta$ in \eqref{f0_regularity}, and thus $\tilde{\theta}$ satisfies the condition in Lemma \ref{Lemma: k tilde}.  

First we note that under the conditions in Theorem \ref{thm:dynamic_regularity}, the dynamical stability~\eqref{dym_stability} has been established in~\cite{EGKM}, and the steady problem~\eqref{linearization_steady} is well-posedness with $L^\infty$ bound
\Be\label{infty_bound}
 \Vert wf_s\Vert_\infty + |wf_s|_\infty  \lesssim \| T_W-T_0\|_{L^\infty(\partial\Omega)} .
 \Ee
Moreover, \cite{CK} constructed the weighted $C^1$ estimate of the steady problem
\begin{align}
 \| w_{\tilde{\theta}}(v) \alpha(x,v)  \nabla_x f_s (x,v)\|_\infty
& \lesssim \| T_W-T_0\|_{C^1(\p\O)}.
\label{estF_n}
\end{align}
Here $\alpha(x,v)$ is the kinetic weight defined in~\eqref{kinetic_distance}. We sketch the proof of~\eqref{estF_n} in Section \ref{sec:proof_sketch}.
{\color{black}\begin{remark}
In Section \ref{sec:proof_sketch}, we will only focus on the a priori estimate. The existence can be justified through the sequential argument, we refer detail to \cite{CK}.   
\end{remark}
}

For $(x,v)\in \Omega \times \mathbb{R}^3$, we apply method of characteristic to~\eqref{f_equation} and obtain
\begin{align*}
   f(t,x,v) & = \mathbf{1}_{t\leq \tb(x,v)} e^{-\nu t} f_0(x-tv,v) \\
    & + \mathbf{1}_{t>\tb(x,v)} e^{-\nu \tb(x,v)} f(t-\tb,\xb,v) \\
    &  + \int^t_{\max\{0,t-\tb\}} e^{-\nu(v)(t-s)} \int_{\mathbb{R}^3} \mathbf{k}(v,u) f(s,x-(t-s)v,u) \dd u \dd s\\ 
    & + \int^t_{\max\{0,t-\tb\}} e^{-\nu(v)(t-s)} h(s,x-(t-s)v,u)  \dd s.
\end{align*}
Here we denote all the $\Gamma$ term in~\eqref{f_equation} as
\begin{equation}\label{h}
h(t,x,v) = [\Gamma(f_s,f)+\Gamma(f,f_s)+\Gamma(f,f)](t,x,v).
\end{equation}

We take spatial derivative to get
\begin{align}
 \nabla_x f(t,x,v)  & =\mathbf{1}_{t\leq \tb(x,v)} e^{-\nu t} \nabla_x f_0(x-tv,v) \label{nabla_x_1}\\
  &    - \mathbf{1}_{t>\tb} \nu \nabla_x \tb(x,v) e^{-\nu \tb(x,v)} f(t-\tb,\xb,v) \label{nabla_x_2}\\
  & - \mathbf{1}_{t>\tb} e^{-\nu \tb} \nabla_x \tb \partial_t f(t-\tb,\xb,v)  \label{nabla_x_3}\\
  & + \mathbf{1}_{t>\tb} e^{-\nu \tb} \sum_{i=1,2} \nabla_x \mathbf{x}_{p^1,i}^1  \partial_{\mathbf{x}_{p^1,i}^1} f(t-\tb,\xb,v) \label{nabla_x_4}\\
  & + \int^t_{\max\{0,t-\tb\}} e^{-\nu(t-s)}\int_{\mathbb{R}^3} \mathbf{k}(v,u) \nabla_x f(s,x-(t-s)v,u) \dd u \dd s \label{nabla_x_5}\\
  & +\mathbf{1}_{t>\tb} e^{-\nu\tb}\nabla_x \tb \int_{\mathbb{R}^3} \mathbf{k}(v,u) f(t-\tb,\xb,u) \dd u \label{nabla_x_6}\\
  & + \int_{\max\{0,t-\tb\}}^t e^{-\nu(t-s)} \nabla_x h(s,x-(t-s)v,u)  \dd s \label{nabla_x_7}\\
  & + \mathbf{1}_{t>\tb} e^{-\nu \tb} \nabla_x \tb h(t-\tb,\xb,u) . \label{nabla_x_8}
\end{align}
Here $h$ in~\eqref{nabla_x_7} and~\eqref{nabla_x_8} is defined in~\eqref{h}. In~\eqref{nabla_x_4}, $\mathbf{x}_{p^1,i}^1$ is defined in Definition \ref{definition: chart} and $\xb=x^1=\eta_{p^1}(\mathbf{x}_{p^1}^1)$. 

\begin{remark}
With a compatibility condition it is standard to check the piece-wise formula~\eqref{nabla_x_1} - \eqref{nabla_x_8} is actually a weak derivative of $f$ and continuous across $\{t=\tb(x,v)\}$(see \cite{GKTT2} for the details).
{\color{black}
We will only focus on the a priori estimate to the weighted $C^1$ estimate $\Vert w_{\tilde{\theta}}(v)\alpha\nabla_x f\Vert_\infty.$ For existence one can apply iteration sequence, we refer detail to \cite{GKTT}. }
\end{remark}

Before proving Theorem \ref{thm:dynamic_regularity}, we need some preparations for~\eqref{nabla_x_3}, \eqref{nabla_x_4}, \eqref{nabla_x_7} and~\eqref{nabla_x_8}, they are presented in Section \ref{sec:para_3}, Section \ref{sec:parp_4} and Section \ref{sec:para_78} respectively. Then in Section \ref{sec:proof_thm_2} we prove Theorem \ref{thm:dynamic_regularity}.

\ \\

\subsection{Preparation for~\eqref{nabla_x_3}.}\label{sec:para_3}
The following corollary gives a decay-in-time estimate for the temporal derivative $\partial_t f(t)$ in \eqref{nabla_x_3}:
\begin{corollary}\label{corollary:time_derivative}
Assume the compatibility condition~\eqref{compatibility} and~\eqref{compatibility_derivative} hold for $f_0(x,v)$ and $\p_t f_0(x,v)$ defined in~\eqref{partial_t_f0}.

If $\partial_t f_0(x,v)$ satisfies~\eqref{partial_f0_bdd} and is continuous away from $\gamma_0$, then $\partial_t f$ is continuous away from the grazing set $\gamma_0$. Moreover, we have the $L^\infty$ estimate
\begin{align}
  \Vert w\partial_t f(t)\Vert_\infty + |w\partial_t f(t)|_\infty  &\lesssim  e^{-\lambda t} \Vert w\partial_t f_0\Vert_\infty. \label{time_derivative_decay}
\end{align}
\end{corollary}

To prove Corollary \ref{corollary:time_derivative} we first cite a result from~\cite{EGKM} regarding the decay-in-time $L^\infty$ estimate for the following linear problem:
\begin{equation}\label{inflow}
\begin{split}
  \partial_t f+v\cdot \nabla_x f + L(f) =g , & \\
    f|_{\gamma_-} = P_\gamma f + r. &   
\end{split}
\end{equation}
Here $P_\gamma f$ is defined as the projection
\begin{equation}\label{p_gamma}
P_\gamma f = \sqrt{\mu(v)} \int_{n(x)\cdot u>0} f(u) \sqrt{\mu(u)}\{n(x)\cdot u\} \dd u.
\end{equation}

\begin{proposition}[Proposition 6.1 and 7.1 in~\cite{EGKM}]\label{prop:linfty_decay}
Let $\Vert wf_0\Vert_\infty+ |\langle v\rangle wr|_\infty + \big\Vert \frac{w}{\langle v\rangle}g\big\Vert_\infty < \infty $ and $\iint \sqrt{\mu}g=\int_\gamma r\sqrt{\mu}=\iint f_0\sqrt{\mu}=0$, then there exists a solution to~\eqref{inflow} that satisfies
{\color{black}\begin{equation}\label{linfty_decay}
\Vert wf(t)\Vert_\infty + |wf(t)|_\infty \lesssim e^{-\lambda t}\{\Vert wf_0\Vert_\infty +\sup_{s\leq t} e^{\lambda s}\Big\Vert \frac{w}{\langle v\rangle}g(s) \Big\Vert_\infty + \sup_{s\leq t}e^{\lambda s}|\langle v\rangle wr(s)|_\infty\}.
\end{equation}}
If $f_0|_{\gamma_-}=P_\gamma f_0 + r_0$, $f_0$, $r$ and $g$ are continuous away from $\gamma_0$, then $f(t,x,v)$ is continuous away from $\gamma_0$.

\end{proposition}

\begin{proof}[Proof of Corollary \ref{corollary:time_derivative}]
First, we consider the following linear problem:
\begin{align*}
    \partial_t \partial_t f + v\cdot \nabla_x \partial_t f+L(\partial_tf) = \partial_t g &\\
    \partial_t f|_{\gamma_-} = P_{\gamma} \partial_t f + \partial_t r,
\end{align*}
with the initial condition
 satisfying \eqref{compatibility_derivative}:
 \begin{equation*}
\partial_t f_0|_{\gamma_-} = P_\gamma \partial_t f_0 + \partial_t r_0.    
 \end{equation*}
Assuming $\p_t r$ and $\p_t g$ are continuous away from $\gamma_0$ and $\iint \sqrt{\mu}\partial_t g=\int_\gamma \partial_t r\sqrt{\mu}=\iint \partial_t f_0\sqrt{\mu}=0$, by Proposition \ref{prop:linfty_decay}, there exists a unique $f$ and $\p_t f$ such that
\begin{align}
   \Vert w\p_t f(t)\Vert_\infty + |w \p_t f(t)|_\infty & \lesssim e^{-\lambda t}\{\Vert w \p_t f_0\Vert_\infty +\sup_{s\leq t} e^{\lambda s} \Big\Vert \frac{w}{\langle v\rangle} \partial_t g(s)\Big\Vert_\infty+ \sup_{s\leq t}e^{\lambda s}|\langle v\rangle w \partial_t r(s)|_\infty \}. \label{linfty_inflow}
\end{align}

In order to apply \eqref{linfty_inflow}, we consider the following iteration sequence:
\begin{equation}\label{iteration}
\partial_t f^{\ell+1} + v\cdot \nabla_x  f^{\ell+1}+L(  f^{\ell+1}) = \Gamma(f_s, f^\ell)+\Gamma( f^\ell,f_s)+\Gamma( f^\ell, f^\ell).  
\end{equation}
The temporal derivative of \eqref{iteration} reads
\begin{equation}\label{iteration_p_t}
\p_t (\partial_t f^{\ell+1})+v\cdot \nabla_x \partial_t f^{\ell+1}+L( \partial_t f^{\ell+1}) =\partial_t [ \Gamma(f_s, f^\ell)+\Gamma( f^\ell,f_s)+\Gamma( f^\ell, f^\ell)],    
\end{equation}
with boundary condition and initial condition given by
\[\p_t f^{\ell+1}|_{\gamma_-} = P_\gamma \p_tf^{\ell+1}+\frac{M_W(x,v) - \mu}{\sqrt{\mu}}\int_{n(x)\cdot u>0}\partial_t f^{\ell}\sqrt{\mu}\{n(x)\cdot u\} \dd u,\] 
\[\partial_t f^{\ell+1}(0,x,v) = \partial_t f_0(x,v).\] 
The initial sequence is defined to be the initial condition $\partial_t f^0(t,x,v) = \p_t f_0(x,v)$.

We first check all assumptions for \eqref{linfty_inflow} hold and then apply \eqref{linfty_inflow} to the iteration sequence \eqref{iteration_p_t}.

We check the condition $\iint \partial_t f_0\sqrt{\mu}=0$. From the definition of $\partial_t f_0(x,v)$ in~\eqref{partial_t_f0}, we apply Green's identity to have
\begin{align*}
  &\iint v\cdot \nabla_x f_0 \sqrt{\mu} \\
  & = \int_{n(x)\cdot v>0} [n(x)\cdot v] f_0 \sqrt{\mu} + \int_{n(x)\cdot v<0} [n(x)\cdot v] f_0 \sqrt{\mu} \\
  & =  \int_{n(x)\cdot v>0} [n(x)\cdot v] f_0 \sqrt{\mu} + \int_{n(x)\cdot v<0} [n(x)\cdot v]\sqrt{\mu} \frac{M_W(x,v)}{\sqrt{\mu(v)}} \int_{n(x)\cdot u>0} f_0(x,u)[n(x)\cdot u] \sqrt{\mu(u)}\\
  & =  \int_{n(x)\cdot v>0} [n(x)\cdot v] f_0 \sqrt{\mu} - \int_{n(x)\cdot u>0} f_0(x,u)[n(x)\cdot u] \sqrt{\mu(u)}=0.
\end{align*}
In the third line, we applied the compatibility condition~\eqref{compatibility} for $f_0$ and used 
\[\int_{n(x)\cdot v<0} M_W(x,v)[ n(x)\cdot v]=-1.\]
It is clear from the mass conservation that $\iint [-L(f_0)+\Gamma(f_s,f_0)+\Gamma(f_0,f_s)+\Gamma(f_0,f_0)]\sqrt{\mu} =0$. These estimates lead to the desired condition $\iint \partial_t f_0(x,v) \sqrt{\mu} = 0$.

For $\p_t r$ and $\p_t g$, we have that
\[\int_{v\in \mathbb{R}^3} \frac{M_W(x,v)-\mu}{\sqrt{\mu}}\sqrt{\mu} [n(x)\cdot v] = 0,\]
and again from mass conservation,
\[\iint \partial_t  [ \Gamma(f_s, f^\ell)+\Gamma( f^\ell,f_s)+\Gamma( f^\ell, f^\ell)]  \sqrt{\mu}  =0.\]
The conditions for~\eqref{linfty_inflow} are satisfied. We apply~\eqref{linfty_inflow} to \eqref{iteration_p_t} and have
\begin{align}
   &\Vert w\partial_t f^{\ell+1}(t)\Vert_\infty + |w\partial_t f^{\ell+1}(t)|_\infty \notag\\
   & \lesssim e^{-\lambda t}\Big\{\Vert w\partial_t f_0\Vert_\infty+ \sup_{s\leq t}e^{\lambda s}\Big|\langle v\rangle w \frac{M_W(x,v)-\mu}{\sqrt{\mu}}\int_{n(x)\cdot u>0}\partial_t f^{\ell}(s)\sqrt{\mu}\{n(x)\cdot u\}\Big|_\infty   \label{bdr_ell}\\
   & \ \ \ \ \ \ + \sup_{s\leq t}e^{\lambda s}\Big[\Big\Vert  \frac{w}{\langle v\rangle}\partial_t \Gamma(f_s,f^{\ell})(s)\Big\Vert_\infty +\Big\Vert  \frac{w}{\langle v\rangle}\partial_t \Gamma(f^{\ell},f_s)(s)\Big\Vert_\infty+\Big\Vert  \frac{w}{\langle v\rangle}\partial_t \Gamma(f^{\ell},f^{\ell})(s)\Big\Vert_\infty\Big]\Big\} . \notag
\end{align}
For the third line in \eqref{bdr_ell}, with $w=e^{\theta|v|^2}$, by~\eqref{Gamma bounded} and~\eqref{infty_bound} we have
{\color{black}\begin{align*}
  \Big\Vert \frac{w}{\langle v\rangle}\partial_t \Gamma(f^\ell,f_s)(s)\Big\Vert_\infty  &= \Big\Vert \frac{w}{\langle v\rangle}\Gamma(\partial_t f^\ell(s),f_s)\Big\Vert_\infty \lesssim \Vert wf_s\Vert_\infty \Vert w\partial_t f^\ell(s) \Vert_\infty = o(1)\Vert w\partial_t f^\ell(s)\Vert_\infty.
  \end{align*}
Similarly
\[  \Big\Vert \frac{w}{\langle v\rangle}\partial_t \Gamma(f_s,f^\ell)(s)\Big\Vert_\infty  \lesssim o(1) \Vert w\partial_t f^\ell(s) \Vert_\infty,\]
and by~\eqref{dym_stability},
\begin{align*}
 \Big\Vert \frac{w}{\langle v\rangle}\partial_t \Gamma(f^\ell,f^\ell)(s)\Big\Vert_\infty   & \lesssim \Big\Vert \frac{w}{\langle v\rangle}\Gamma(\partial_t f^\ell,f^\ell)(s)\Big\Vert_\infty + \Big\Vert \frac{w}{\langle v\rangle}\Gamma(f^\ell,\partial_t f^\ell)(s)\Big\Vert_\infty \\
    &\lesssim \Vert wf^\ell(s)\Vert_\infty\Vert w\partial_t f^\ell(s)\Vert_\infty=o(1)\Vert w\partial_t f^\ell(s)\Vert_\infty.
\end{align*}}

To deal with the boundary term in the second line of \eqref{bdr_ell}, with $\theta<\frac{1}{4}$, there exists $\delta$ such that $\theta<\frac{1}{4}-\delta$. Then we have
\begin{align*}
 &\langle v\rangle w\frac{M_W(x,v)-\mu}{\sqrt{\mu}}  \\
 & =\langle v\rangle e^{\theta|v|^2} \frac{e^{-|v|^2/2T_W(x)}/T_W(x) - e^{-|v|^2/2}}{e^{-|v|^2/4}}  \\
 & =\langle v\rangle e^{(\theta+\frac{1}{4})|v|^2}\times \Big[\frac{1}{T_W(x)}\big(e^{-\frac{|v|^2}{2T_W(x)}} - e^{-\frac{|v|^2}{2}} \big) + e^{-\frac{|v|^2}{2}}\big(\frac{1}{T_W(x)}-1 \big) \Big]\\
 &\lesssim \langle v\rangle e^{(\frac{1}{2}-\delta)|v|^2}\Big[ e^{-\frac{|v|^2}{2}}\Vert T_W-1\Vert_{L^\infty(\p\O)} + \frac{|v|^2}{2} \big|\frac{1}{T_W(x)} -1 \big|e^{-\frac{|v|^2}{2\inf \{ T_W(x)\}}} \Big] \\
 & \lesssim \Vert T_W-1\Vert_{L^\infty(\p\O)} \langle v\rangle[1+|v|^2] e^{(\frac{1}{2}-\delta-\frac{1}{2\inf \{T_W(x)\}})|v|^2} \\
 &\lesssim  \Vert T_W-1\Vert_{L^\infty(\p\O)} \frac{1}{\delta} e^{(\frac{1}{2}-3\delta-\frac{1}{2\inf \{T_W(x)\}})|v|^2} = o(1).
\end{align*}
In the fourth line, we applied mean value theorem. In the last line we used the Taylor expansion of $e^{x}$ to have $|v|^2 \leq \frac{e^{\delta|v|^2}}{\delta}$ and thus
\[\langle v\rangle[1+|v|^2] \lesssim 1+|v|^4\lesssim \frac{e^{2\delta |v|^2}}{\delta}.\] 
Also we take $\Vert T_W-1\Vert_{L^\infty(\p\O)}$ to be small enough such that $\frac{\Vert T_W-1\Vert_{L^\infty(\p\O)}}{\delta} \ll 1$ and $\frac{1}{2}-3\delta - \frac{1}{2\inf\{T_W(x)\}}<0$.

Then we control the boundary term in the second line \eqref{bdr_ell} as
\[\Big| \langle v\rangle w \frac{M_W(x,v)-\mu}{\sqrt{\mu}}\int_{n(x)\cdot u>0}\partial_t f^{\ell}(s)\sqrt{\mu}\{n(x)\cdot u\}\Big|_\infty\lesssim o(1) | w\partial_t f^\ell(s)|_\infty.\]

Hence we obtain the $L^\infty$ estimate for $\p_t f^{\ell+1}$ in \eqref{bdr_ell} as
\begin{align}
   &\Vert w\partial_t f^{\ell+1}(t)\Vert_\infty + |w\partial_t f^{\ell+1}(t)|_\infty \notag\\
   & \lesssim e^{-\lambda t}\{\Vert w\partial_t f_0\Vert_\infty + o(1)\sup_{s\leq t}e^{\lambda s}\Vert w\partial_t f^\ell(s)\Vert_\infty + o(1)\sup_{s\leq t}e^{\lambda s}| w\partial_t f^\ell(s)|_\infty \}. \label{ell+1}
\end{align}
This is equivalent to the fact that, for some $C>1$ and $c\ll 1$ and $c\times C\ll 1$, we have
\begin{align*}
   &\Vert w\partial_t f^{\ell+1}(t)\Vert_\infty + |w\partial_t f^{\ell+1}(t)|_\infty \\
   & \leq C e^{-\lambda t}\{\Vert w\partial_t f_0\Vert_\infty + c\sup_{s\leq t}e^{\lambda s}\Vert w\partial_t f^\ell(s)\Vert_\infty + c\sup_{s\leq t}e^{\lambda s}| w\partial_t f^\ell(s)|_\infty \}.
\end{align*}

Next we use induction to prove that, for all $\ell\geq 0$,
\begin{align*}
  \Vert w\partial_t f^{\ell}(t)\Vert_\infty + |w\partial_t f^{\ell}(t)|_\infty  & \leq 2Ce^{-\lambda t}\Vert w\partial_t f_0\Vert_\infty.
\end{align*}
Since the initial sequence is given by $\p_t f^0(t,x,v)=\p_t f_0(x,v)$, we have $\Vert w\partial_t f^{1}(t)\Vert_\infty + |w\partial_t f^{1}(t)|_\infty \leq Ce^{-\lambda t}\Vert w\p_t f_0\Vert_\infty$.

Assume the induction hypothesis holds, i.e,
\[  \Vert w\partial_t f^{\ell}(t)\Vert_\infty + |w\partial_t f^{\ell}(t)|_\infty   \leq 2Ce^{-\lambda t}\Vert w\partial_t f_0\Vert_\infty.\]
Applying \eqref{ell+1} to $\p_t f^{\ell+1}$ yields
\begin{align*}
   &\Vert w\partial_t f^{\ell+1}(t)\Vert_\infty + |w\partial_t f^{\ell+1}(t)|_\infty \\
   & \leq C e^{-\lambda t}\{\Vert w\partial_t f_0\Vert_\infty + 2Cc\Vert w\partial_t f_0\Vert_\infty  \} \leq 2C e^{-\lambda t}\Vert w\p_t f_0\Vert_\infty.
\end{align*}

The induction is valid and we conclude that for all $\ell$,
\begin{align*}
    & \sup_t e^{\lambda t}\Vert w\partial_t f^{\ell+1}(t)\Vert_\infty + \sup_t e^{\lambda t}|w\partial_t f^{\ell+1}(t)|_\infty \lesssim \Vert w\partial_t f_0\Vert_\infty.  
\end{align*}

With the uniform-in-$\ell$ estimate, it is standard to take the difference $\p_t f^{\ell+1}-\p_t f^{\ell}$ to conclude that $\p_t f^{\ell}$ is Cauchy sequence in $L^\infty$ and then pass the limit to conclude Corollary \ref{corollary:time_derivative}.

\end{proof}

\ \\

\subsection{Preparation for~\eqref{nabla_x_4}.} \label{sec:parp_4}
For the estimate of the derivative to~\eqref{nabla_x_4} at the boundary, we will use the boundary condition~\eqref{f_bc} with some properties from the convexity~\eqref{convex} in the following two lemmas.

\begin{lemma}\label{Lemma: nv<v2}
Given a $C^2$ convex domain in~\eqref{convex},
\Be\label{nv<v2}
\begin{split}
|n_{p^{j}} (\mathbf{x}_{p^{j}} ^{j}) \cdot  (x^{1} -
 \eta_{p^{2}} (\mathbf{x}_{p^{2}} ^{2})
 )| \thicksim  |x^{1} -
 \eta_{p^{2}} (\mathbf{x}_{p^{2}} ^{2})
 |^2,
 \ \  &j=1,2,
 \\
  {|\mathbf{v}^{1}_{p^{1}, 3}|  }  / {|\mathbf{v}^{1}_{p^{1}  }|}  \thicksim |x^{1} -
 \eta_{p^{2}} (\mathbf{x}_{p^{2}} ^{2})
 |.
 \end{split}
\Ee
For $j'=1,2$,
\Be\label{bound_vb_x}
  \bigg| \frac{\p[ n_{p^{j}} (\mathbf{x}_{p^{j}} ^{j}) \cdot  (x^{1} -
 \eta_{p^{2}} (\mathbf{x}_{p^{2}} ^{2})
 )]}{\p {\mathbf{x}_{p^{2},j^\prime}^{2}}}  \bigg| \lesssim \| \eta \|_{C^2}
 |x^{1} -
 \eta_{p^{2}} (\mathbf{x}_{p^{2}} ^{2})|
 ,
  \ \ j=1,2.
\Ee
\end{lemma}

\begin{proof}
First we prove~\eqref{nv<v2}. By Taylor's expansion, for $x,y\in \partial \Omega$ and some $0\leq t\leq 1$,
\begin{align*}
  \xi(y)-\xi(x)  & =0-0=\nabla \xi(x)\cdot (y-x) + \frac{1}{2}(y-x)^T \nabla^2 \xi(x+t(y-x)) (y-x).
\end{align*}
 Thus from~\eqref{normal}
\[ |n(x)\cdot (x-y)| \sim  (y-x)^T \nabla^2 \xi(x+t(y-x)) (y-x).  \]

From the convexity~\eqref{convex}, we have
\[|n_{p^{j}} (\mathbf{x}_{p^{j}} ^{j}) \cdot  (x^{1} -
 \eta_{p^{2}} (\mathbf{x}_{p^{2}} ^{2})
 )|\geq C_\Omega |x^{1} -
 \eta_{p^{2}} (\mathbf{x}_{p^{2}} ^{2})
 |^2.\]
 Since $\xi$ is $C^2$ at least,
\[|\{x^{1}-y\}\cdot n(x^{1})|\leq \Vert \xi\Vert_{C^2}|x^{1}-y|^2.\]

Also, notice that
\[|n_{p^{1}} (\mathbf{x}_{p^{1}} ^{1}) \cdot  (x^{1} -
 \eta_{p^{2}} (\mathbf{x}_{p^{2}} ^{2})
 )|=|\mathbf{v}^{1}_{p^{1}, 3}| (t^{2}-t^{1}),\]
thus
\begin{equation*}
  \begin{split}
    \frac{|\mathbf{v}^{1}_{p^{1}, 3}|  }{|\mathbf{v}^{1}_{p^{1}  }|}\geq  &  \frac{1}{|\mathbf{v}^{1}_{p^{1}  }|}\frac{C_\Omega}{|t^{1}-t^{2}|}\Big|x^{1}-x^{2} \Big|^2\\
     = & C_\Omega|x^{1}-x^{2}|
     =  C_\Omega |x^{1}-\eta_{p^{2}}(\mathbf{x}_{p^{2}}^{2})|.
  \end{split}
\end{equation*}
By the same computation, we can conclude
\[\frac{|\mathbf{v}^{1}_{p^{1}, 3}|  }{|\mathbf{v}^{1}_{p^{1}  }|} \leq C_\xi |x^{1}-\eta_{p^{2}}(\mathbf{x}_{p^{2}}^{2})| .\]

Then we prove~\eqref{bound_vb_x}. For $j=1$, $j'=1,2$ we apply~\eqref{orthogonal} and~\eqref{normal at xk} to obtain
\begin{equation}\label{tang*nor}
\begin{split}
   & \bigg| \frac{\p[ n_{p^{1}} (\mathbf{x}_{p^{1}} ^{1}) \cdot  (x^{1} -
 \eta_{p^{2}} (\mathbf{x}_{p^{2}} ^{2})
 )]}{\p {\mathbf{x}_{p^{2},j^\prime}^{2}}}  \bigg|\leq |n_{p^{1}} (\mathbf{x}_{p^{1}} ^{1}) \cdot \partial_{j'}\eta_{p^{2}}(\mathbf{x}_{p^{2}} ^{2})| \\
   & =\Big|n_{p^{1}} (\mathbf{x}_{p^{1}} ^{1}) \cdot \partial_{j'}\eta_{p^{1}}(\mathbf{x}_{p^{1}} ^{1})+n_{p^{1}} (\mathbf{x}_{p^{1}} ^{1}) \cdot \big[\partial_{j'}\eta_{p^{2}}(\mathbf{x}_{p^{2}} ^{2})-\partial_{j'}\eta_{p^{1}}(\mathbf{x}_{p^{1}} ^{1})\big]\Big|\\
   & \leq 0+ \Vert \eta\Vert_{C^2} |x^{1} -
 \eta_{p^{2}} (\mathbf{x}_{p^{2}} ^{2})|.
\end{split}
\end{equation}

For $j=2$, we again apply~\eqref{orthogonal} and~\eqref{normal at xk} to have
\begin{equation*}
  \begin{split}
     \bigg| \frac{\p[ n_{p^{2}} (\mathbf{x}_{p^{2}} ^{2}) \cdot  (x^{1} -
 \eta_{p^{2}} (\mathbf{x}_{p^{2}} ^{2})
 )]}{\p {\mathbf{x}_{p^{2},j^\prime}^{2}}}  \bigg|\lesssim &  |n_{p^{2}} (\mathbf{x}_{p^{2}} ^{2}) \cdot \partial_{j'}\eta_{p^{2}}|+\Vert \eta\Vert_{C^2}|x^{1} -
 \eta_{p^{2}} (\mathbf{x}_{p^{2}} ^{2})| \\
     = & \Vert \eta\Vert_{C^2}|x^{1} -
 \eta_{p^{2}} (\mathbf{x}_{p^{2}} ^{2})|.
  \end{split}
\end{equation*}

\end{proof}

\begin{lemma}\label{Lemma: change of variable} 
The following map is one-to-one
\Be\label{map_v_to_xbtb}
v^{1} \in   \{ n(x^{1}) \cdot v^{1} >0: \xb(x^{1},v^{1}) \in B(p^{2}, \delta_2)\} \mapsto
(\mathbf{x}^{2}_{p^{2},1}, \mathbf{x}^{2}_{p^{2},2}, \tb^{1}),
\Ee
with
\Be\label{jac_v_to_xbtb}
\det\left(\frac{\p (\mathbf{x}^{2}_{p^{2},1}, \mathbf{x}^{2}_{p^{2},2}, \tb^{1})}{\p v^{1}}\right)= \frac{1}{\sqrt{ g_{p^{2},11}(\mathbf{x}^{2}_{p^{2}})  g_{p^{2},22}(\mathbf{x}_{p^{2}}^{2}) }}
\frac{|\tb^{1}|^3}{  |n(x^{2}) \cdot v^{1}| }.
\Ee

\end{lemma}

\begin{proof}
Combining~\eqref{vi deri tb} and~\eqref{vi deri xbp} we conclude
\begin{equation*}
  \begin{split}
  &\det\left(\frac{\p (\mathbf{x}^{2}_{p^{2},1}, \mathbf{x}^{2}_{p^{2},2}, \tb^{1})}{\p v^{1}}\right)\\
      &= |\tb^{1}|^3\det\left(
                            \begin{array}{c}
                                       -\frac{1}{\mathbf{v}^{1}_{p^{2},3}} \frac{\partial_3 \eta_{p^{2}}}{\sqrt{g_{p^{2},33}}}\Big|_{x^{2}} \\
                                       \frac{1}{\sqrt{g_{p^{2},11}(\mathbf{x}_{p^{2}}^{2})}} \Big[\frac{\partial_1 \eta_{p^{2}}}{\sqrt{g_{p^{2},11}}}\Big|_{x^{2}}- \frac{\mathbf{v}_{p^{1},1}^{1}}{\mathbf{v}_{p^{2},3}^{1}}\frac{\partial_3 \eta_{p^{2}}}{\sqrt{g_{p^{2},33}}}\Big|_{x^{2}} \Big] \\
                                                                              \frac{1}{\sqrt{g_{p^{2},22}(\mathbf{x}_{p^{2}}^{2})}} \Big[\frac{\partial_2 \eta_{p^{2}}}{\sqrt{g_{p^{2},22}}}\Big|_{x^{2}}- \frac{\mathbf{v}_{p^{1},2}^{1}}{\mathbf{v}_{p^{2},3}^{1}}\frac{\partial_3 \eta_{p^{2}}}{\sqrt{g_{p^{2},33}}}\Big|_{x^{2}} \Big] \\
                                     \end{array}
                                   \right) \\
       &=  -|\tb^{1}|^3\frac{1}{\mathbf{v}^{1}_{p^{2},3}}    \frac{1}{\sqrt{g_{p^{2},11}(\mathbf{x}_{p^{2}}^{2})g_{p^{2},22}(\mathbf{x}_{p^{2}}^{2})}} \frac{\partial_3 \eta_{p^{2}}}{\sqrt{g_{p^{2},33}}}\Big|_{x^{2}}   \\
        &\quad \cdot \bigg(\Big[\frac{\partial_1 \eta_{p^{2}}}{\sqrt{g_{p^{2},11}}}\Big|_{x^{2}}- \frac{\mathbf{v}_{p^{1},1}^{1}}{\mathbf{v}_{p^{2},3}^{1}}\frac{\partial_3 \eta_{p^{2}}}{\sqrt{g_{p^{2},33}}}\Big|_{x^{2}} \Big]\times \Big[\frac{\partial_2 \eta_{p^{2}}}{\sqrt{g_{p^{2},22}}}\Big|_{x^{2}}- \frac{\mathbf{v}_{p^{1},2}^{1}}{\mathbf{v}_{p^{2},3}^{1}}\frac{\partial_3 \eta_{p^{2}}}{\sqrt{g_{p^{2},33}}}\Big|_{x^{2}} \Big] \bigg)\\
      & =\frac{1}{\sqrt{g_{p^{2},11}(\mathbf{x}_{p^{2}}^{2})g_{p^{2},22}(\mathbf{x}_{p^{2}}^{2})}}\frac{|\tb^{1}|^3}{\mathbf{v}^{1}_{p^{2},3}}=~\eqref{jac_v_to_xbtb},
   \end{split}
\end{equation*}
where we have used ~\eqref{orthogonal}.

Now we prove the map~\eqref{map_v_to_xbtb} is one-to-one. Assume that there exists $v$ and $\tilde{v}$ satisfy $\xb(x^{1},v)=\xb(x^{1},\tilde{v})$ and $\tb(x^{1},v)=\tb(x^{1},\tilde{v})$. We choose $p\in \partial \Omega$ near $\xb(x^{1},v)$ and use the same parametrization. Then by an expansion, for some ${\color{black}\bar{v}\in \{a\tilde{v}+(1-a)v: a\in [0,1]\},}$
\[0=\left(
      \begin{array}{c}
         \mathbf{x}_{p,1}(x^{1},\tilde{v})    \\
          \mathbf{x}_{p,2}(x^{1},\tilde{v})\\
         \tb(x^{1},\tilde{v})  \\
      \end{array}
    \right)-\left(
              \begin{array}{c}
               \mathbf{x}_{p,1}(x^{1},v)    \\
                \mathbf{x}_{p,2}(x^{1},v)   \\
                \tb(x^{1},v)    \\
              \end{array}
            \right)=\left(
                      \begin{array}{c}
                        \nabla_v \mathbf{x}_{p,1}({\color{black} x^{1}},\bar{v}) \\
                      \nabla_v \mathbf{x}_{p,2}({\color{black} x^{1}},\bar{v})   \\
                  \nabla_v \tb({\color{black} x^{1}},\bar{v})       \\
                      \end{array}
                    \right)(\tilde{v}-v).
\]
This equality can be true only if the determinant of the Jacobian matrix equals zero. Then~\eqref{jac_v_to_xbtb} implies that $\tb(x^{1},\bar{v})=0$. But this implies $x^{1}=\xb(x^{1},\bar{v})$ and hence $n(x^{1})\cdot \bar{v}=0$ which is out of our domain.

\end{proof}

\ \\

\subsection{Preparation for~\eqref{nabla_x_7} and~\eqref{nabla_x_8}.}\label{sec:para_78}

Applying Lemma \ref{lemma:nonlocal_to_local}, we have the following estimate for~\eqref{nabla_x_7} and~\eqref{nabla_x_8} respectively:
\begin{lemma}\label{lemma:est_h}
Recall the definition of $h$ in~\eqref{h}. We have an $L^\infty$ control for $h$ as
\begin{equation}\label{h_bdd}
|w_{\tilde{\theta}}(v)h(t,x,v)|_\infty \lesssim e^{-\lambda t}\big[(\sup_t e^{\lambda t}| wf(t)|_\infty)^2 + \sup_t e^{\lambda t}| wf(t)|_\infty | wf_s|_\infty\big] \lesssim e^{-\lambda t} \sup_t e^{\lambda t}|wf(t)|_\infty. 
\end{equation}
And the derivative along the characteristic is bounded as
\begin{align}
    &\Big|\int^t_{\max\{0,t-\tb\}}e^{-\nu(t-s)} \nabla_x h(s,x-(t-s)v,v) \dd s \Big| \notag\\
    &\lesssim e^{-\lambda t}\Big[\sup_t e^{\lambda t}\Vert wf(t)\Vert_\infty \frac{\Vert w_{\tilde{\theta}}\alpha\nabla_x f_s\Vert_\infty}{w_{\tilde{\theta}}(v)\alpha(x,v)}  +\Vert wf_s\Vert_\infty \frac{\sup_{s\leq t}e^{\lambda s}\Vert w_{\tilde{\theta}}\alpha \nabla_x f(s) \Vert_\infty}{w_{\tilde{\theta}}(v)\alpha(x,v)}   \notag\\
    &  \ \ \ \ \ \ \  + \sup_t e^{\lambda t}\Vert wf(t)\Vert_\infty \frac{\sup_{s\leq t} e^{\lambda s}\Vert w_{\tilde{\theta}}\alpha\nabla_x f(s)\Vert_\infty}{w_{\tilde{\theta}}(v)\alpha(x,v)}\Big] \label{nabla_h_bdd}  \\
    & =\frac{o(1)e^{-\lambda t}}{w_{\tilde{\theta}}(v)\alpha(x,v)}\big[\sup_{s\leq t}e^{\lambda s}\Vert w_{\tilde{\theta}}\alpha\nabla_x f(s)\Vert_\infty + \Vert w_{\tilde{\theta}}\alpha\nabla_x f_s\Vert_\infty\big]\notag   .
\end{align}

\end{lemma}

\begin{proof}
First we prove~\eqref{h_bdd}. From the definition of $h(t,x,v)$ in~\eqref{h}, we apply~\eqref{Gamma bounded} with $w_{\tilde{\theta}}(v)\lesssim \frac{w(v)}{\langle v\rangle}$ to have
\begin{align*}
  |w_{\tilde{\theta}}(v)h(t,x,v)|_\infty  & \lesssim  |wf(t)|_\infty\times |wf_s|_\infty + |wf(t)|_\infty^2 \\
 & \lesssim e^{-\lambda t}\big[\sup_t e^{\lambda t}|wf(t)|_\infty \times |wf_s|_\infty + \sup_t e^{2\lambda t} |wf(t)|_\infty^2 \big].  
\end{align*}
This concludes the first inequality in \eqref{h_bdd}. The second inequality in~\eqref{h_bdd} follows from $|wf_s|_\infty\ll 1$ and $e^{\lambda t}|wf(t)|_\infty \ll 1$ from~\eqref{dym_stability} and~\eqref{infty_bound}.

Then we prove~\eqref{nabla_h_bdd}. First we consider the contribution of $\Gamma(f,f_s)$ and compute
\[\Big|\int^t_{\max\{0,t-\tb\}} e^{-\nu(t-s)}\nabla_x \Gamma(f,f_s)(s,x-(t-s)v,v) \dd s \Big|.\]
We bound $\nabla_x\Gamma(f,f_s)$ by~\eqref{Gamma_est}, the contribution of the first term on RHS of~\eqref{Gamma_est} is bounded as
\begin{align*}
    &   \int^t_{\max\{0,t-\tb\}}\dd s e^{-\nu(t-s)}    \Big[{\color{black}\langle v\rangle}\Vert wf(s)\Vert_\infty |\nabla_x f_s(x,v)|\Big] \\
    &  \leq  \int^t_{\max\{0,t-\tb\}}\dd s e^{-\nu(t-s)} \langle v\rangle\Big[ e^{-\lambda s} \sup_t e^{\lambda t}\Vert wf(t)\Vert_\infty  \frac{\Vert w_{\tilde{\theta}}\alpha\nabla_x f_s\Vert_\infty }{w_{\tilde{\theta}}(v)\alpha(x,v)}\Big] \\
    & \lesssim \frac{e^{-\lambda t}\sup_t e^{\lambda t}\Vert wf(t)\Vert_\infty \Vert w_{\tilde{\theta}}\alpha\nabla_x f_s\Vert_\infty}{w_{\tilde{\theta}}(v)\alpha(x,v)} \int^t_{\max\{0,t-\tb\}} e^{-\nu(t-s)/2} \langle v\rangle\dd s \\
    &\lesssim \frac{e^{-\lambda t}\sup_t e^{\lambda t}\Vert wf(t)\Vert_\infty \Vert w_{\tilde{\theta}}\alpha\nabla_x f_s\Vert_\infty}{w_{\tilde{\theta}}(v)\alpha(x,v)}.
\end{align*}
In the second line we used $\lambda\ll 1$ and $1\lesssim \nu(v)$ from~\eqref{nablav nu} so that $e^{-\nu(t-s)}e^{-\lambda s}\leq e^{-\lambda t}e^{-\nu(t-s)/2}$.

The contribution of the second term of the RHS of~\eqref{Gamma_est} is bounded as
\begin{align*}
    & \int^t_{\max\{0,t-\tb\}} \dd s e^{-\nu(t-s)} e^{-\lambda s}\sup_t e^{\lambda t}\Vert wf(t)\Vert_\infty \int_{\mathbb{R}^3} \dd u \mathbf{k}(v,u)\frac{\Vert w_{\tilde{\theta}}\alpha\nabla_x f_s \Vert_\infty}{w_{\tilde{\theta}}(u)\alpha(x-(t-s)v,u)}   \\
    & \lesssim \frac{e^{-\lambda t}\sup_t e^{\lambda t} \Vert wf(t)\Vert_\infty \Vert w_{\tilde{\theta}}\alpha\nabla_x f_s\Vert_\infty}{ w_{\tilde{\theta}}(v)} \int^t_{\max\{0,t-\tb\}}e^{-\frac{\nu(t-s)}{2}} \int_{\mathbb{R}^3}  \frac{\mathbf{k}(v,u)w_{\tilde{\theta}}(v)}{w_{\tilde{\theta}}(u)\alpha(x-(t-s)v,u)} \dd u \dd s \\
    &\lesssim \frac{e^{-\lambda t}\sup_t e^{\lambda t} \Vert wf(t)\Vert_\infty \Vert w_{\tilde{\theta}}\alpha\nabla_x f_s\Vert_\infty}{w_{\tilde{\theta}}(v)} \int^t_{\max\{0,t-\tb\}}e^{-\frac{\nu(t-s)}{2}} \int_{\mathbb{R}^3}  \frac{\mathbf{k}_{\tilde{\varrho}}(v,u)}{\alpha(x-(t-s)v,u)} \dd u \dd s\\
    &\lesssim \frac{ e^{-\lambda t}\sup_t e^{\lambda t} \Vert wf(t)\Vert_\infty \Vert w_{\tilde{\theta}}\alpha\nabla_x f_s\Vert_\infty }{w_{\tilde{\theta}}(v)\alpha(x,v)}.
\end{align*}
In the third line, we have applied \eqref{k_theta} Lemma \ref{Lemma: k tilde}. In the last line, we have applied Lemma \ref{lemma:nonlocal_to_local}.

By a similar computation, the contribution of the third term of the RHS of~\eqref{Gamma_est} is bounded as
\begin{align*}
    & \int^t_{\max\{0,t-\tb\}} \dd s e^{-\nu(t-s)} e^{-\lambda s}\Vert wf_s\Vert_\infty \int_{\mathbb{R}^3} \dd u \mathbf{k}(v,u)\frac{e^{\lambda s}\Vert w_{\tilde{\theta}}\alpha\nabla_x f(s) \Vert_\infty}{w_{\tilde{\theta}}(u)\alpha(x-(t-s)v,u)}   \\
    & \lesssim \frac{e^{-\lambda t}\Vert wf_s\Vert_\infty \sup_{s\leq t}e^{\lambda s}\Vert w_{\tilde{\theta}}\alpha \nabla_x f(s)\Vert_\infty}{ w_{\tilde{\theta}}(v)} \int^t_{\max\{0,t-\tb\}}e^{-\frac{\nu(t-s)}{2}} \int_{\mathbb{R}^3}  \frac{\mathbf{k}(v,u)w_{\tilde{\theta}}(v)}{w_{\tilde{\theta}}(u)\alpha(x-(t-s)v,u)} \dd u \dd s \\
    &\lesssim \frac{ e^{-\lambda t}\Vert wf_s\Vert_\infty \sup_{s\leq t}e^{\lambda s}\Vert w_{\tilde{\theta}}\alpha \nabla_x f(s)\Vert_\infty }{w_{\tilde{\theta}}(v)\alpha(x,v)}.
\end{align*}

Thus the contribution of $\Gamma(f,f_s)$ in the $h$ of~\eqref{nabla_h_bdd} corresponds to the second line in~\eqref{nabla_h_bdd}.

Then we consider the contribution of $\Gamma(f_s,f)$ and bound
\[\Big|\int^t_{\max\{0,t-\tb\}} e^{-\nu(t-s)}\nabla_x \Gamma(f_s,f)(s,x-(t-s)v,v) \dd s \Big|.\]
Again by~\eqref{Gamma_est}, the contribution of the first term on RHS is bounded as
\begin{align*}
    & \int^t_{\max\{0,t-\tb\}} e^{-\nu(t-s)}\langle v\rangle\Big[\Vert wf_s\Vert_\infty e^{-\lambda s}  \frac{\sup_{s\leq t}e^{\lambda s}\Vert w_{\tilde{\theta}}\alpha\nabla_x f(s)\Vert_\infty}{w_{\tilde{\theta}}(v)\alpha(x,v)} \Big] \dd s \\
    & \lesssim  \frac{e^{-\lambda t}\Vert wf_s\Vert_\infty \sup_{s\leq t}e^{\lambda s}\Vert w_{\tilde{\theta}}\alpha\nabla_x f(s)\Vert_\infty}{w_{\tilde{\theta}}(v)\alpha(x,v)}.
\end{align*}
The contribution of the second and third terms on RHS of~\eqref{Gamma_est} are bounded using the same computation for $\nabla_x \Gamma(f,f_s)$. Thus the contribution of $\Gamma(f_s,f)$ in $h$ of~\eqref{nabla_h_bdd} corresponds to the second line in~\eqref{nabla_h_bdd}.

Last we consider the contribution of $\Gamma(f,f)$ and bound
\[\Big|\int^t_{\max\{0,t-\tb\}} e^{-\nu(t-s)}\nabla_x \Gamma(f,f)(s,x-(t-s)v,v) \dd s \Big|.\]
By~\eqref{Gamma_est}, the contribution of the first term on RHS is bounded as
\begin{align*}
    & \int^t_{\max\{0,t-\tb\}} e^{-\nu(t-s)}e^{-2\lambda s} \langle v\rangle\Big[\sup_t e^{\lambda t}\Vert wf(t)\Vert_\infty   \frac{\sup_{s\leq t}e^{\lambda s}\Vert w_{\tilde{\theta}}\alpha\nabla_x f(s)\Vert_\infty}{w_{\tilde{\theta}}(v)\alpha(x,v)} \Big] \dd s \\
    & \lesssim  \frac{e^{-\lambda t}\sup_t e^{\lambda t}\Vert wf(t)\Vert_\infty \sup_{s\leq t}e^{\lambda s}\Vert w_{\tilde{\theta}}\alpha\nabla_x f(s)\Vert_\infty}{w_{\tilde{\theta}}(v)\alpha(x,v)}.
\end{align*}
The contribution of the second and third terms on RHS of~\eqref{Gamma_est} are bounded as
\begin{align*}
    & \int^t_{\max\{0,t-\tb\}}e^{-\nu(t-s)} e^{-2\lambda s}\sup_t e^{\lambda t}\Vert wf(t)\Vert_\infty \int_{\mathbb{R}^3} \mathbf{k}(v,u)\frac{\sup_{s\leq t}e^{\lambda s}\Vert w_{\tilde{\theta}}\alpha\nabla_x f(s) \Vert_\infty}{w_{\tilde{\theta}}(u)\alpha(x-(t-s)v,u)} \dd u \dd s \\
    &\lesssim e^{-\lambda t}\sup_t e^{\lambda t}\Vert wf(t)\Vert_\infty \sup_{s\leq t}e^{\lambda s}\Vert w_{\tilde{\theta}}\alpha\nabla_x f(s)\Vert_\infty w_{\tilde{\theta}}^{-1}(v) \\
    &\times \int^t_{\max\{0,t-\tb\}}e^{-\nu(t-s)/2} \int_{\mathbb{R}^3}  \frac{\mathbf{k}_{\tilde{\varrho}}(v,u)}{\alpha(x-(t-s)v,u)}\\
    &\lesssim \frac{ e^{-\lambda t}\sup_t e^{\lambda t}\Vert wf(t)\Vert_\infty \sup_{s\leq t}e^{\lambda s}\Vert w_{\tilde{\theta}}\alpha\nabla_x f(s)\Vert_\infty  }{w_{\tilde{\theta}}(v)\alpha(x,v)}.
\end{align*}
Thus the contribution of $\Gamma(f,f)$ in $h$ of~\eqref{nabla_h_bdd} corresponds to the third line of~\eqref{nabla_h_bdd}.

We conclude the first inequality in~\eqref{nabla_h_bdd}. The second inequality follows by\\
$\Vert wf_s\Vert_\infty\ll 1$ and $\sup_t e^{\lambda t}\Vert wf(t)\Vert_\infty\ll 1$ from~\eqref{infty_bound} and~\eqref{dym_stability}.

\end{proof}

\ \\

\subsection{Proof of Theorem \ref{thm:dynamic_regularity}}\label{sec:proof_thm_2}

We prove Theorem \ref{thm:dynamic_regularity} by estimating each term in~\eqref{nabla_x_1} - \eqref{nabla_x_8}.

For fixed and small $\e\ll 1$, we denote
\begin{align}
  \mathcal{B}  &:= \Vert w_{\tilde{\theta}}\alpha\nabla_x f_s\Vert_\infty +  \Vert w_{\tilde{\theta}}\alpha\nabla_x f_0\Vert_\infty  \notag\\
  & \ \ \ +  \sup_t e^{\lambda t}|w\partial_t f(t)|_\infty + \sup_t e^{\lambda t}|wf(t)|_\infty + \e^{-1}\sup_t e^{\lambda t}\Vert wf(t)\Vert_\infty. \label{upp_bdd_B}
\end{align}
Here we note that $\sup_t e^{\lambda t}| w\partial_t f(t)|_\infty$ is bounded in Corollary \ref{corollary:time_derivative}, \\
$\sup_t e^{\lambda t}\Vert wf(t)\Vert_\infty, \ \sup_t e^{\lambda t}|wf(t)|_\infty$ are bounded in~\eqref{dym_stability} and $\Vert w_{\tilde{\theta}}\alpha\nabla_x f_s\Vert_\infty$ is bounded in~\eqref{estF_n}.

We will prove the following estimate:
\begin{equation}\label{estimate_proof}
|\nabla_x f(t,x,v)| \lesssim \frac{e^{-\lambda t}}{w_{\tilde{\theta}}(v)\alpha(x,v)}\times [\mathcal{B}+ o(1)\sup_{s\leq t}e^{\lambda s}\Vert w_{\tilde{\theta}}\alpha\nabla_x f(s)\Vert_\infty].
\end{equation}
We present the proof of the theorem in four steps.

In Step 1, we estimate~\eqref{nabla_x_1}, \eqref{nabla_x_2}, \eqref{nabla_x_3}, \eqref{nabla_x_6}, \eqref{nabla_x_7}, and \eqref{nabla_x_8}.

In Step 2, we estimate~\eqref{nabla_x_4} using the boundary condition~\eqref{f_bc} with the representation~\eqref{eqn: diffuse for f}. We will apply the characteristic~\eqref{nabla_x_1} - \eqref{nabla_x_8}, so that the estimate in Step 1 can be utilized.

In Step 3, we estimate~\eqref{nabla_x_5}. We will apply the characteristic~\eqref{nabla_x_1} - \eqref{nabla_x_8}, and thus the estimates in Step 1 and Step 2 can be utilized.

Finally, in Step 4, we summarize all the estimates and conclude the theorem.

In fact, Theorem \ref{thm:dynamic_regularity} follows directly from~\eqref{estimate_proof} as illustrated in Step 4.

\textit{Step 1: estimate of~\eqref{nabla_x_1}, \eqref{nabla_x_2}, \eqref{nabla_x_3}, \eqref{nabla_x_6}, \eqref{nabla_x_7} and \eqref{nabla_x_8}.}

From $\lambda \ll 1$ and $\nu(v)\gtrsim 1$, \eqref{nabla_x_1}, \eqref{nabla_x_2}, \eqref{nabla_x_3}, \eqref{nabla_x_6} are bounded as
\begin{align}
& |\eqref{nabla_x_1}| + |\eqref{nabla_x_2}| + |\eqref{nabla_x_3}| + |\eqref{nabla_x_6}|  \notag\\
    &\lesssim  e^{-\lambda t}\frac{w^{-1}_{\tilde{\theta}}(v)}{\alpha(x,v)} \Vert w_{\tilde{\theta}}\alpha\nabla_x f_0 \Vert_\infty   + \frac{w^{-1}(v)\nu}{\alpha(x,v)}e^{-\lambda (\tb+t-\tb)} e^{\lambda(t-\tb)}| wf(t-\tb)|_\infty \notag\\
    & + \frac{w^{-1}(v)}{\alpha(x,v)} e^{-\lambda(\tb+t-\tb)} e^{\lambda (t-\tb)} | w\partial_tf(t-\tb)|_\infty+ \frac{w^{-1}_{\tilde{\theta}}(v)}{\alpha(x,v)}e^{-\lambda(\tb+t-\tb)} | wf(t-\tb)|_\infty \notag\\
    &\lesssim {\color{black} \frac{e^{-\lambda t}}{w_{\tilde{\theta}}(v)\alpha(x,v)}\times \mathcal{B}. } \label{1_2_3_6_bdd}
\end{align}
$\mathcal{B}$ is defined in~\eqref{upp_bdd_B}. Here we applied~\eqref{nabla_tbxb} and~\eqref{bdr_alpha} to $\nabla_x \tb$ and~\eqref{Kf_bdd} to~\eqref{nabla_x_6}.

\eqref{nabla_x_7} and~\eqref{nabla_x_8} are bounded using \eqref{h_bdd} and~\eqref{nabla_h_bdd}:
\begin{align}
  & |\eqref{nabla_x_7}| + |\eqref{nabla_x_8}| \notag\\
  & \lesssim  \frac{o(1)e^{-\lambda t}}{w_{\tilde{\theta}}(v)\alpha(x,v)}\big[\sup_{s\leq t}e^{\lambda s}\Vert w_{\tilde{\theta}}\alpha\nabla_x f(s)\Vert_\infty + \Vert w_{\tilde{\theta}}\alpha\nabla_x f_s\Vert_\infty\big]\notag\\
  &+\frac{e^{-\nu \tb}e^{-\lambda (t-\tb)} \sup_t e^{\lambda t} |wf(t)|_\infty}{\alpha(x,v)w_{\tilde{\theta}}(v)}\notag \\
   &   \lesssim {\color{black} \frac{e^{-\lambda t}}{w_{\tilde{\theta}}(v)\alpha(x,v)}\times [\mathcal{B}+o(1)\sup_{s\leq t}e^{\lambda s} \Vert w_{\tilde{\theta}}\alpha\nabla_x f(s)\Vert_\infty]. }\label{7_8_bdd}
\end{align}

\textit{Step 2: estimate of~\eqref{nabla_x_4}.}

\eqref{nabla_x_4} corresponds to the boundary condition~\eqref{f_bc}. We use the representation in~\eqref{eqn: diffuse for f} and apply~\eqref{xi deri xbp} with~\eqref{bdr_alpha} to bound it as
\begin{align}
&|\eqref{nabla_x_4}| \notag \\
& \lesssim \frac{e^{-\nu \tb}|v| \big|\partial_{\mathbf{x}_{p^1,j}^1} M_W(\eta_{p^1}(\bxp),v)\big|}{\sqrt{\mu(v)}\alpha(x,v)} \times \Big|\int_{n(x^1)\cdot v^1>0} f(t^1,x^1,v^1)\sqrt{\mu(v^1)}[n(x^1)\cdot v^1]\dd v^1 \Big| \label{partial_wall}\\
&+\frac{e^{-\nu \tb}|v|M_W(\eta_{p^1}(\mathbf{x}_{p^1}^1),v)}{\sqrt{\mu(v)}\alpha(x,v)}\times
\Big|\int_{\mathbf{v}^{1}_{p^{1},3}>0}  \underbrace{\p_{\mathbf{x}^{1}_{p^{1},j}}
 [f(t^1, \eta_{p^{1}} ( \mathbf{x}_{p^{1} }^{1}  ), T^t_{\mathbf{x}^{1}_{p^{1}}} \mathbf{v}^{1}_{p^{1}}) ]}_{(\ref{fBD_x})_*}\sqrt{\mu(\mathbf{v}^{1}_{p^{1}})}
 \mathbf{v}^{1}_{p^{1},3}
\dd  \mathbf{v}^{1}_{p^{1}}\label{fBD_x}\Big|.
\end{align}
Here we recall that $x^1, v^1$ and $t^1$ are defined in Definition \ref{definition: sto cycle}.

The second term in~\eqref{partial_wall} is directly bounded by
\begin{equation}\label{partial_wall_second}
e^{-\lambda t^1} \sup_t e^{\lambda t}|wf(t)|_\infty=   e^{-\lambda (t-\tb)} \sup_t e^{\lambda t}|wf(t)|_\infty.
\end{equation}

For the first term in \eqref{partial_wall}, we directly take derivative to~\eqref{Wall Maxwellian} to bound it as
\begin{align*}
  \frac{|v| \Big| \partial_{\mathbf{x}_{p^1,j}^1}M_W(\eta_{p^1}(\mathbf{x}_{p^1}^1),v)\Big|}{\sqrt{\mu(v)}}  & \lesssim \Vert \eta\Vert_{C^1}\Vert T_W\Vert_{C^1(\p\O)}[1+|v|]\frac{M_W(x,v)}{\sqrt{\mu(v)}}\lesssim w^{-1}_{\tilde{\theta}}(v) . \end{align*}
Here we used $\tilde{\theta}\ll 1$ and $\Vert T_W-1\Vert_{L^\infty(\p\O)} \ll 1$.

Thus combining with the extra $e^{-\lambda(t-\tb)}$ from \eqref{partial_wall_second}, we bound
\begin{align}
   \eqref{partial_wall} & \lesssim \frac{e^{-\lambda t}}{w_{\tilde{\theta}}(v)\alpha(x,v)}\times \sup_t e^{\lambda t}|wf(t)|_\infty \lesssim \frac{e^{-\lambda t}}{w_{\tilde{\theta}}(v)\alpha(x,v)}\times \mathcal{B}. \label{partial_wall_bdd}
\end{align}

Now we consider (\ref{fBD_x}). First we consider the velocity derivative in $\eqref{fBD_x}_*$ given by
\Be
\begin{split}
& \left(\p_{\mathbf{x}^{1}_{p^{1},j}} T^t_{\mathbf{x}^{1}_{p^{1}}}   \mathbf{v}^{1}_{p^{1}}\right) \cdot
\nabla_v
 f(t^1 ,\eta_{p^{1}} ( \mathbf{x}_{p^{1} }^{1}  ), T^t_{\mathbf{x}^{1}_{p^{1}}} \mathbf{v}^{1}_{p^{1}})\\
= &
  \sum_{l,m} \frac{\p }{ \p{\mathbf{x}^{1}_{p^{1}, j}}  }\left(
  \frac{\p_m \eta_{p^{1},  l} (\mathbf{x}^{1}_{p^{1}})}{\sqrt{g_{p^{1},mm}(\mathbf{x}^{1}_{p^{1}})} }\right)
   \mathbf{v}_{p^{1},m}^{1}\p_{v_l}f(t^1,\eta_{p^{1}} ( \mathbf{x}_{p^{1} }^{1}  ), T^t_{\mathbf{x}^{1}_{p^{1}}} \mathbf{v}^{1}_{p^{1}})\\
=&  \sum_{  m,n}
(\ref{v_under_v_mn})_{mn}
\mathbf{v}_{p^{1},m}^{1} \p_{\mathbf{v}^{1}_{p^{1},n}} [f(t^1,  \eta_{p^{1}} ( \mathbf{x}_{p^{1} }^{1}  ) , T^t_{\mathbf{x}^{1}_{p^{1}}}\mathbf{v}^{1}_{p^{1}})],
\label{v_under_v}
\end{split}\Ee
where
\Be\label{v_under_v_mn}
(\ref{v_under_v_mn})_{mn} :=\sum_l \frac{\p }{ \p{\mathbf{x}^{1}_{p^{1}, j}}  }\left(
  \frac{\p_m \eta_{p^{1},l} (\mathbf{x}^{1}_{p^{1}})}{\sqrt{g_{p^{1},mm}(\mathbf{x}^{1}_{p^{1}})} }\right)
  \frac{\p_n \eta_{p^{1},l} (\mathbf{x}^{1}_{p^{1}})}{\sqrt{g_{p^{1},nn}(\mathbf{x}^{1}_{p^{1}})} } .
\Ee

Then we apply an integration by parts with respect to $\p_{\mathbf{v}_{p^{1}}^{1}}$ to derive that
\Be\begin{split}\label{IBP_v}
& \Big|\int_{\mathbf{v}^{1}_{p^{1},3}>0} f(t^1,\eta_{p^{1}} ( \mathbf{x}_{p^{1}}^{1}  ), T^t_{\mathbf{x}^{1}_{p^{1}}} \mathbf{v}^{1}_{p^{1}})
\sum_{m,n}   (\ref{v_under_v_mn})_{mn}
\p_{\mathbf{v}^{1}_{p^{1},n}}
\big[\mathbf{v}^{1}_{p^{1}, m}  \mathbf{v}^{1}_{p^{1},3}\sqrt{\mu(\mathbf{v}^{1}_{p^{1}})}\big] \dd \mathbf{v}^{1}_{p^{1}} \Big| \\
 & \lesssim  e^{-\lambda (t-\tb)} e^{\lambda (t-\tb)} | wf(t-\tb)|_\infty.
\end{split}\Ee
Here we notice that 
\[f(t^1,\eta_{p^{1}} ( \mathbf{x}_{p^{1} }^{1}  ), T^t_{\mathbf{x}^{1}_{p^{1}}} \mathbf{v}^{1}_{p^{1}})
\sum_{m,n}   (\ref{v_under_v_mn})_{mn}\mathbf{v}^{1}_{p^{1}, m}  \mathbf{v}^{1}_{p^{1},3} \sqrt{\mu(\mathbf{v}^{1}_{p^{1}})}\equiv 0\] 
when $\mathbf{v}^{1}_{p^{1},3}=0$ for ${\color{black}| wf(t-\tb)|_\infty  }<\infty$.

We conclude that the contribution of~\eqref{v_under_v} in~\eqref{fBD_x}, i.e, the partial derivative with respect to the velocity variable, is bounded by
\begin{align}
    &\frac{e^{-\lambda t}w^{-1}_{\tilde{\theta}}(v)}{\alpha(x,v)}\sup_t e^{\lambda t}| wf(t)|_\infty \leq \frac{e^{-\lambda t}}{w_{\tilde{\theta}}(v)\alpha(x,v)}\times \mathcal{B}.\label{bdr_deri_v_bdd}
\end{align}

Next we consider the spatial derivative in $\eqref{fBD_x}_*$, which is given by 
\begin{align}
    &  \int_{n(x^1)\cdot v^1 > 0} \partial_{\mathbf{x}_{p^1,j}^1}f(t^1, \eta_{p^1}(\mathbf{x}_{p^1}^1),v^1) \sqrt{\mu(v^1)} [n(x^1)\cdot v^1] \dd v^1. \label{bdr_deri_x}
\end{align}
Here we rewrite the velocity variable using the Cartesian coordinate $v^1$.

Note that from~\eqref{fBD_x1}, we have $\partial_{\xpj} f(t^1,\eta_{p^1}(\bxp),v^1) = \partial_j \eta_{p^1}(\bxp) \nabla_x f(t^1,\eta_{p^1}(\bxp),v^1)$. Since $n(x^1)\cdot v^1>0$, we can expand 
$\nabla_x f(t^1,\eta_{p^1}(\bxp),v^1)$ along the characteristic using~\eqref{nabla_x_1}-\eqref{nabla_x_8}. 

The contribution of \eqref{nabla_x_1}, \eqref{nabla_x_2}, \eqref{nabla_x_3}, \eqref{nabla_x_6}, \eqref{nabla_x_7} and~\eqref{nabla_x_8} in the expansion of $\nabla_x f(t^1,\eta_{p^1}(\bxp),v^1)$ can be computed similarly as~\eqref{1_2_3_6_bdd} and~\eqref{7_8_bdd}, where we just need to replace $t$ by $t^1$ and $(x,v)$ by $(x^1,v^1)$. Thus the contribution of these six terms in~\eqref{bdr_deri_x} are bounded by
\begin{align}
    & \int_{n(x^1)\cdot v^1>0} \frac{e^{-\lambda (t-\tb)}}{\alpha(x^1,v^1)}[n(x^1)\cdot v^1]\sqrt{\mu(v^1)}[\mathcal{B} + o(1)\sup_{s\leq t}e^{\lambda s}\Vert w_{\tilde{\theta}}\alpha\nabla_x f(s)\Vert_\infty] \dd v^1 \notag\\
    & \lesssim e^{-\lambda(t-\tb)} [\mathcal{B} + o(1)\sup_{s\leq t}e^{\lambda s}\Vert w_{\tilde{\theta}}\alpha\nabla_x f(s)\Vert_\infty], \label{bdd_deri_x_1236}
\end{align}
where we have used \eqref{kinetic_distance} and \eqref{bdr_alpha} to cancel the singularity of $\frac{1}{\alpha(x^1,v^1)}$ as follows:
\begin{align}
  &  \frac{|n(x^1)\cdot v^1|}{\alpha(x^1,v^1)} \sqrt{\mu(v^1)}  = \mathbf{1}_{\alpha(x^1,v^1) \geq \frac{1}{2}} + \mathbf{1}_{\alpha(x^1,v^1) < \frac{1}{2}} \notag\\
    & \leq 2|n(x^1)\cdot v^1| \sqrt{\mu(v^1)} + \frac{|n(x^1)\cdot v^1|}{\tilde{\alpha}(x^1,v^1)}  \lesssim [|n(x^1)\cdot v^1| + 1]\sqrt{\mu(v^1)}. \label{singlarity_alpha}
\end{align}

Next we consider the contribution of~\eqref{nabla_x_5} in~\eqref{bdr_deri_x}, which reads
\begin{align}
    & \Big|\int_{n(x^1)\cdot v^1>0} \dd v^1 \sqrt{\mu(v^1)}[n(x^1)\cdot v^1]\partial_j \eta_{p^1}(\bxp) \notag\\
    &\times  \int_{\max\{0,t^1-\tb(x^1,v^1)\}}^{t^1} \dd s e^{-\nu(v^1)(t^1-s)}\int_{\mathbb{R}^3} \dd u\mathbf{k}(v^1,u) \nabla_x f(s,x^1-(t^1-s)v^1,u)  \Big|. \label{bdr_x_K}
\end{align}

{\color{black}

We separate the $\dd s$ integral into $t^1-s<\e$ and $t^1-s\geq \e$. When $t^1-s<\e$, we denote the above term as
\begin{align}
  \eqref{bdr_x_K}_1  &:= \Big|\int_{n(x^1)\cdot v^1>0} \dd v^1 \sqrt{\mu(v^1)}[n(x^1)\cdot v^1]       \partial_j \eta_{p^1}(\bxp)  \notag\\
  & \times \int_{\max\{0,t^1-\tb(x^1,v^1),t^1-\e\}}^{t^1} \dd s e^{-\nu(v^1)(t^1-s)}\int_{\mathbb{R}^3} \dd u\mathbf{k}(v^1,u) \nabla_x f(s,x^1-(t^1-s)v^1,u)  \Big|\notag\\
   & \lesssim \int_{n(x^1)\cdot v^1>0} \sqrt{\mu(v^1)}[n(x^1)\cdot v^1] \int_{\max\{t^1-\tb(x^1,v^1),t^1-\e\}}^{t^1} \dd s e^{-\nu(v^1)(t^1-s)} \notag\\
    &\times \int_{\mathbb{R}^3} \dd u \frac{\mathbf{k}(v^1,u)e^{-\lambda s}}{\alpha(x^1-(t^1-s)v^1,u)} \sup_{s\leq t} e^{\lambda s}\Vert w_{\tilde{\theta}}\alpha\nabla_x f(s)\Vert_\infty. \label{bdd_bdr_x_K_1}
\end{align}
Note that from~\eqref{nablav nu}, we have $\nu(v^1) \thicksim \sqrt{|v^1|^2+1}$, together with $\lambda \ll 1$, we derive \\
$e^{-\nu(v^1)(t^1-s)}e^{-\lambda s^1}\leq e^{-\nu(v^1)(t^1-s)/2} e^{-\lambda t^1}$. We further bound
\begin{align}
   \eqref{bdd_bdr_x_K_1} & \lesssim e^{-\lambda t^1} \sup_{s\leq t} e^{\lambda s}\Vert w_{\tilde{\theta}}\alpha\nabla_x f(s)\Vert_\infty \int_{n(x^1)\cdot v^1>0 }  \dd v^1 \sqrt{\mu(v^1)}[n(x^1)\cdot v^1] \notag\\
   &\times \int_{\max\{t^1-\e,t^1-\tb(x^1,v^1)\}}^{t^1} \dd s e^{-\nu(v^1)(t^1-s)/2} \int_{\mathbb{R}^3} \dd u \frac{\mathbf{k}(v^1,u)}{\alpha(x^1-(t^1-s)v^1,u)} \notag\\
   &\lesssim o(1) e^{-\lambda t^1} \sup_{s\leq t} e^{\lambda s}\Vert w_{\tilde{\theta}}\alpha\nabla_x f(s)\Vert_\infty \int_{n(x^1)\cdot v^1>0} \sqrt{\mu(v^1)}\frac{[n(x^1)\cdot v^1]}{\alpha(x^1,v^1)}\dd v^1 \notag\\
   &\lesssim o(1) e^{-\lambda t^1} \sup_{s\leq t} e^{\lambda s}\Vert w_{\tilde{\theta}}\alpha\nabla_x f(s)\Vert_\infty. \label{bdr_x_K_1_bdd}
\end{align}
Here we have applied Lemma \ref{lemma:nonlocal_to_local} to get the third line, and we applied \eqref{singlarity_alpha} in the last line.

When $t^1-s\geq \e$, we further split the $v^1$ integration into $n(x^1)\cdot v^1<\delta$ and $n(x^1)\cdot v^1 > \delta$. For $n(x^1)\cdot v^1<\delta$, we have
\begin{align}
  &  \Big|\int_{n(x^1)\cdot v^1<\delta} \dd v^1 \sqrt{\mu(v^1)}[n(x^1)\cdot v^1]       \partial_j \eta_{p^1}(\bxp)  \notag\\
  & \times \int_{\max\{0,t^1-\tb(x^1,v^1)\}}^{t^1-\e} \dd s e^{-\nu(v^1)(t^1-s)}\int_{\mathbb{R}^3} \dd u\mathbf{k}(v^1,u) \nabla_x f(s,x^1-(t^1-s)v^1,u)  \Big| \notag\\
  &\lesssim e^{-\lambda t^1} \sup_{s\leq t} e^{\lambda s}\Vert w_{\tilde{\theta}}\alpha\nabla_x f(s)\Vert_\infty \int_{n(x^1)\cdot v^1 < \delta }  \dd v^1 \sqrt{\mu(v^1)}[n(x^1)\cdot v^1] \notag\\
   &\times \int_{\max\{t^1-\tb(x^1,v^1)\}}^{t^1-\e} \dd s e^{-\nu(v^1)(t^1-s)/2} \int_{\mathbb{R}^3} \dd u \frac{\mathbf{k}(v^1,u)}{\alpha(x^1-(t^1-s)v^1,u)} \notag\\
   & \lesssim e^{-\lambda t^1} \sup_{s\leq t} e^{\lambda s}\Vert w_{\tilde{\theta}}\alpha\nabla_x f(s)\Vert_\infty \int_{n(x^1)\cdot v^1<\delta} \sqrt{\mu(v^1)}\frac{[n(x^1)\cdot v^1]}{\alpha(x^1,v^1)}\dd v^1 \notag \\
   &\lesssim o(1) e^{-\lambda t^1} \sup_{s\leq t} e^{\lambda s}\Vert w_{\tilde{\theta}}\alpha\nabla_x f(s)\Vert_\infty.   \label{bdr_x_K_2_1_bdd}
\end{align}
Here we applied Lemma \ref{lemma:nonlocal_to_local} in the second last line.

For $n(x^1)\cdot v^1 > \delta$ and $t^1-s\geq \e$, we use the following change of variable 
\begin{equation}\label{cov_v1}
\nabla_x f(s,x^1-(t^1-s)v^1,u) = -\frac{\nabla_{v^1}[f(s,x^1-(t^1-s)v^1,u)]}{t^1-s}.  
\end{equation}
In such case, without loss of generality, we assume $t^1>\e$, otherwise, we directly bound \eqref{bdr_x_K} by \eqref{bdr_x_K_1_bdd}. Then in \eqref{bdr_x_K}, we only integrate over the $v^1$ such that $\tb(x^1,v^1)>\e$. Then we denote 
\begin{align*}
    V: = \{v^1\in \mathbb{R}^3: n(x^1)\cdot v^1 > \delta, \ \tb(x^1,v^1)>\e\}.
\end{align*}
We also denote the corresponding hypersurface as $\tilde{V}$. Since $n(x^1)\cdot v^1>\delta$, by \eqref{nabla_tbxb}, $\tilde{V}$ is piecewise smooth.

In such case, we denote \eqref{bdr_x_K} as
\begin{align*}
    \eqref{bdr_x_K}_2 & :=  \Big|\int_{V} \dd v^1 \sqrt{\mu(v^1)}[n(x^1)\cdot v^1]       \partial_j \eta_{p^1}(\bxp)  \notag\\
  & \times \int_{\max\{0,t^1-\tb(x^1,v^1)\}}^{t^1-\e} \dd s e^{-\nu(v^1)(t^1-s)}\int_{\mathbb{R}^3} \dd u\mathbf{k}(v^1,u) \frac{-\nabla_{v^1} f(s,x^1-(t^1-s)v^1,u)}{t^1-s}  \Big|\notag.
\end{align*}
Here we applied \eqref{cov_v1}.

Then we apply an integration by part in $v^1$ to have
\begin{equation}\label{bdr_ibp}
\begin{split}
 & \eqref{bdr_x_K}_2 \\
 & = \Big|\int_{V} \nabla_{v^1}[\sqrt{\mu(v^1)}[n(x^1)\cdot v^1]]\partial_j \eta_{p^1}(\bxp) \\
  &\times  \int^{t^1-\e}_{\max\{0,t^1-\tb(x^1,v^1)\}} \dd s e^{-\nu(v^1)(t^1-s)} \int_{\mathbb{R}^3} \dd u\mathbf{k}(v^1,u)  \frac{f(s,x^1-(t^1-s)v^1,u)}{t^1-s}    \\
    & + \int_{V} \sqrt{\mu(v^1)}[n(x^1)\cdot v^1]\partial_j \eta_{p^1}(\bxp)\\
    &\times  \int^{t^1-\e}_{\max\{0,t^1-\tb(x^1,v^1)\}} \dd s  \int_{\mathbb{R}^3} \dd u \nabla_{v^1}[ e^{-\nu(v^1)(t^1-s)}\mathbf{k}(v^1,u)]  \frac{f(s,x^1-(t^1-s)v^1,u)}{t^1-s}  \\
    & +\int_{V} \sqrt{\mu(v^1)}[n(x^1)\cdot v^1]\partial_j \eta_{p^1}(\bxp) \mathbf{1}_{t^1-\tb(x^1,v^1)>0}\\
    &\times \nabla_{v^1} \tb(x^1,v^1) e^{-\nu(v^1)\tb(x^1,v^1)}\int_{\mathbb{R}^3} \dd u \mathbf{k}(v^1,u) \frac{f(t^1-\tb(x^1,v^1),\xb(x^1,v^1),u)}{\tb(x^1,v^1)} \\
    & - \int_{\tilde{V}} \sqrt{\mu(v^1)}[n(x^1)\cdot v^1] \p_j \eta_{p^1}(\bxp) \mathbf{n} \\
    &\times \int^{t^1-\e}_{\max\{0,t^1-\tb(x^1,v^1)\}} \dd s  \int_{\mathbb{R}^3} \dd u  e^{-\nu(v^1)(t^1-s)}\mathbf{k}(v^1,u)  \frac{f(s,x^1-(t^1-s)v^1,u)}{t^1-s}\Big| .    
\end{split}
\end{equation}
In the last term $\mathbf{n} = \frac{-\nabla_{v^1} \tb(x^1,v^1)}{|\nabla_{v^1} \tb(x^1,v^1)|}$. Then the last term is bounded using $\Vert w f(s)\Vert_{\infty}<\infty$ by
\begin{align}
    & \e^{-1} e^{-\lambda t^1} \sup_t e^{\lambda t}\Vert wf(t)\Vert_\infty . \label{bdd_last}
\end{align}
Here we applied $e^{-\nu(v^1)(t^1-s)}e^{-\lambda s^1}\leq e^{-\nu(v^1)(t^1-s)/2} e^{-\lambda t^1}$, $\mathbf{k}(v^1,u)\in L^1_u$ and $t^1-s\geq \e$.}

Again by the three conditions above, the first term on the RHS of \eqref{bdr_ibp} is bounded as 
\begin{align}
    & \e^{-1} e^{-\lambda t^1} \sup_t e^{\lambda t}\Vert wf(t)\Vert_\infty . \label{bdd_first_two}
\end{align}

{\color{black}
For the second term in \eqref{bdr_ibp}, we use~\eqref{nabla_k_theta} with $\tilde{\theta} \ll \theta$, and $|v^1|^2 w^{-1}_{\tilde{\theta}}(v^1)\lesssim 1$ to have 
\begin{align*}
 |\nabla_{v^1}\mathbf{k}(v^1,u)| |f(s,u)|  &   \leq \Vert wf(s)\Vert_\infty  w^{-1}_{\tilde{\theta}}(v^1)|\nabla_{v^1} \mathbf{k}(v^1,u)| \frac{w_{\tilde{\theta}}(v^1)}{w_{\tilde{\theta}}(u)}      \\
 & \lesssim  \Vert wf(s)\Vert_\infty  w^{-1}_{\tilde{\theta}}(v^1) [1+|v^1|^2 ] \frac{\mathbf{k}_{\tilde{\varrho}}(v^1,u)}{|v^1-u|}\\
 &\lesssim e^{-\lambda s} \sup_t e^{\lambda t}\Vert wf(t)\Vert_\infty  \frac{\mathbf{k}_{\tilde{\varrho}}(v^1,u)}{|v^1-u|} \in L^1_u.
\end{align*}

Thus the contribution of $\nabla_{v^1} \mathbf{k}(v^1,u)$ in the second term on the RHS of \eqref{bdr_ibp} is bounded as
\begin{align}
    & \e^{-1} \int_{n(x^1)\cdot v^1>0}  \nu(v^1)\sqrt{\mu(v^1)} [n(x^1)\cdot v^1] \int^{t^1-\e}_{\max\{0,t^1-\tb(x^1,v^1)\}} \dd s e^{-\nu(v^1)(t^1-s)} e^{-\lambda t^1} \notag\\
    & \lesssim \e^{-1}e^{-\lambda t^1} \sup_t e^{\lambda t}\Vert wf(t)\Vert_\infty . \label{bdd_two_two}
\end{align}}
The contribution of $\frac{|\nabla_{v^1} e^{-\nu(v^1)(t^1-s)}|}{t^1-s}\lesssim e^{-\nu(v^1)(t^1-s)}$ in the second term on the RHS of \eqref{bdr_ibp} can be bounded by \eqref{bdd_two_two} by the same computation. Here we have used \eqref{nablav nu}.

For the third term we apply~\eqref{nabla_tbxb} to have $\big| \frac{\nabla_{v^1}\tb(x^1,v^1)}{\tb(x^1,v^1)} \big| \lesssim \frac{1}{\alpha(x^1,v^1)}$, and thus the third term in \eqref{bdr_ibp} is bounded by
\begin{align}
    & e^{-\lambda t^1} \sup_t e^{\lambda t}| wf(t)|_\infty.\label{bdd_third}
\end{align}

Collecting~\eqref{bdr_x_K_1_bdd} and \eqref{bdd_last} -- \eqref{bdd_third} we conclude that the contribution of~\eqref{nabla_x_5} in~\eqref{bdr_deri_x} is bounded by
\begin{align}
    & e^{-\lambda t^1}\big[\mathcal{B}+ o(1)\sup_{s\leq t}e^{\lambda s}\Vert w_{\tilde{\theta}}\alpha\nabla_x f(s)\Vert_\infty \big] . \label{bdd_deri_x_K}
\end{align}
Here we recall $\mathcal{B}$ is defined in~\eqref{upp_bdd_B}.

{\color{black} Next, we consider the contribution of \eqref{nabla_x_4} in \eqref{bdr_deri_x}. We use the boundary condition \eqref{f_bc} with the representation \eqref{eqn: diffuse for f} at $x^2$, then such contribution reads
\begin{align}
    & \int_{n(x^1)\cdot v^1>0}  \sum_{j'=1,2} \frac{\p \mathbf{x}_{p^2,j'}^2}{ \p \mathbf{x}_{p^1,j}^1} e^{-\nu(v^1)\tb^1}  \sqrt{\mu(v^1)} [n(x^1)\cdot v^1] \dd v^1 \notag\\
    & \times \Big[\frac{\p_{\mathbf{x}_{p^2,j'}^2} M_W(\eta_{p^2}(\mathbf{x}_{p^2}^2),v^1)}{\sqrt{\mu(v^1)}} \int_{n(x^2)\cdot v^2>0} f(t^2,x^2,v^2) \sqrt{\mu(v^2)}[n(x^2)\cdot v^2] \dd v^2   \label{bdr_bdr_1}\\
    &+ \frac{M_W(\eta_{p^2}(\mathbf{x}_{p^2}^2),v^1)}{\sqrt{\mu(v^1)}} \int_{\mathbf{v}_{p^2,3}^2>0}  \p_{\mathbf{x}_{p^2,j'}^2} [f(t^2,\eta_{p^2}(\mathbf{x}_{p^2}^2),T^t_{\mathbf{x}_{p^2}^2}\mathbf{v}_{p^2}^2)] \sqrt{\mu(\mathbf{v}_{p^2}^2)} \mathbf{v}^2_{p^2,3} \dd \mathbf{v}_{p^2}^2   \Big] .  \label{bdr_bdr_2}
\end{align}
Here we have applied \eqref{fBD_x1} to have $\p_j \eta_{p^1}(\mathbf{x}_{p^1}^1) \nabla_x \mathbf{x}_{p^2,j'}^2 = \frac{\p \mathbf{x}_{p^2,j'}^2}{ \p \mathbf{x}_{p^1,j}^1}$.

From \eqref{xip deri xbp}, the first term \eqref{bdr_bdr_1} is directly bounded as
\begin{align*}
   \eqref{bdr_bdr_1} & \lesssim \Vert \eta\Vert_{C^1(\p\O)} \sup_t e^{\lambda t} | wf(t)|_\infty e^{-\lambda \tb^1} e^{-\lambda (t^1-\tb^1)} = e^{-\lambda t^1}\sup_t e^{\lambda t}| wf(t)|_\infty,
\end{align*}
where we have used $t^2 = t^1-\tb^1$.

For the second term, we apply the change of variable in Lemma \ref{Lemma: change of variable} to have
\begin{align*}
    \eqref{bdr_bdr_2}& = \sum_{p^2\in \mathcal{P}} \iint_{|\mathbf{x}_{p^2}^2|<\delta_1} \dd \mathbf{x}_{p^2,1}^2 \mathbf{x}_{p^2,2}^2  \int_0^{t^1} \dd \tb^1 \sqrt{g_{p^2,11} g_{p^2,22}} e^{-\nu(v^1)\tb^1} \iota_{p^2}(\eta_{p^2}(\mathbf{x}_{p^2}^2)) \sum_{j^\prime=1,2}
\frac{\p \mathbf{x}^{2}_{p^{2},j^\prime}}{\p{\mathbf{x}^{1  }_{p^{1 },j}}} \\
    &\times\frac{n_{p^{1}} (\mathbf{x}_{p^{1}}^{1}) \cdot  (x^{1} -
 \eta_{p^{2}} (\mathbf{x}_{p^{2}}^{2})
 )
}{\tb^{1}}
 \frac{ n_{p^{2}}(\mathbf{x}^{2}_{p^{2}}) \cdot (x^{1} -
 \eta_{p^{2}} (\mathbf{x}_{p^{2}}^{2})
 ) }{|\tb^{1}|^4} \notag \\
 & \times M_W(\eta_{p^2}(\mathbf{x}_{p^2}^2),v^1) \int_{\mathbf{v}_{p^2,3}^2>0}  \p_{\mathbf{x}_{p^2,j'}^2} [f(t^1-\tb^1,\eta_{p^2}(\mathbf{x}_{p^2}^2),T^t_{\mathbf{x}_{p^2}^2}\mathbf{v}_{p^2}^2)] \sqrt{\mu(\mathbf{v}_{p^2}^2)} \mathbf{v}^2_{p^2,3} \dd \mathbf{v}_{p^2}^2.
\end{align*}
We apply the integration by parts with respect to $\p_{\mathbf{x}^{2}_{p^{2}, j^\prime}}$ for $j^{\prime}=1,2$. Then we derive
\begin{align}
    \eqref{bdr_bdr_2}& = \sum_{p^{2} \in \mathcal{P}}\iint 
\int^{t^{1}}_0 \int_{\mathbf{v}_{p^2,3}^2>0}  f(t^1-\tb^1,\eta_{p^2}(\mathbf{x}_{p^2}^2),T^t_{\mathbf{x}_{p^2}^2}\mathbf{v}_{p^2}^2)\sqrt{\mu(\mathbf{v}_{p^2}^2)} \mathbf{v}^2_{p^2,3} \dd \mathbf{v}_{p^2}^2 \notag \\
&\times \sum_{j'=1,2}\Big\{
 \p_{\mathbf{x}_{p^{2},j^\prime}^{2}} \Big[
e^{- \nu(v^{1}) \tb^{1}}
\iota_{p^{2}} (
\eta_{p^{2}} (\mathbf{x}_{p^{2}}^{2} )
) M_W(\eta_{p^2}(\mathbf{x}_{p^2}^2),v^1)\Big] \cdots  \label{p_xf_total1_under1} \\
& +  \p_{\mathbf{x}_{p^{2},j^\prime}^{2}} \Big[
\frac{\p \mathbf{x}^{2}_{p^{2},j^\prime}}{\p{\mathbf{x}^{1  }_{p^{1 },j}}}
\sqrt{g_{p^{2},11}g_{p^{2},22}  } \Big] \cdots \label{p_xf_total1_under2} \\
& + \p_{\mathbf{x}_{p^{2},j^\prime}^{2}} \Big[
\frac{n_{p^{1}} (\mathbf{x}_{p^{1}}^{1}) \cdot  (x^{1} -
 \eta_{p^{2}} (\mathbf{x}_{p^{2}}^{2})
 )
}{\tb^{1}} \cdot
 \frac{ n_{p^{2}}(\mathbf{x}^{2}_{p^{2}}) \cdot (x^{1} -
 \eta_{p^{2}} (\mathbf{x}_{p^{2}}^{2})
 ) }{|\tb^{1}|^4}\Big]\cdots  \Big\}. \label{p_xf_total1_under3}
\end{align}
Here we note that for $\iota_{p^{2} } (
\eta_{p^{2}} (\mathbf{x}_{p^{2}}^{2} )
)=0$ when $|\mathbf{x}_{p^{2}}^{2}|= \delta_1$ from (\ref{iota}), such contribution of $|\mathbf{x}_{p^{2}}^{2}|= \delta_1$ vanishes.

}

From~\eqref{nv<v2} and~\eqref{bound_vb_x}, we bounds terms in~\eqref{p_xf_total1_under1} - \eqref{p_xf_total1_under3} as follows:
\begin{align*}
    &|n(x^1)\cdot v^1| = \frac{|n(x^1)\cdot (x^2-x^1)|}{\tb^1}\lesssim \frac{|x^1-\eta_{p^2}(\mathbf{x}_{p^2}^2)|^2}{\tb^1} , \\
    & \Big| \frac{n_{p^{1}} (\mathbf{x}_{p^{1}}^{1}) \cdot  (x^{1} -
 \eta_{p^{2}} (\mathbf{x}_{p^{2}}^{2})
 )
}{\tb^{1}}
 \frac{ n_{p^{2}}(\mathbf{x}^{2}_{p^{2}}) \cdot (x^{1} -
 \eta_{p^{2}} (\mathbf{x}_{p^{2}}^{2})
 ) }{|\tb^{1}|^4} \Big|\lesssim \frac{|x^1-\eta_{p^2}(\mathbf{x}_{p^2}^2)|^4}{|\tb^1|^5},
\end{align*}

\[\Big| \frac{\partial}{\partial_{\mathbf{x}_{p^2,j'}^2} } [n(x^1)\cdot v^1]\Big| \lesssim \frac{|x^1-\eta_{p^2}(\mathbf{x}_{p^2}^2)|}{\tb^1}, \ \ \Big|\frac{\partial}{\partial_{\mathbf{x}_{p^2,j'}^2} } [n(x^2)\cdot (x^1-\eta_{p^2}(\mathbf{x}_{p^2}^2))] \Big| \lesssim |x^1-\eta_{p^2}(\mathbf{x}_{p^2}^2)|,\]

\begin{align*}
    & \Big|\frac{\partial}{\partial_{\mathbf{x}_{p^2,j'}^2} } v^1 \Big| = \Big|\frac{\partial}{\partial_{\mathbf{x}_{p^2,j'}^2} } [x^1-\eta_{p^2}(\mathbf{x}_{p^2}^2)] \Big|/\tb^1 \lesssim  \frac{\Vert \eta\Vert_{C^1}}{\tb^1} , \\
    &\Big|\frac{\partial}{\partial_{\mathbf{x}_{p^2,j'}^2} } e^{-\nu(v^1)\tb^1} \Big| \lesssim \frac{\Vert \eta\Vert_{C^1} e^{-\nu(v^1)\tb^1}}{\tb^1}, \\
    &\Big|\frac{\partial}{\partial_{\mathbf{x}_{p^2,j'}^2} } [M_W(\eta_{p^2}(\mathbf{x}_{p^2}^2),v^1)] \Big| \lesssim \Vert \eta \Vert_{C^1}[1+|v^1|^2]M_W(\eta_{p^2}(\mathbf{x}_{p^2}^2),v^1) + |v^1|\Big|\frac{\partial}{\partial_{\mathbf{x}_{p^2,j'}^2} } v^1 \Big|M_W(\eta_{p^2}(\mathbf{x}_{p^2}^2),v^1) \\
    &\lesssim \Vert \eta \Vert_{C^1}\Big[1 + \frac{1}{\tb^1} \Big] e^{-\frac{|v^1|^2}{4}}\lesssim \Big[1 + \frac{1}{\tb^1} \Big] e^{-\frac{|x^1 - \eta_{p^2}(\mathbf{x}_{p^2}^2)|^2}{4|\tb^1|^2}},
\end{align*}
\Be\begin{split}\label{px_pxx}
&\bigg| \frac{\p }{\p {\mathbf{x}_{p^{2},j^\prime}^{2}}} \bigg(
\sum_{j^\prime=1,2}
\frac{\p \mathbf{x}^{2}_{p^{2},j^\prime}}{\p{\mathbf{x}^{1  }_{p^{1 },j}}}
\sqrt{g_{p^{2},11}g_{p^{2},22}  }
\bigg) \bigg|\lesssim   \| \eta \|_{C^2} \Big\{1+
\frac{|\mathbf{v}^{2}_{p^{2} ,\parallel}|}{|\mathbf{v}^{2}_{p^{2}, 3}|^2  } | \p_3 \eta_{p^{2}}
(
\mathbf{x}^{2  }_{p^{2 }}
)
\cdot \p_j \eta_{p^{1}} (
\mathbf{x}^{1  }_{p^{1 }}
)|
\Big\}\\
  \lesssim & 
   \Big\{1+  \frac{|\mathbf{v}^{2}_{p^{2}  }|}{|\mathbf{v}^{2}_{p^{2}, 3}|^2  }
   |x^{1} - \eta_{p^{2}}
(
\mathbf{x}^{2  }_{p^{2 }}
) |
   \Big\}
   \lesssim  \frac{1}{|\mathbf{v}^{2}_{p^{2}, 3}|}=\frac{|\tb^{1}|}{|n_{p^{2}}(\mathbf{x}_{p^{2}}^{2})  \cdot  (x^{1} -
 \eta_{p^{2}} (\mathbf{x}_{p^{2}} ^{2})|} \lesssim   \frac{|\tb^1|}{|x^1-\eta_{p^2}(\mathbf{x}_{p^2}^2)|^2}.
\end{split}
\Ee
For the last estimate \eqref{px_pxx}, we also used~\eqref{xip deri xbp} and \eqref{tang*nor}.

Thus the integrand of~\eqref{p_xf_total1_under1} is bounded as
\begin{align*}
&\Big[1+ \frac{1}{\tb^1} \Big] |v^1|e^{-\frac{|x^1-\eta_{p^2}(\mathbf{x}_{p^2}^2) |^2}{4|\tb^1|^2}} \frac{|x^1-\eta_{p^2}(\mathbf{x}_{p^2}^2)|^4}{|\tb^1|^5} \notag \\
&\lesssim \Big[ \frac{|x^1-\eta_{p^2}(\mathbf{x}_{p^2}^2)|^4}{|\tb^1|^6}+\frac{|x^1-\eta_{p^2}(\mathbf{x}_{p^2}^2)|^4}{|\tb^1|^5}\Big]  e^{-\frac{\big|x^1-\eta_{p^2}(\mathbf{x}_{p^2}^2) \big|^2}{4|\tb^1|^2}},
\end{align*}
the integrand of~\eqref{p_xf_total1_under2} is bounded as
\begin{align*}
    &   \frac{|x^1-\eta_{p^2}(\mathbf{x}_{p^2}^2)|^4}{|\tb^1|^5}\frac{|\tb^1|}{|x^1-\eta_{p^2}(\mathbf{x}_{p^2}^2)|^2} e^{-\frac{\big|x^1-\eta_{p^2}(\mathbf{x}_{p^2}^2) \big|^2}{4|\tb^1|^2}} \lesssim \frac{|x^1-\eta_{p^2}(\mathbf{x}_{p^2}^2)|^2}{|\tb^1|^4}e^{-\frac{\big|x^1-\eta_{p^2}(\mathbf{x}_{p^2}^2) \big|^2}{4|\tb^1|^2}} ,
\end{align*}
the integrand of~\eqref{p_xf_total1_under3} is bounded as
\begin{align*}
    &   \frac{|x^1-\eta_{p^2}(\mathbf{x}_{p^2}^2)|^3}{|\tb^1|^5} e^{-\frac{\big|x^1-\eta_{p^2}(\mathbf{x}_{p^2}^2) \big|^2}{4|\tb^1|^2}}.
\end{align*}

With $e^{-\nu(v^1)\tb^1}e^{-\lambda(t^1-\tb^1)} \leq e^{-\nu(v^1)\tb^1/2}e^{-\lambda t^1}$, we derive
\Be\begin{split}\label{est:fBD_x_e1}
&|\eqref{bdr_bdr_2}|\\
&\lesssim e^{-\lambda t^1}\sup_t e^{\lambda t} |wf(t)|_\infty
\iint \int_0^{t^{1}} e^{-\nu(v^1) \tb^{1}/2}e^{-\frac{\big|x^{1} - \eta_{p^{2} } (\mathbf{x}^{2}_{p^{2}})\big|^2}{4|\tb^{1}|^2}}  \\
& \quad \quad \quad \quad \quad \quad \quad \quad \quad \quad \quad  \times \Big[
\frac{|x^{1} - \eta_{p^{2} } (\mathbf{x}^{2}_{p^{2}})|^3}{|\tb^{1}|^5}+ \frac{|x^{1} - \eta_{p^{2} } (\mathbf{x}^{2}_{p^{2}})|^2}{|\tb^{1}|^4}+ \frac{|x^{1} - \eta_{p^{2} } (\mathbf{x}^{2}_{p^{2}})|^4}{|\tb^{1}|^6}\Big]
\\
&\lesssim   e^{-\lambda t^1}\sup_t e^{\lambda t} |wf(t)|_\infty
 \int_0^{\infty}
\frac{e^{-\nu(v^1) \tb^{1}/2}}{|\tb^{1}|^{1/2}}\iint \Big[\frac{1}{|x^{1} - \eta_{p^{2} } (\mathbf{x}^{2}_{p^{2}})|^{3/2}} + \frac{1}{|x^{1} - \eta_{p^{2} } (\mathbf{x}^{2}_{p^{2}})|^{1/2}}\Big]\\
&\lesssim  e^{-\lambda t^1}\sup_t e^{\lambda t} |wf(t)|_\infty ,
\end{split}\Ee
where we have used 
\Be
\begin{split}\notag
&\Big[
\frac{|x^{1} - \eta_{p^{2} } (\mathbf{x}^{2}_{p^{2}})|^3}{|\tb^{1}|^5}+ \frac{|x^{1} - \eta_{p^{2} } (\mathbf{x}^{2}_{p^{2}})|^2}{|\tb^{1}|^4}+\frac{|x^{1} - \eta_{p^{2} } (\mathbf{x}^{2}_{p^{2}})|^4}{|\tb^{1}|^6}+\frac{|x^{1} - \eta_{p^{2} } (\mathbf{x}^{2}_{p^{2}})|^4}{|\tb^{1}|^5}\Big]
e^{-\frac{|x^{1} - \eta_{p^{2} } (\mathbf{x}^{2}_{p^{2}})|^2}{4|\tb^{1}|^2}}\\
&\leq 
\Big[\frac{|x^{1} - \eta_{p^{2} } (\mathbf{x}^{2}_{p^{2}})|^{11/2}}{|\tb^{1}|^{11/2}}+\frac{|x^{1} - \eta_{p^{2} } (\mathbf{x}^{2}_{p^{2}})|^{9/2}}{|\tb^{1}|^{9/2}} + \frac{|x^1-\eta_{p^2}(\mathbf{x}_{p^2}^2)|^{7/2}}{|\tb^1|^{7/2}}\Big]
\\
&\times \Big[\frac{1}{|\tb^{1}|^{1/2}} \frac{1}{|x^{1} - \eta_{p^{2} } (\mathbf{x}^{2}_{p^{2}})|^{3/2}} + \frac{1}{|\tb^{1}|^{1/2}} \frac{1}{|x^{1} - \eta_{p^{2} } (\mathbf{x}^{2}_{p^{2}})|^{1/2}} \Big] e^{-\frac{|x^{1} - \eta_{p^{2} } (\mathbf{x}^{2}_{p^{2}})|^2}{4|\tb^{1}|^2}}\\
&\lesssim \frac{1}{|\tb^{1}|^{1/2}} \frac{1}{|x^{1} - \eta_{p^{2} } (\mathbf{x}^{2}_{p^{2}})|^{3/2}} + \frac{1}{|\tb^{1}|^{1/2}} \frac{1}{|x^{1} - \eta_{p^{2} } (\mathbf{x}^{2}_{p^{2}})|^{1/2}}.
\end{split}
\Ee

Collecting~\eqref{bdd_deri_x_1236} \eqref{bdd_deri_x_K} and \eqref{est:fBD_x_e1} we conclude
\begin{align*}
  \eqref{bdr_deri_x}  & \lesssim e^{-\lambda t^1}\big[\mathcal{B}+o(1)\sup_{s\leq t}e^{\lambda s}\Vert w_{\tilde{\theta}}\alpha\nabla_x f(s)\Vert_\infty \big].
\end{align*}
This, together with~\eqref{bdr_deri_v_bdd}, conclude that
{\color{black}
\begin{align}
   \eqref{nabla_x_4} & \lesssim  \eqref{partial_wall_bdd}+\eqref{fBD_x} \lesssim \frac{e^{-\lambda t}}{w_{\tilde{\theta}}(v)\alpha(x,v)}\times \big[\mathcal{B}+o(1)\sup_{s\leq t}e^{\lambda s}\Vert w_{\tilde{\theta}}\alpha\nabla_x f(s)\Vert_\infty\big]. \label{4_bdd}
\end{align}
}

\textit{Step 3: estimate of~\eqref{nabla_x_5}.}
Last we compute~\eqref{nabla_x_5}. We apply the characteristic~\eqref{nabla_x_1} - \eqref{nabla_x_8} to expand $\nabla_x f(s,x-(t-s)v,u)$ along $u$. 

The contribution of~\eqref{nabla_x_1} - \eqref{nabla_x_4} and~\eqref{nabla_x_6} - \eqref{nabla_x_8} in the expansion of $\nabla_x f(s,x-(t-s)v,u)$ can be computed using the same computation in~\eqref{1_2_3_6_bdd} and~\eqref{4_bdd}, with replacing $x$ by $x-(t-s)v$, $v$ by $u$, $t$ by $s$. This leads to a bound for these seven terms as
\begin{align}
    & \int^t_{\max\{0,t-\tb\}} e^{-\nu(t-s)} \int_{\mathbb{R}^3}   \dd u \mathbf{k}(v,u) \frac{e^{-\lambda s}w^{-1}_{\tilde{\theta}}(u)}{\alpha(x-(t-s)v,u)}\times \big[\mathcal{B}+o(1)\sup_{s\leq t}e^{\lambda s}\Vert w_{\tilde{\theta}}\alpha\nabla_x f(s)\Vert_\infty\big]  \notag\\
    & \lesssim \frac{e^{-\lambda t}w^{-1}_{\tilde{\theta}}(v)}{\alpha(x,v)}\times \big[\mathcal{B}+o(1)\sup_{s\leq t}e^{\lambda s}\Vert w_{\tilde{\theta}}\alpha\nabla_x f(s)\Vert_\infty\big] \label{bdd_K_12346}.
\end{align}
Here we have used $e^{-\nu(t-s)}e^{-\lambda s}\leq e^{-\lambda t} e^{-\nu(t-s)/2}$ and applied Lemma \ref{lemma:nonlocal_to_local} with Lemma \ref{Lemma: k tilde}.

Then we consider the contribution of~\eqref{nabla_x_5}, denoting $y=x-(t-s)v$, such contribution reads
\begin{align}
    &\int_{\max\{0,t-\tb\}}^t \dd s e^{-\nu(t-s)}\int_{\mathbb{R}^3} \dd u \mathbf{k}(v,u)\notag \\
    &\times \int^s_{\max\{0,s-\tb(y,u)\}} \dd s'e^{-\nu(u)(s-s')} \int_{\mathbb{R}^3} \dd u'\mathbf{k}(u,u')\nabla_x f(s',y-(s-s')u,u')  . \label{K_K}
\end{align}
We split the $\dd s'$ integral into $s-s'<\e$ and $s-s'\geq \e$.

{\color{black} When $s-s'<\e$, we apply Lemma \ref{lemma:nonlocal_to_local} and Lemma \ref{Lemma: k tilde} to have
\begin{align}
    &\eqref{K_K}_1:=\Big|\int_{\max\{0,t-\tb\}}^t \dd s e^{-\nu(t-s)}\int_{\mathbb{R}^3} \dd u \mathbf{k}(v,u)\notag \\
    &\times \int^s_{\max\{0,s-\tb(y,u),s-\e\}} \dd s' e^{-\nu(u)(s-s')} \int_{\mathbb{R}^3} \dd u'\mathbf{k}(u,u')\nabla_x f(s',y-(s-s')u,u') \Big|\notag \\
 & \lesssim  \int_{\max\{0,t-\tb\}}^t \dd s e^{-\nu(t-s)}\int_{\mathbb{R}^3} \dd u \mathbf{k}(v,u) \int^s_{\max\{s-\tb(y,u),s-\e\}} \dd s' e^{-\nu(u)(s-s')}  \notag\\
  & \  \times \int_{\mathbb{R}^3} \dd u'\mathbf{k}(u,u') \frac{e^{-\lambda s'}w^{-1}_{\tilde{\theta}}(u')}{\alpha(y-(s-s')u,u')}\sup_{s\leq t} e^{\lambda s}\Vert w_{\tilde{\theta}}\alpha \nabla_x f(s)\Vert_\infty \notag\\
  &\lesssim  \int_{\max\{0,t-\tb\}}^t \dd s e^{-\nu(t-s)}e^{-\lambda s}\int_{\mathbb{R}^3} \dd u \frac{\mathbf{k}(v,u)}{w_{\tilde{\theta}}(u)} \int^s_{\max\{s-\tb(y,u),s-\e\}} \dd s' e^{-\nu(u)(s-s')/2}  \notag\\
  & \  \times \int_{\mathbb{R}^3} \dd u' \frac{\mathbf{k}_{\tilde{\varrho}}(u,u')}{\alpha(y-(s-s')u,u')}\sup_{s\leq t} e^{\lambda s}\Vert w_{\tilde{\theta}}\alpha \nabla_x f(s)\Vert_\infty \notag\\
  &\lesssim o(1)e^{-\lambda t}w^{-1}_{\tilde{\theta}}(v)\times \sup_{s\leq t} e^{\lambda s}\Vert w_{\tilde{\theta}}\alpha \nabla_x f\Vert_\infty\int_{\max\{0,t-\tb\}}^t \dd s e^{-\nu(t-s)/2}\int_{\mathbb{R}^3} \dd u \frac{\mathbf{k}_{\tilde{\varrho}}(v,u)}{\alpha(y,u)} \notag\\
  &\lesssim o(1)\frac{e^{-\lambda t}w^{-1}_{\tilde{\theta}}(v)}{\alpha(x,v)}\sup_{s\leq t} e^{\lambda s}\Vert w_{\tilde{\theta}}\alpha \nabla_x f(s)\Vert_\infty. \label{kk_leq_e}
\end{align}
Here we note that we applied~\eqref{est:nonlocal_with_e} in the fifth line and~\eqref{est:nonlocal_wo_e} in the last line.

When $s-s'\geq \e$, we further split the $u$ integration into $|u|>\delta$ and $|u|<\delta$. For $|u|<\delta$, by a similar computation as \eqref{kk_leq_e}, we have
\begin{align*}
    &\Big|\int_{\max\{0,t-\tb\}}^{t-\e} \dd s e^{-\nu(t-s)}\int_{|u|<\delta} \dd u \mathbf{k}(v,u)\notag \\
    &\times \int^s_{\max\{0,s-\tb(y,u),s-\e\}} \dd s' e^{-\nu(u)(s-s')} \int_{\mathbb{R}^3} \dd u'\mathbf{k}(u,u')\nabla_x f(s',y-(s-s')u,u') \Big|\\
    &\lesssim e^{-\lambda t}w^{-1}_{\tilde{\theta}}(v)\times \sup_{s\leq t} e^{\lambda s}\Vert w_{\tilde{\theta}}\alpha \nabla_x f\Vert_\infty\int_{\max\{0,t-\tb\}}^t \dd s e^{-\nu(t-s)/2}\int_{|u|<\delta} \dd u \frac{\mathbf{k}_{\tilde{\varrho}}(v,u)}{\alpha(y,u)} \\
    &\lesssim o(1)\frac{e^{-\lambda t}w^{-1}_{\tilde{\theta}}(v)}{\alpha(x,v)}\sup_{s\leq t} e^{\lambda s}\Vert w_{\tilde{\theta}}\alpha \nabla_x f(s)\Vert_\infty. \label{kk_leq_delta}
\end{align*}
In the last line we used \eqref{est:nonlocal_delta}.

For $s-s'\geq \e$ and $|u|>\delta$, we apply the following change of variable
\begin{equation}\label{cov:u}
\nabla_x f(s',y-(s-s')u,u') = -\frac{\nabla_{u}[f(s',y-(s-s')u,u')]}{s-s'}.    
\end{equation}
Without loss of generality, we assume $s>\e$, otherwise, we directly bound \eqref{K_K} by \eqref{kk_leq_e}. In such case, we only integrate over the $u$ such that $\tb(y,u)>\e$. Then we denote
\begin{align*}
    U:=\{u\in\mathbb{R}^3: |u|>\delta, \ \tb(y,u)>\e\}.
\end{align*}
Then by Lemma \ref{lemma:convex_property}, we have $|n((\xb(y,u)))\cdot u|\gtrsim \e|\delta|^2$. With \eqref{nabla_tbxb}, the corresponding hypersurface, denoted as $\tilde{U}$, is piecewise smooth.

We use the change of variable \eqref{cov:u} to express such case as
\begin{align*}
    &\eqref{K_K}_2:= \Big| \int_{\max\{0,t-\tb\}}^{t} \dd s e^{-\nu(t-s)}\int_{U} \dd u \mathbf{k}(v,u) \\
    &\times \int^{s-\e}_{\max\{0,s-\tb(y,u)\}} \dd s'e^{-\nu(u)(s-s')} \int_{\mathbb{R}^3} \dd u'\mathbf{k}(u,u') \frac{-\nabla_u f(s',y-(s-s')u,u')}{s-s'} \Big|  .
\end{align*}

Now we apply an integration by part in $\dd u$ to obtain
\begin{equation}\label{k_ibp}
\begin{split}
  & \eqref{K_K}_2=  \Big|\int_{\max\{0,t-\tb\}}^t \dd s e^{-\nu(t-s)}\int_{\mathbb{R}^3} \dd u  \int^{s-\e}_{\max\{0,s-\tb(y,u)\}} \dd s'  \\
  &\times \int_{\mathbb{R}^3} \dd u'\nabla_u[e^{-\nu(u)(s-s')}\mathbf{k}(v,u)\mathbf{k}(u,u')] \frac{f(s',y-(s-s')u,u')}{s-s'} \\
  &+ \int_{\max\{0,t-\tb\}}^t \dd s e^{-\nu(t-s)}\int_{\mathbb{R}^3} \dd u \mathbf{k}(v,u) \\
  &\times\int_{\mathbb{R}^3} \dd u' \mathbf{k}(u,u')\nabla_u \tb(y,u)e^{-\nu(u)\tb(y,u)} \frac{f(s-\tb(y,u),\xb(y,u),u')}{\tb(y,u)} \\
  & + \int_{\max\{0,t-\tb\}}^{t} \dd s e^{-\nu(t-s)}\int_{\tilde{U}} \dd u \mathbf{n}_u \mathbf{k}(v,u) \int^{s-\e}_{\max\{0,s-\tb(y,u)\}} \dd s'  e^{-\nu(u)(s-s')}
    \\
    &\times    \int_{\mathbb{R}^3} \dd u' \mathbf{k}(u,u') \frac{f(s',y-(s-s')u,u')}{s-s'}        \Big|.    
\end{split}
\end{equation}
In the last term, $\mathbf{n}_u = \frac{-\nabla_u \tb(y,u)}{|\nabla_u \tb(y,u)|}$. Then the last term is bounded as
\begin{align}
    & \frac{\e^{-1} \sup_t e^{\lambda t}\Vert wf(t)\Vert_\infty}{w_{\tilde{\theta}}(v)} \int^t_{\max\{0,t-\tb\}} \dd e^{-\nu(t-s)} \int_{\tilde{U}} \dd u \mathbf{k}(v,u)\frac{w_{\tilde{\theta}}(v)}{w_{\tilde{\theta}}(u)}  \notag \\
    &\times \int^{s-\e}_{s-\tb(y,u)} \dd s' e^{-\nu(u)(s-s')}\int_{\mathbb{R}^3} \dd u' \mathbf{k}(u,u') e^{-\lambda s'}\frac{w_{\tilde{\theta}}(u)}{w_{\tilde{\theta}}(u')} \notag\\
    &\lesssim \frac{\e^{-1}e^{-\lambda t} \sup_t e^{\lambda t}\Vert wf(t)\Vert_\infty}{w_{\tilde{\theta}}(v)\alpha(x,v)}.   \label{kk_geq_e_last}
\end{align}
In the last line we used $\alpha(x,v)\lesssim 1$ from \eqref{n geq alpha}, and we used \eqref{k_theta}.}

The first term in \eqref{k_ibp} is bounded as
\begin{align}
    &\sup_t e^{\lambda t}\Vert wf(t)\Vert_\infty \notag\\
    &\times \int_{\max\{0,t-\tb\}}^t \dd s e^{-\nu(t-s)}\int_{\mathbb{R}^3} \dd u  \int^{s-\e}_{\max\{0,s-\tb(y,u)\}} \dd s' \notag\\
    &\times \Big|\int_{\mathbb{R}^3} \dd u'\nabla_u[e^{-\nu(u)(s-s')}\mathbf{k}(v,u)\mathbf{k}(u,u')] \frac{e^{-\lambda s'}w^{-1}(u')}{s-s'} \Big|\notag\\
    &\lesssim \e^{-1}e^{-\lambda t} w^{-1}_{\tilde{\theta}}(v)\sup_s e^{\lambda s}\Vert wf(s)\Vert_\infty \lesssim \frac{e^{-\lambda t}\times \mathcal{B}}{w_{\tilde{\theta}}(v)\alpha(x,v)} . \label{kk_geq_e_1}
\end{align}
In the last inequality, we used~\eqref{n geq alpha}. {\color{black}To get the first inequality in the last line, we have used $\tilde{\theta} \ll \theta$ so that $w^{-1}(u')\leq w^{-1}_{2\tilde{\theta}}(u')$, $|u|^2 w_{\tilde{\theta}}^{-1}(u)\lesssim 1$, then we \eqref{k_theta} and~\eqref{nabla_k_theta} to control the $\mathbf{k}$ terms in the third line as
\begin{align*}
   & |\nabla_u [\mathbf{k}(v,u)\mathbf{k}(u,u')] | w^{-1}(u') \leq |\nabla_u [\mathbf{k}(v,u)\mathbf{k}(u,u')] | w^{-1}_{2\tilde{\theta}}(u')\\
   & \lesssim  |\nabla_u \mathbf{k}(v,u)| w^{-1}_{2\tilde{\theta}}(u) \mathbf{k}(u,u')\frac{w_{2\tilde{\theta}}(u)}{w_{2\tilde{\theta}}(u')} + \mathbf{k}(v,u) w^{-1}_{2\tilde{\theta}}(u) |\nabla_u \mathbf{k}(u,u')|\frac{w_{2\tilde{\theta}}(u)}{w_{2\tilde{\theta}}(u')}  \\
    &\lesssim w^{-1}_{\tilde{\theta}}(v)w^{-1}_{\tilde{\theta}}(u) |\nabla_u \mathbf{k}(v,u)| \frac{w_{\tilde{\theta}}(v)}{w_{\tilde{\theta}}(u)} \mathbf{k}_{\tilde{\varrho}}(u,u') + \mathbf{k}(v,u) w^{-1}_{\tilde{\theta}}(u) [1+|u|^2]w^{-1}_{\tilde{\theta}}(u) \frac{\mathbf{k}_{\tilde{\varrho}}(u,u')}{|u-u'|} \\
    & \lesssim w^{-1}_{\tilde{\theta}}(v)[1+|u|^2] w^{-1}_{\tilde{\theta}}(u) \frac{\mathbf{k}_{\tilde{\varrho}}(v,u)}{|v-u|} \mathbf{k}_{\tilde{\varrho}}(u,u') + w^{-1}_{\tilde{\theta}}(v) \mathbf{k}(v,u) \frac{w_{\tilde{\theta}}(v)}{w_{\tilde{\theta}}(u)} \frac{\mathbf{k}_{\tilde{\varrho}}(u,u')}{|u-u'|} \\
    &\lesssim w^{-1}_{\tilde{\theta}}(v) \frac{\mathbf{k}_{\tilde{\varrho}}(v,u)}{|v-u|} \mathbf{k}_{\tilde{\varrho}}(u,u') + w^{-1}_{\tilde{\theta}}(v)\mathbf{k}_{\tilde{\varrho}}(v,u)  \frac{\mathbf{k}_{\tilde{\varrho}}(u,u')}{|u-u'|} \in L^1_{u,u'}.
\end{align*}}
The contribution of $\nabla_u e^{-\nu(u)(s-s')}$ can be controlled by \eqref{nablav nu} and the extra $\frac{1}{s-s'}$.

For the second term in \eqref{k_ibp}, we apply~\eqref{nabla_tbxb}, $e^{-\nu(u)\tb(y,u)} e^{-\lambda(s-\tb(y,u))} \leq e^{-\lambda s}$ and $e^{-\nu(t-s)}e^{-\lambda s}\leq e^{-\nu(t-s)/2}e^{-\lambda t}$ to obtain a bound as
\begin{align}
    & \sup_t e^{\lambda t}| wf(t)|_\infty \int_{\max\{0,t-\tb\}}^t \dd s e^{-\nu(t-s)}\int_{\mathbb{R}^3} \dd u \frac{\mathbf{k}(v,u)}{\alpha(y,u)} \notag\\
    &\times\int_{\mathbb{R}^3} \dd u' \mathbf{k}(u,u') e^{-\nu(u)\tb(y,u)}e^{-\lambda (s-\tb(y,u))} w^{-1}(u') \notag\\
    &\lesssim  \sup_t e^{\lambda t}| wf(t)|_\infty \int_{\max\{0,t-\tb\}}^t \dd s e^{-\nu(t-s)}\int_{\mathbb{R}^3} \dd u \frac{\mathbf{k}(v,u)}{\alpha(y,u)w_{\tilde{\theta}}(u)} \int_{\mathbb{R}^3} \dd u' \mathbf{k}_{\tilde{\varrho}}(u,u') e^{-\lambda s} \notag\\
    &\lesssim \frac{e^{-\lambda t}w_{\tilde{\theta}}^{-1}(v)\sup_t e^{\lambda t}| wf(t)|_\infty}{\alpha(x,v)} \lesssim \frac{e^{-\lambda t}}{w_{\tilde{\theta}}(v)\alpha(x,v)}\times \mathcal{B}. \label{kk_geq_e_2}
\end{align}
Here we recall $\mathcal{B}$ is defined in~\eqref{upp_bdd_B}.

Collecting~\eqref{kk_leq_e} and \eqref{kk_geq_e_last} -- \eqref{kk_geq_e_2} we conclude
\begin{align}
   \eqref{K_K} & \lesssim  \frac{e^{-\lambda t}}{w_{\tilde{\theta}}(v)\alpha(x,v)} \big[\mathcal{B}+o(1)\sup_{s\leq t}e^{\lambda s}\Vert w_{\tilde{\theta}}\alpha \nabla_x f(s)\Vert_\infty \big]\label{bdd_K_5}.
\end{align}

Combining~\eqref{bdd_K_12346} and~\eqref{bdd_K_5} we conclude that
{\color{black}
\begin{align}
  \eqref{nabla_x_5}  & \lesssim  \frac{e^{-\lambda t}}{w_{\tilde{\theta}}(v)\alpha(x,v)} \big[\mathcal{B}+o(1)\sup_{s\leq t}e^{\lambda s}\Vert w_{\tilde{\theta}}\alpha \nabla_x f(s)\Vert_\infty \big] \label{5_bdd}.
\end{align}

\textit{Step 4: conclusion.}
To conclude the theorem we combine~\eqref{1_2_3_6_bdd}, \eqref{4_bdd} and \eqref{5_bdd} to have
\begin{align*}
  |\nabla_x f(t,x,v)|  & \lesssim \frac{e^{-\lambda t}w^{-1}_{\tilde{\theta}}(v)}{\alpha(x,v)} \times \big[\mathcal{B}+o(1)\sup_{s\leq t}e^{\lambda s}\Vert w_{\tilde{\theta}}\alpha \nabla_x f(s)\Vert_\infty \big] ,
\end{align*}
and thus 
\begin{align*}
    |e^{\lambda t}w_{\tilde{\theta}}(v)\alpha \nabla_x f(t)| &\lesssim \mathcal{B}+o(1)\sup_{s\leq t}e^{\lambda s}\Vert w_{\tilde{\theta}}\alpha \nabla_x f(s)\Vert_\infty .
\end{align*}
Since the inequality holds for all $t>0$, we take the supremum in $s\leq t$ to have
\begin{align*}
  \sup_{t}e^{\lambda t}\Vert w_{\tilde{\theta}}\alpha\nabla_x f(t)\Vert_\infty  & \lesssim \mathcal{B}.
\end{align*}
We recall $\mathcal{B}$ is defined in~\eqref{upp_bdd_B}, which is bounded, in particular, $\sup_t e^{\lambda t}|w\partial_t f(t)|_\infty$ is bounded by $\Vert w\partial_t f_0\Vert_\infty$ in Corollary \ref{corollary:time_derivative}. With fixed $\e>0$, we finish the proof.}

\ \\

\subsection{Proof of~\eqref{estF_n}.}\label{sec:proof_sketch}
From~\eqref{F_steady} and~\eqref{bc_F_steady}, the equation for $f_s$ in~\eqref{linearization_steady} reads
\begin{equation}\label{eqn_f_s}
\begin{split}
    &v\cdot \nabla_x f_s +\nu f_s = K(f_s)+\Gamma(f_s,f_s) , \\
    &  f_s(x,v)|_{n(x)\cdot v<0}=\frac{M_W(x,v)}{\sqrt{\mu(v)}}\int_{n(x)\cdot u>0}f_s(x,u) \sqrt{\mu(u)} \{n(x)\cdot u\}\dd u+r(x,v),
\end{split}
\end{equation}
here the remainder term is
\begin{equation}\label{remainder term}
r(x,v):=\frac{M_W(x,v)-\mu(v)}{\sqrt{\mu(v)}}.
\end{equation}

Denoting $g=K(f_s)+\Gamma(f_s,f_s)$, we apply method of characteristic to~\eqref{eqn_f_s} with fixed $t\gg 1$ to have 
\begin{align}
\nabla_x   f_s(x,v)&=  \mathbf{1}_{t\geq \tb } e^{-\nu(v)\tb(x,v)} \nabla_x   [f_s(\xb(x,v),v) ]
 \label{fx_x_1}
\\
    & -  \mathbf{1}_{t\geq \tb } \nu(v) \nabla_x \tb(x,v)  e^{-\nu(v)\tb(x,v)} f_s(\xb(x,v),v)
 \label{fx_x_2}
    \\
    &
    +\mathbf{1}_{t<\tb }e^{-\nu(v)t} \nabla_x [ f_s(x-tv,v)]
     \label{fx_x_3}
    \\
    & +\int_{\max\{0,t-\tb \}}^t e^{-\nu(v)(t-s)} \nabla_x  g (   x-(t-s) v,v) \dd s
     \label{fx_x_4}\\
    &- \mathbf{1}_{t\geq \tb} \nabla_x \tb e^{-\nu(v)\tb}g(x-\tb v,v) .\label{fx_x_5}
\end{align}
Compared with~\eqref{nabla_x_1} - \eqref{nabla_x_8}, the difference is that there is no temporal derivative, no initial condition, while there is an extra term~\eqref{fx_x_3}, and the boundary condition has an extra term~\eqref{remainder term}. On the boundary, $f_s$ is bounded in~\eqref{infty_bound}. Thus we only need to consider the contribution of~\eqref{fx_x_3} and $r$ in~\eqref{fx_x_1}.

With $t\gg 1$, we bound \eqref{fx_x_3} as
\begin{align*}
  |\eqref{fx_x_3}|  & \lesssim \frac{o(1)\Vert w_{\tilde{\theta}}\alpha\nabla_x f_s\Vert_\infty}{w_{\tilde{\theta}}(v)\alpha(x,v)}.
\end{align*}

For~\eqref{fx_x_1}, by a similar computation as~\eqref{partial_wall}, such extra contribution reads
\begin{align*}
 & \frac{e^{-\nu \tb}|v| \partial_{\mathbf{x}_{p^1,j}^1} M_W(\eta_{p^1}(\mathbf{x}_{p^1}^1),v)}{\sqrt{\mu(v)}\alpha(x,v)} \\
 & \lesssim  \frac{e^{-\nu \tb}|v| }{\sqrt{\mu(v)}\alpha(x,v)} \Vert \eta\Vert_{C^1}\Vert T_W\Vert_{C^1(\p\O)}[1+|v|]M_W(x,v)  \lesssim \frac{\Vert T_W-T_0\Vert_{C^1(\p\O)}}{w_{\tilde{\theta}}(v)\alpha(x,v)}.
\end{align*}
Following the same computation in Section \ref{sec:proof_thm_2} for the rest terms, we conclude~\eqref{estF_n}.

\ \\

\section{Uniform-in-time $W^{1,p}$-estimate for $p<3$}\label{sec:W1p}
We aim to utilize Theorem \ref{thm:dynamic_regularity} to obtain the $\alpha$-weighted $C^1$ estimate~\eqref{weight_C1}, and then deduce a $W^{1,p}$ estimate without weight for $p<3$. Since the weight $\alpha$ in Theorem \ref{thm:dynamic_regularity} behaves as $\alpha(x,v)\thicksim |n(\xb(x,v))\cdot v|$, we address the singularity of $1/|n(\xb(x,v))\cdot v|$ in the $W^{1,p}$ estimate through Lemma \ref{lemma:W1p}.

We will use the change of variable in Lemma \ref{lemma:cov} and some properties of the convex domain in Lemma \ref{lemma:convex_property} to deduce Lemma \ref{lemma:W1p}.

\begin{lemma}\label{lemma:cov}
For any $g$,
\begin{equation}\label{cov}
\begin{split}
  \iint_{\O\times \mathbb{R}^3} g(y,v) \dd y \dd v  & = \int_{\gamma_+}\int_0^{\tb(x,v)}g(x-sv,v)|n(x)\cdot v|\dd s \dd v \dd S_x \\
    & = \int_{\gamma_-} \int_0^{\tb(x,-v)}g(x+sv,v)|n(x)\cdot v|\dd s \dd v \dd S_x.
\end{split}
\end{equation}
\end{lemma}

\begin{proof}
First, we prove that the following map is a bijection
\begin{equation}\label{map}
(s,x,v) \in (0,\tb(x,v)) \times \gamma_+ \to (x-sv,v)\in \Omega \times \mathbb{R}^3,
\end{equation}
with Jacobian given by
\begin{equation}\label{Jacobian_xv}
\Big|\det \Big(\frac{\partial (x-sv,v)}{\partial (s,x,v)}\Big)\Big| = |n(x)\cdot v|.
\end{equation}
Clearly the map is well-defined since $x-sv\in \Omega$ for $0<s<\tb(x,v)$.

To show the map is injective we suppose $(x_1-s_1v_1,v_1) = (x_2-s_2v_2,v_2)$, then $v_1=v_2=v$, and thus $x_1-s_1v = x_2-s_2v$, which implies $x_1-x_2 = (s_1-s_2)v$. If $s_1\neq s_2$, then $x_1\neq x_2$, and $x_1,x_2$ are in the same characteristic. $(x_1,v)\in \gamma_+$ and $x_2\in\partial \Omega$ implies $x_2 = x_1 - \tb(x_1,v)v=\xb(x_1,v)$, which further implies $(x_2,v)\in \gamma_-\cup \gamma_0$. This contradicts to the fact that $(x_2,v)\in \gamma_+$. By the contradiction argument, we conclude $s_1=s_2$ and $x_1=x_2$.

To show the map is surjective, we define the 
\begin{equation}\label{tf}
\tf(x,v) := \sup\{s>0:x+sv\in \Omega\}, \ \ \xf(x,v):=x+\tf(x,v)v.
\end{equation}
For any $(y,v)\in \Omega\times \mathbb{R}^3$, we take $x=y+\tf(y,v)v\in \partial \Omega$. Then $x-\tf(y,v)v = y$ and $(x,v)\in \gamma_+$. Since $y\in \Omega$, we then have $\tf(y,v)<\tb(x,v)$. Thus~\eqref{map} is surjective and we conclude that it is also bijective.

Then we compute the Jacobian of the map. For $x\in \Omega$, by~\eqref{O_p} $x = \eta_{p}(\mathbf{x}_p)$, then we evaluate the following determinant as
\begin{align*}
    &\Big| \det\Big(\frac{\partial(\eta_p(\mathbf{x}_p)-sv, v)}{\partial (\mathbf{x}_{p,1},\mathbf{x}_{p,2},v)}\Big)\Big|  \notag\\
    & = \Big| \det\left( \begin{array}{cccc}
v  &  \partial_1 \eta_{p}(\mathbf{x}_p) & \partial_2 \eta_p(\mathbf{x}_p) & -s Id \\
0 & 0& 0& Id
\end{array} \right)\Big|  \\
& = \Big| \det\left( \begin{array}{ccc}
v  &  \partial_1 \eta_{p}(\mathbf{x}_p) & \partial_2 \eta_p(\mathbf{x}_p)
\end{array} \right)\Big|  \\
& = |v\cdot (\p_1 \eta_p(\mathbf{x}_p)\times  \p_2\eta_p(\mathbf{x}_p))| = |v\cdot n(x)||(\p_1 \eta_p(\mathbf{x}_p)\times  \p_2\eta_p(\mathbf{x}_p))|.
\end{align*}
Here we have applied~\eqref{orthogonal}.

Since the surface measure equals $\dd S_x = |\p_1 \eta_p(\mathbf{x}_p) \times \p_2 \eta_p(\mathbf{x}_p)|\dd \mathbf{x}_{p,1}\dd \mathbf{x}_{p,2}$, we conclude~\eqref{Jacobian_xv} and the first line of~\eqref{cov}. The proof of the second line of~\eqref{cov} is similar.

\end{proof}

\begin{lemma}\label{lemma:convex_property}
In a convex and $C^2$ domain~\eqref{convex}, $\tb(x,v)$ is bounded as
\begin{equation}\label{tb_bdd}
  \tb(x,v) \lesssim \frac{|n(\xb(x,v))\cdot v|}{|v|^2}  .
\end{equation}
For $x\in \partial \Omega$, we have
\begin{equation}\label{nb_proportion_nxb}
|n(x)\cdot v| \thicksim |n(\xb(x,v))\cdot v|.
\end{equation}
    
\end{lemma}

\begin{proof}
{\color{black}The proof is similar to \cite{G}.} First we prove~\eqref{tb_bdd}. It suffices to prove $\tb(x,v)\lesssim \frac{|n(\xb(x,v))\cdot v|}{|v|^2}$ when $(x,v)\in \gamma_+$. When $x\notin \p \O$, \eqref{tb_bdd} follows from the fact that $\tb(x,v)\leq \tb(\xf(x,v),v)$, where $\xf(x,v)$ is defined in~\eqref{tf}. Then for $(x,v)\in \gamma_+$, we have $\xi(x)=0=\xi(x-\tb(x,v)v)$, then
\begin{align*}
  0  &=\xi(x-\tb(x,v)v)=\xi(x)+\int_0^{\tb(x,v)} [-v\cdot \nabla_x \xi(x-sv)] \dd s  \\
    & = [-v\cdot \nabla_x \xi(x-\tb(x,v)v)]\tb(x,v)+\int_0^{\tb(x,v)}\int_{\tb(x,v)}^s \{v\cdot \nabla_x^2 \xi(x-\tau v)\cdot v\}\dd \tau \dd s.
\end{align*}
This leads to
\[-[v\cdot \nabla_x \xi(\xb(x,v))]\tb(x,v) = \int_0^{\tb(x,v)} \int_s^{\tb(x,v)}  \{v\cdot \nabla_x^2 \xi(x-\tau v)\cdot v\} \dd \tau \dd s  .\]
From the convexity~\eqref{convex}, we have $-[v\cdot \nabla_x \xi(\xb(x,v))]\tb(x,v)\gtrsim (\tb(x,v))^2|v|^2$, thus we conclude~\eqref{tb_bdd}.

Then we prove~\eqref{nb_proportion_nxb}. By the definition of the $\tilde{\alpha}$ in~\eqref{kinetic_distance}, for $x\in \p\O$ and $\xb(x,v)\in \p\O$ such that $\xi(x)=0,\xi(\xb(x,v))=0$, we have
\[\tilde{\alpha}(x,v)\thicksim |n(x)\cdot v|, \ \ \tilde{\alpha}(\xb(x,v),v)\thicksim |n(\xb(x,v))\cdot v|.\]
Then by the velocity lemma~\eqref{Velocity_lemma},
we conclude $|n(\xb(x,v))\cdot v|\thicksim |n(x)\cdot v|$.

\end{proof}

Now we control the singularity of $|n(\xb(x,v))\cdot v|^{-p}$ from the following lemma.
\begin{lemma}\label{lemma:W1p}
In a convex domain~\eqref{convex}, for $p<3$, we have
\begin{equation}\label{W1p_bdd}
\iint_{\Omega\times \mathbb{R}^3} \frac{w^{-p}_{\tilde{\theta}}(v)}{|n(\xb(x,v))\cdot v|^p}\dd x\dd v \lesssim 1.  
\end{equation}
\end{lemma}

\begin{proof}
We directly compute the LHS of~\eqref{W1p_bdd} as
\begin{align*} &\iint_{\Omega\times \mathbb{R}^3} \frac{w^{-p}_{\tilde{\theta}}(v)}{|n(\xb(x,v))\cdot v|^p}\dd x\dd v \\
& \lesssim  \int_{\gamma_+}\int_0^{\tb(x,v)} w^{-p}_{\tilde{\theta}}(v) \frac{|n(x)\cdot v|}{|n(\xb(x-sv,v))\cdot v|^p} \dd s \dd v \dd S_x    \\
    & \lesssim \int_{\gamma_+}w^{-p}_{\tilde{\theta}}(v) \frac{|n(x)\cdot v|\tb(x,v)}{|n(\xb(x,v))\cdot v|^p} \dd v \dd S_x\\
    & \lesssim \int_{\gamma_+} w^{-p}_{\tilde{\theta}}(v) \frac{|n(x)\cdot v| |n(\xb(x,v))\cdot v|}{|v|^2|n(\xb(x,v))\cdot v|^p} \dd s \dd v \dd S_x \\
    & \lesssim \int_{\gamma_+} w^{-p}_{\tilde{\theta}}(v) \frac{1}{|v|^2|n(x)\cdot v|^{p-2}}  \dd v \dd S_x \lesssim 1.
\end{align*}
In the second line, we apply the change of variable~\eqref{cov}. In the fourth line, we apply~\eqref{tb_bdd}. In the last line, we apply~\eqref{nb_proportion_nxb}. In the last inequality, for fixed $x\in \p \O$ and fixed $n(x)$, we used the spherical coordinate 
\[v_{\parallel,1}=r\sin \theta \cos \phi, \ v_{\parallel,2} = r\sin \theta \sin \phi, \ n(x)\cdot v = r\cos \theta,\]
here $v_{\parallel,i}$ is defined to be the projection of $v$ to  $\tau_i$, where $(\tau_1,\tau_2,n(x))$ is a basis of $\mathbb{R}^3$ such that $\tau_1\cdot \tau_2=0$ and $\tau_i\cdot n(x)=0$. Then for $p<3$, we apply the spherical coordinate to have
\begin{align*}
   \int_{\mathbb{R}^3}  \frac{w^{-p}_{\tilde{\theta}}(v)}{|v|^2 |n(x)\cdot v|^{p-2}}\dd v & \lesssim \int_0^{\pi} \frac{1}{|\cos \theta|^{p-2}} \dd \theta \int_0^{\infty} \dd r w_{\tilde{\theta}}^{-p}(r) \frac{r^2}{r^{p} } \lesssim 1.
\end{align*}

\end{proof}

We are ready to prove Theorem \ref{thm:W1p}.
\begin{proof}[\textbf{Proof of Theorem \ref{thm:W1p}}]
Recall the weighted $C^1$ estimate for $f_s$ in~\eqref{estF_n}
\[\Vert w_{\tilde{\theta}}(v)\alpha\nabla_x f_s\Vert_\infty<\infty.\]
Then
\begin{align*}
 \Vert \nabla_x f_s \Vert_p  &  =  \Big(\iint_{\Omega\times \mathbb{R}^3} |\nabla_x f_s|^p \dd x\dd v \Big)^{1/p} \\
  &\lesssim\Vert w_{\tilde{\theta}}\alpha\nabla_x f_s\Vert_\infty \Big[\iint_{\Omega\times \mathbb{R}^3}\Big(\frac{w_{\tilde{\theta}}^{-1}(v)}{\alpha(x,v)} \Big)^{p} \dd x \dd v\Big]^{1/p} \\
  &\lesssim \Vert w_{\tilde{\theta}}\alpha\nabla_x f_s\Vert_\infty \Big( \iint_{\Omega\times \mathbb{R}^3} \frac{w^{-p}_{\tilde{\theta}/2}(v)}{|n(\xb(x,v))\cdot v|^p} \dd x \dd v\Big)^{1/p} \\
  &\lesssim \Vert w_{\tilde{\theta}}\alpha\nabla_x f_s\Vert_\infty<\infty.
\end{align*}
In the last line, we have applied Lemma \ref{lemma:W1p}. In the third line we have used~\eqref{bdr_alpha} and $\tilde{\alpha}(x,v)\lesssim |v|$ to have
\begin{align*}
   \frac{w^{-1}_{\tilde{\theta}}(v)}{\alpha(x,v)} & = \frac{1}{w_{\tilde{\theta}/2}(v)\tilde{\alpha}(x,v)} \frac{\tilde{\alpha}(x,v)}{w_{\tilde{\theta}/2}(v)\alpha(x,v)}  \\
   & \lesssim  \frac{1}{w_{\tilde{\theta}/2}(v)|n(\xb(x,v))\cdot v|} \frac{\langle v\rangle}{w_{\tilde{\theta}/2}(v)} \lesssim \frac{1}{w_{\tilde{\theta}/2}(v)|n(\xb(x,v))\cdot v|} .
\end{align*}
We conclude \eqref{f_s_W1p}.

Similarly, when all conditions in Theorem \ref{thm:dynamic_regularity} are satisfied, we have
\[e^{\lambda t}\Vert w_{\tilde{\theta}}\alpha\nabla_x f(t)\Vert_\infty<\infty,\]
then
\begin{align*}
  \sup_{t} e^{\lambda t}\Vert \nabla_x f(t) \Vert_p  &  =\sup_t  \Big(\iint_{\Omega\times \mathbb{R}^3} |e^{\lambda t}\nabla_x f(t)|^p \dd x\dd v \Big)^{1/p} \\
  &\lesssim \sup_t e^{\lambda t}\Vert w_{\tilde{\theta}}\alpha\nabla_x f(t)\Vert_\infty \Big[\iint_{\Omega\times \mathbb{R}^3}\Big(\frac{w_{\tilde{\theta}}^{-1}(v)}{\alpha(x,v)} \Big)^{p} \dd x \dd v\Big]^{1/p} \\
  &\lesssim \sup_t e^{\lambda t}\Vert w_{\tilde{\theta}}\alpha\nabla_x f(t)\Vert_\infty \Big( \iint_{\Omega\times \mathbb{R}^3} \frac{w^{-p}_{\tilde{\theta}/2}(v)}{|n(\xb(x,v))\cdot v|^p} \dd x \dd v\Big)^{1/p} \\
  &\lesssim \sup_t e^{\lambda t}\Vert w_{\tilde{\theta}}\alpha\nabla_x f(t)\Vert_\infty<\infty,
\end{align*}
we conclude \eqref{f_W1p}. 
\end{proof}

\appendix

\section{Remark on the second order derivative}
Here we illustrate our attempts for the second order derivative in the interior and highlight the obstacle. We investigate the second order derivative to the steady problem~\eqref{eqn_f_s}. Taking one more derivative to the~\eqref{fx_x_1} - \eqref{fx_x_5} with $g=K(f_s)+\Gamma(f_s,f_s)$, we have
\begin{align}
   \partial_{x_i}\partial_{x_j} f_s(x,v) & = \mathbf{1}_{t<\tb} e^{-\nu(v)t}\partial_{x_i}\partial_{x_j}f_s(x-tv,v)\\
   &-\mathbf{1}_{t\geq \tb} \nu(v)\p_{x_i}\tb(x,v) e^{-\nu(v)\tb(x,v)}\partial_{x_j}[f_s(\xb(x,v),v)] \label{fxx_1}\\
    &  +\mathbf{1}_{t\geq \tb} e^{-\nu(v)\tb(x,v)}
    \partial_{x_j}[\partial_{x_i}[f_s(\xb(x,v)),v]]\\
    & -\mathbf{1}_{t\geq \tb}\nu(v) \partial_{x_i}\partial_{x_j}\tb(x,v) e^{-\nu(v)\tb(x,v)} f_s(\xb(x,v),v) \\
    & + \mathbf{1}_{t\geq \tb}(\nu(v))^2 \partial_{x_j}\tb(x,v)\partial_{x_i}\tb(x,v) e^{-\nu(v)\tb(x,v)}f_s(\xb(x,v),v)\\
    & -\mathbf{1}_{t\geq \tb}\nu(v)\partial_{x_j} \tb(x,v)e^{-\nu(v)\tb(x,v)}\partial_{x_i}[f_s(\xb(x,v),v)]\\
    & + \int_{\max\{0,t-\tb\}}^t e^{-\nu(t-s)} K \partial_{x_i}\partial_{x_j} f_s(x-(t-s)v,v) \dd s \label{fxx_K}\\
    &+\int_{\max\{0,t-\tb\}}^t e^{-\nu(t-s)} \partial_{x_i}\partial_{x_j} \Gamma(f_s,f_s)(x-(t-s)v,v) \dd s \\
    & -\mathbf{1}_{t\geq \tb}\partial_{x_i}\tb e^{-\nu(v)\tb} K \partial_{x_j} f_s(x-\tb v,v) \label{fxx_k_px}\\
    & - \mathbf{1}_{t\geq \tb}\p_{x_i}\tb e^{-\nu \tb}\p_{x_j}\Gamma(f_s,f_s)(x-\tb v,v) \\
    & -\mathbf{1}_{t\geq \tb}\partial_{x_i} \partial_{x_j} \tb e^{-\nu(v)\tb} (Kf_s+\Gamma(f_s,f_s))(x-\tb v,v) \label{fxx_tb_xx}\\
    &+\mathbf{1}_{t\geq \tb}\nu(v) \partial_{x_i}\tb \partial_{x_j}\tb e^{-\nu(v) \tb} (Kf_s+\Gamma(f_s,f_s))(x-\tb v,v) \\
    & +\mathbf{1}_{t\geq \tb}\partial_{x_j}\tb e^{-\nu(v)\tb} \big\{\p_{x_i}[Kf_s(\xb(x,v),v)] + \p_{x_i}[\Gamma(f_s,f_s)(\xb(x,v),v)]\big\}. \label{fxx_11}
\end{align}
We consider~\eqref{fxx_k_px}, which reads
\begin{align}
    & \p_{x_i} \tb e^{-\nu \tb} \int_{\mathbb{R}^3} \dd u\mathbf{k}(v,u)\p_{x_j} f_s(\xb(x,v),u) . \label{second_derivative}
\end{align}

On the boundary, we express the normal derivative as 
\[v\cdot \nabla_x f_s = (n\cdot v)\partial_n f_s + \sum_{i=1,2}(\tau_i\cdot v)\partial_{\tau_i}f_s,\]
where $\tau_1,\tau_2,n$ is a basis. Thus using the equation~\eqref{eqn_f_s} the normal derivative exhibits a singularity as 
\[\partial_{n}f_s(x,u) = \frac{-\sum_{i=1,2}(\tau_i\cdot u)\partial_{\tau_i}f_s-Lf_s+\Gamma(f_s,f_s)}{n(x)\cdot u}.\]
Whenever $e_j\cdot n(\xb(x,v))\neq 0$, $\p_{x_j}f(\xb(x,v),u)$ exhibits a singularity, which leads to
\begin{align*}
  \eqref{fxx_k_px} =  \eqref{second_derivative} &  \thicksim \p_{x_i}\tb e^{-\nu \tb} \int_{\mathbb{R}^3} \dd u \mathbf{k}(v,u) \frac{1}{|n(\xb(x,v))\cdot u|}.
\end{align*}
Due to the fact that $\frac{1}{|n(\xb(x,v))\cdot u|}\notin L^1_u$, in the presence of the Boltzmann operator, the argument fails to control any second order derivatives.

\medskip

\noindent \textbf{Statements and Declarations.} The authors have no conflicts of interest to declare that are relevant to the content of this article.\\

\noindent \textbf{Data availability statements.} The authors declare that the data supporting the findings of this study are available within the paper.\\

\noindent \textbf{Acknowledgement.} This project is partly supported by NSF-DMS 1900923 and NSF CAREER
2047681. HC is partly supported by NSFC 8201200676 and GRF grant (2130715) from RGC of Hong Kong. CK is partly supported by Simons fellowship in Mathematics and the Brain Pool fellowship
funded by the Korean Ministry of Science and ICT. HC thanks Professor Zhennan Zhou for the kind hospitality
during his stay at Peking University.

\bibliographystyle{siam}
\bibliography{citation}

\end{document}